\documentclass[10pt]{article}
\usepackage[ansinew]{inputenc}
\usepackage{amsmath}
\usepackage{amsfonts}
\usepackage{amssymb}
\usepackage{amsthm}
\usepackage{newlfont}
\usepackage{amsbsy}
\usepackage{bm}
\usepackage{color}
\usepackage{epsfig}
\usepackage{dsfont}
\usepackage{latexsym}
\usepackage{array}
\usepackage{longtable}
\usepackage{enumerate}
\usepackage{multicol}
\usepackage{mathrsfs}
\usepackage{wasysym}
\usepackage{graphics}
\usepackage{graphicx}
\DeclareMathOperator\supp{supp}
\usepackage{hyperref}
\hypersetup{
    colorlinks,
    citecolor=blue,
    filecolor=black,
    linkcolor=blue,
    urlcolor=black
}
\usepackage{bbm}

\DeclareMathOperator{\sign}{sign}

\numberwithin{equation}{section}

\newtheorem{lemma}{Lemma}[section]
\newtheorem{theorem}{Theorem}[section]
\newtheorem{corol}{Corollary}[section]

\newtheorem{prop}{Proposition}[section]

\newtheorem{claim}{Claim}
\newtheorem*{remarks}{Remarks}

\date{}
\title{\textbf{Well-posedness for a two-dimensional dispersive model arising from capillary-gravity flows}}
\author{Oscar G. Ria\~no  \thanks{IMPA - Instituto de Matem\'atica Pura e Aplicada,
E-mail: {\tt ogrianoc@impa.br}}}

\begin{document}

\maketitle 

\begin{abstract}  
This paper is aimed to establish well-posedness in several settings for the Cauchy problem associated to a model arising in the study of capillary-gravity flows. More precisely, we determinate local well-posedness conclusions in classical Sobolev spaces and some spaces adapted to the energy of the equation. A key ingredient is a commutator estimate involving the Hilbert transform and fractional derivatives. We also study local well-posedness for the associated periodic initial value problem. Additionally, by determining well-posedness in anisotropic weighted Sobolev spaces as well as some unique continuation principles, we characterize the spatial behavior of solutions of this model. As a further consequence of our results, we derive new conclusions for the Shrira equation which appears in the context of waves in shear flows. 
\end{abstract}



\section{Introduction}

This work concerns the initial value problem (IVP) for the equation:
\begin{equation}\label{EQBO}
\begin{cases}
  \partial_t u +\mathcal{H}_xu-\mathcal{H}_x\partial_x^2u\pm \mathcal{H}_x\partial_y^2u+u\partial_{x} u=0,\hskip 15pt (x,y)\in \mathbb{R}^2 \, (\text{or } (x,y)\in \mathbb{T}^2 ),\,  t\in \mathbb{R}, \\
  u(x,0)= u_0,
  \end{cases}
\end{equation}
where $\mathcal{H}_x$ denotes the Hilbert transform in the $x$-direction defined via the Fourier transform as $\mathcal{F}(\mathcal{H}_x \phi )(\xi,\eta)=-i\sign(\xi)\widehat{\phi}(\xi,\eta)$ for $\phi \in \mathcal{S}(\mathbb{R}^2)$, and its periodic equivalent $\mathcal{F}(\mathcal{H}_x \phi )(m,n)=-i\sign(m)\widehat{\phi}(m,n)$ for $\phi \in C^{\infty}(\mathbb{T}^2)$. This model was derived in \cite{AkersMile} as an approximation to the equations for deep water gravity-capillary waves. Numerical results confirming existence of line solitary waves (solutions of the form $u(x,y,t)=\varphi(x-ct,y)$, $c>0$ and $\varphi$ real valued with suitable decay at infinity) as well as wave packet lump solitary waves were also presented in \cite{AkersMile}.

We are also interested in studying the IVP associated to the Shrira equation:
\begin{equation}\label{EQSH}
\begin{cases}
  \partial_t u -\mathcal{H}_x\partial_x^2u-\mathcal{H}_x\partial_y^2u+u\partial_{x} u=0,\hskip 15pt (x,y)\in \mathbb{R}^2 \, (\text{or } (x,y)\in \mathbb{T}^2 ),\,  t\in \mathbb{R}, \\
  u(x,0)= u_0.
  \end{cases}
\end{equation}
This equation was deduced as a simplified model to describe a two-dimensional weakly nonlinear long-wave perturbation on the background of a boundary-layer type plane-parallel shear flow (see \cite{PShrira}). Existence and asymptotic behavior of solitary-wave solutions were studied in \cite{Esfahani2018}. 

The models in \eqref{EQBO} and \eqref{EQSH} can be regarded, at least from a mathematical point of view, as two-dimensional versions of the Benjamin-Ono equation (see, \cite{AbBOnaFellSaut,FonPO,Kenig,MolinetPeriod,molinetPilodBO,Ponce1991,TaoBO} and the references therein):
\begin{equation}\label{BO}
    \partial_tu -\mathcal{H}_x\partial_x^2u +u\partial_x u=0.
\end{equation}

Alternatively, the equation in \eqref{EQBO} can be considered as a two-dimensional extension of the so called Burgers-Hilbert equation (see, \cite{Biello, BuHil}):
\begin{equation}\label{BHEQ}
    \partial_tu +\mathcal{H}_xu +u\partial_x u=0.
\end{equation}

This manuscript is intended to analyze well-posedness issues for the IVP \eqref{EQBO} and \eqref{EQSH}. 
Here we adopt Kato's notion of \emph{well-posedness}, which consists of  existence, uniqueness, persistence property (i.e.,  if the data $u_0\in X$ a function space, then the corresponding solution $u(\cdot)$ describes a continuous curve in $X$, $u \in C([0,T ];X), T > 0$), and continuous dependence of the map data-solution. In this regard, referring to the IVP \eqref{EQBO}, by implementing a parabolic regularization argument (see \cite{Iorio}) local well-posedness (LWP) in $H^s(\mathbb{R}^2)$ and $Y^s(\mathbb{R}^2)=\{f\in H^s: \|f\|_{Y^s}=\|f\|_{H^s}+\|\partial_x^{-1} f\|_{H^s}<\infty \}$, $s>2$ were inferred in \cite{Omarths}. It was also showed in \cite{Omarths} that the IVP \eqref{EQBO} is LWP in weighted Sobolev spaces $Y^s(\mathbb{R}^2)\cap L^2((|x|^{2r}+|y|^{2r} )\, dx dy)$, $0\leq r\leq 1$ and $s>2$. 
 
Concerning the IVP \eqref{EQSH},  by adapting the short-time linear Strichartz estimate approach employed in \cite{KenigKP,LinarFKP}, LWP in $H^s(\mathbb{R}^2)$ $s>3/2$ was deduced in \cite{Paisas1}. In \cite{paisas2}, inspired by the works in \cite{IonescuKenigPerioKP,perioZK}, LWP was established in $H^s(\mathbb{T}^2)$ $s>7/4$ assuming that the initial data satisfy $\int_{0}^{2\pi} u_0(x,y)\, dx=0$ for almost every $y$. Recently, in \cite{SCHIPPAShreq}, by employing short-time bilinear Strichartz estimates the conclusion on the periodic setting was improved to regularity $s>3/2$ without any assumption on the initial data. Now, with respect to weighted spaces, in \cite{JulioLi} LWP was deduced in $H^{s_1,s_2}(\mathbb{R}^2)\cap L^2((|x|^{2\theta}+|y|^{2r}) \, dx dy)$ $s_1\geq 2$ and $s_2\geq r$, where $0< \theta<1/2$ for arbitrary initial data, and $1/2< \theta<1$ assuming that $\widehat{u}(0,\eta)=0$ for almost every $\eta$.

It is worth pointing out that the equation in \eqref{EQBO} does not enjoy scale-invariance. In contrast, if $u$ solves the equation in \eqref{EQSH}, $u_{\lambda}(x,y,t)=\lambda u(\lambda x,\lambda y, \lambda^2 t)$ solves \eqref{EQSH} whenever $\lambda>0$, and so this equation is $L^2$-critical. On the other hand, real solutions of the IVP \eqref{EQBO} formally satisfy the following conserved quantities (time invariant):
\begin{align} 
    M(u)&=\int u^2(x,y,t)\, dxdy, \label{mass}\\
    E(u)&=\frac{1}{2} \int |D_x^{1/2}u(x,y,t)|^2+|D_x^{-1/2}u(x,y,t)|^2 \mp |D_x^{-1/2}\partial_y u(x,y,t)|^2-\frac{1}{3}u^3(x,y,t) \, dxdy, \label{Energy}
\end{align}
and real solutions of \eqref{EQSH} preserve the quantity $M(u)$ and 
\begin{equation}\label{Energy2}
        \widetilde{E}(u)=\frac{1}{2} \int |D_x^{1/2}u(x,y,t)|^2+ |D_x^{-1/2}\partial_y u(x,y,t)|^2-\frac{1}{3}u^3(x,y,t) \, dxdy,
\end{equation}
where $D_x^{\pm 1/2}$ is the fractional derivative operator in the $x$ variable defined by its Fourier transform as $\mathcal{F}(D_x^{\pm 1/2}u)(\xi,\eta)=|\xi|^{\pm 1/2}\widehat{u}(\xi,\eta)$. 

The aim of this paper is to obtain new well-posedness conclusions for both models \eqref{EQBO} and \eqref{EQSH} in the spaces $H^{s}(\mathbb{K}^2)$, $\mathbb{K}\in\{\mathbb{R},\mathbb{T}\}$ and some spaces adapted to \eqref{Energy} and \eqref{Energy2}. Furthermore, by establishing well-posedness in anisotropic spaces and some unique continuation principles, we will study the spatial behavior of solutions, determining that in general arbitrary polynomial type decay in the $x$-spatial variable is not preserved by the flow of these equations. 

Let us now state our results. We will mainly work on equation \eqref{EQBO}  without distinguishing between the signs of the term $\pm \mathcal{H}_x\partial_y^2 u$. Firstly, to justify the quantity \eqref{Energy}, we consider the spaces $X^s(\mathbb{R}^2)$ defined by 
\begin{equation}\label{EberSpa}
    \left\|f\right\|_{X^s}=\left\|J_x^s f\right\|_{L^2_{xy}}+\|D_x^{-1/2} f\|_{L^2_{xy}}+\|D_x^{-1/2}\partial_y f\|_{L^2_{xy}}.
\end{equation}
Our first conclusion establishes local well-posedness in the spaces $H^s(\mathbb{R}^2)$ and $X^s(\mathbb{R}^2)$.
\begin{theorem}\label{Improwellp}
Let $s>3/2$ and let $\mathfrak{X}^s(\mathbb{R}^2)$ be any (fixed) of the spaces $H^s(\mathbb{R}^2)$ and $X^s(\mathbb{R}^2)$. Then for any $u_0\in \mathfrak{X}^s(\mathbb{R}^2)$, there exist a time $T=T(\|u_0\|_{\mathfrak{X}^s})$ and a unique solution $u$ to the equation in \eqref{EQBO} that belongs to
\begin{equation}\label{classolu1}
    C([0,T];H^s(\mathbb{R}^2))\cap L^1([0,T];W^{1,\infty}(\mathbb{R}^2))
\end{equation}
if $u_0 \in H^s(\mathbb{R}^2)$, or it belongs to
\begin{equation}\label{classolu2}
    C([0,T];X^s(\mathbb{R}^2))\cap L^1([0,T];W^{1,\infty}_x(\mathbb{R}^2))
\end{equation}
if $u_0 \in X^s(\mathbb{R}^2)$. Moreover, the flow map $u_0 \mapsto u(t)$ is continuous from $\mathfrak{X}^s(\mathbb{R}^2)$ to $\mathfrak{X}^s(\mathbb{R}^2)$.
\end{theorem}
The Sobolev space $W^{1,\infty}(\mathbb{R}^2)$ is defined as usual with norm $\|f\|_{W^{1,\infty}}:=\|f\|_{L^{\infty}_{xy}}+\|\nabla f\|_{L^{\infty}_{xy}}$, and $W_x^{1,\infty}(\mathbb{R}^d)$ by $\|f\|_{W_x^{1,\infty}}:=\|f\|_{L^{\infty}_{xy}}+\|\partial_x f\|_{L^{\infty}_{xy}}$. The proof of Theorem \ref{Improwellp} is adapted from the ideas of Kenig \cite{KenigKP} and Linares, Pilod and Saut \cite{LinarFKP}. A novelty in the present work is the study of the operators $D_x^{-1/2}$ and $D_x^{-1/2}\partial_y$ which yields additional difficulties in contrast with the operator $\partial_x^{-1}\partial_y$ considered in the previous references. Among them, we required to deduce the following commutator  relation:
\begin{prop}\label{CalderonCom}
Let $1<p<\infty$, $0\leq \alpha, \beta \leq 1$, $\beta>0$ with $\alpha+\beta= 1$, then
\begin{equation}\label{Comwell}
    \|D_x^{\alpha}[\mathcal{H}_x,g]D_x^{\beta}f\|_{L^p(\mathbb{R})} \lesssim_{p,\alpha.\beta} \|\partial_x g\|_{L^{\infty}(\mathbb{R})}\|f\|_{L^p(\mathbb{R})}.
\end{equation}
\end{prop}
This estimate can be regarded as a nonlocal version of Calderon's first commutator estimate deduced in \cite[Lemma 3.1]{DawsonMCPON} and \cite[Proposition 3.8]{Dli} (see Proposition \ref{CalderonComGU} in the present document). Proposition \ref{CalderonCom} is proved in the appendix, and it is useful to perform energy estimates involving the operator $D_x^{-1/2}\partial_y$ and the nonlinearity in the equation in \eqref{EQBO}. 

 We remark that Theorem \ref{Improwellp} improves the conclusion in \cite{Omarths}, lowering the regularity in the Sobolev scale  to $s>3/2$ and obtaining well-posedness conclusion in spaces well-adapted to \eqref{Energy}. Furthermore, we believe that these results could certainly be used to study existence and stability of solitary wave solutions, where one employs the quantity $E(u)$ 
 (see for instance \cite{Esfahani2018}).
\\ \\
Next, we present our result in the periodic setting.
\begin{theorem}\label{LocalwellTorus}
Let $s>3/2$. Then for any $u_0\in H^s(\mathbb{T}^2)$, there exist $T=T(\|u_0\|_{H^s})$ and a unique solution $u$ of the IVP \eqref{EQBO} that belongs to
\begin{equation*}
    C([0,T];H^s(\mathbb{T}^2))\cap F^s(T)\cap B^s(T).
\end{equation*}
Moreover, for any $0<T'<T$, there exists a neighborhood $\mathcal{U}$ of $u_0$ in $H^s(\mathbb{T}^2)$ such that
the flow map data-solution, $$v\in \mathcal{U}\mapsto v\in C([0,T'];H^s(\mathbb{T}^2))$$ is continuous.
\end{theorem}
The function spaces $F^s(T)$ and $B^s(T)$ are defined in the Section \ref{sectionPeri} below. Theorem \ref{LocalwellTorus} is proved by means of the short-time Fourier restriction norm method developed by Ionescu, Kenig and Tataru \cite{IonescuKeniTata}, see also \cite{RibaVento,ZhangKP}. Mainly, this technique combines energy estimates with linear and nonlinear estimates in short-time Bourgain's spaces $F^s(T)$ and their dual $\mathcal{N}^s(T)$ (see Section \ref{sectionPeri}), where the former spaces enjoy the $X^{s,b}$ structure with localization in small time intervals whose length is of order $2^{-j}$, $j\in \mathbb{Z}^{+}\cup\{0\}$. We emphasize that up to our knowledge, Theorem \ref{LocalwellTorus} seems to be the first non-standard result dealing with the periodic equation \eqref{EQBO}.

Regarding the periodic quantity $E(u)$, we consider the Sobolev spaces
\begin{equation*}
    X^s(\mathbb{T}^2)=\{f\in H^s(\mathbb{T}^2): \, \widehat{f}(0,n)=0, \text{ for all } n\in \mathbb{Z}\}
\end{equation*}
equipped with the norm $\|f\|_{X^s(\mathbb{T}^2)}=\|f\|_{H^s(\mathbb{T}^2)}$. Then, since $X^s(\mathbb{T}^2)$ is a closed subspace of $H^s(\mathbb{T}^2)$, replacing the spaces $H^s(\mathbb{T}^2)$ by $X^s(\mathbb{T}^2)$ in Section \ref{sectionPeri} below, the same proof of Theorem \ref{LocalwellTorus} yields: 
\begin{corol}
Let $s>3/2$. Then the IVP \eqref{EQBO} is locally well-posed in $X^s(\mathbb{T}^2)$.
\end{corol}

\begin{remarks}\begin{itemize}
    \item[(i)] Our local theory is still not sufficient to reach the energy spaces $X^1(\mathbb{K}^2)$, $\mathbb{K}\in\{\mathbb{R}, \mathbb{T}\}$ determined by \eqref{Energy}.  
    \item[(ii)] For the one dimensional Benjamin-Ono equation \eqref{BO}, many authors, see \cite{Kenig,molinetPilodBO,TaoBO} for instance, have applied the gauge transformation to establish local and global results. Unfortunately, we do not know if there exists such gauge transformation for \eqref{EQBO}. Additionally, we do not know if there is a maximal norm estimate available for solutions of \eqref{EQBO}, which would allow us to argue as in \cite{KenigKo} to improve the results in Theorem \ref{Improwellp}.
    \item[(iii)] Concerning $\mathbb{R}^2$ solutions of \eqref{EQBO}, we do not have a standard approach to derive bilinear estimates in the spaces $F^s(T)$ and $\mathcal{N}^s(T)$. As a consequence, the short-time Fourier restriction norm method applied to this case leads the same regularity attained in Theorem \ref{Improwellp}. For this reason, we have proved Theorem \ref{Improwellp} by employing the short-time linear Strichartz approach instead, which also provides solutions in the class $ L^{1}([0,T];W^{1,\infty}(\mathbb{R}^2))$. The advantage of using this consequence lies in its application to methods based on energy estimates as the one we employ here to deduce well-posedness in weighted spaces.
\end{itemize}
\end{remarks}
Next, we study LWP issues in anisotropic weighted Sobolev spaces:
\begin{equation}\label{weightespace}
Z_{s,r_1,r_2}(\mathbb{R}^2)=H^{s}(\mathbb{R}^2)\cap L^{2}((| x|^{2r_1}+|y|^{2r_2}) \, dx dy), \hspace{0.5cm} s,r_1,r_2 \in \mathbb{R}
\end{equation}
and 
\begin{equation}\label{weightespacedot}
\dot{Z}_{s,r_1,r_2}(\mathbb{R}^2)=\left\{f\in H^{s}(\mathbb{R}^2)\cap L^{2}((| x|^{2r_1}+|y|^{2r_2}) \, dx dy):\, \widehat{f}(0,\eta)=0 \right\}, \hspace{0.5cm} s,r_1,r_2 \in \mathbb{R}.
\end{equation}
To motivate our results, we observe that for a function $f$ sufficiently regular with enough decay, 
 $x(\mathcal{H}_xu\pm\mathcal{H}_{x}\partial_y^2)f\in L^2(\mathbb{R}^2)$ requires the condition $\int f(x,y)e^{iy\eta}\, dx dy=0$ for almost every $\eta$. Thus, formally transferring this idea to the equation in \eqref{EQBO}, we do not expect that in general solutions of this model propagate weights of arbitrary order in the $x$-variable. Additionally, for arbitrary initial data, we contemplate to propagate weights of order $|x|^{\alpha}$ for some $0<\alpha<1$. In this regard, we have:
\begin{theorem}\label{localweigh}
\begin{itemize}
\item[(i)] If $r_1\in[0,1/2)$ and $r_2\geq 0$ with $s \geq \max\{(3/2)^{+}, r_2\} $, then the IVP associated to \eqref{EQBO} is locally well-posed in $Z_{s,r_1,r_2}(\mathbb{R}^2)$.
\item[(ii)] Let $r_2\geq 0$, $s\geq \max\{(3/2)^{+},r_2\}$. Then the IVP \eqref{EQBO} is locally well-posed in the space
\begin{equation*}
    ZH_{s,1/2,r_2}(\mathbb{R}^2)=\{f\in Z_{s,1/2,r_2}(\mathbb{R}^2): \|f\|_{Z_{s,1/2,r_2}}+\||x|^{1/2}\mathcal{H}_xf\|_{L^2_{xy}}<\infty\}.
\end{equation*}
\item[(iii)] If $r_1\in(1/2,3/2)$ and $r_2\geq 0$ with $s \geq \max\{(3/2)^{+}, r_2\} $, then the IVP associated to \eqref{EQBO} is locally well-posed in $\dot{Z}_{s,r_1,r_2}(\mathbb{R}^2)$.
\end{itemize}
\end{theorem}
In particular, Theorem \ref{localweigh} shows that the IVP \eqref{EQBO} admits weights of arbitrary order in the $y$-variable. The proof of these results follows the ideas of Fonseca, Linares and Ponce \cite{FLinaPonceWeBO,FLinaPioncedGBO,FonPO}. We emphasize that our conclusions involve further difficulties, since here we deal with anisotropic spaces in two spatial variables, and the $x$-spatial decay allowed by \eqref{EQBO} for arbitrary initial data does not even reach an integer number (cf.  \cite[Theorem 1]{FonPO} for the BO equation). 
Finally, we remark that Theorem \ref{localweigh} improves the range of weights determined in the work of \cite{Omarths}, and we do not require  the assumption $\partial_x^{-1}u \in H^s(\mathbb{R}^2)$.

Next, we state some unique continuation principles for solutions of the IVP \eqref{EQBO}.

\begin{theorem}\label{sharpdecay}
Let $r_1\in (1/4,1/2)$, $r_2\geq r_1$ and $s\geq \max\{ \frac{ 2r_1}{(4r_1-1)^{-}},r_2\}$. Let $u$ be a solution of the IVP \eqref{EQBO} such that $u\in C([0,T];Z_{s,r_1,r_2}(\mathbb{R}^2))\cap L^{1}([0,T];W_{1,x}^{\infty}(\mathbb{R}^2))$. If there exist two different times $t_1<t_2$ in $[0,T]$ for which
\begin{equation*}
u(\cdot,t_1)\in Z_{s,(1/2)^{+},r_2}(\mathbb{R}^2) \text{ and } u(\cdot,t_2)\in Z_{s,1/2,r_2}(\mathbb{R}^2),
\end{equation*} 
then $\widehat{u}(0,\eta,t)=0$ for all $t\in[t_1,T]$ and almost every $\eta$.
\end{theorem}

\begin{theorem}\label{sharpdecay1}
Let $r_2\geq r_1=(3/2)^{-}$ and $s> \max\{3,r_2\}$. Let $u$ be a solution of the IVP \eqref{EQBO} such that $u\in C([0,T];\dot{Z}_{s,r_1,r_2}(\mathbb{R}^2))$. If there exist two different times $t_1<t_2$ in $[0,T]$ for which
\begin{equation*}
u(\cdot,t_1)\in Z_{s,(3/2)^{+},r_2}(\mathbb{R}^2) \text{ and } u(\cdot,t_2)\in Z_{s,3/2,r_2}(\mathbb{R}^2),
\end{equation*} 
Then the following identity holds true
\begin{equation}\label{identw1}
    \begin{aligned}
    &2i \sin((1\mp \eta^2)(t_2-t_1))\partial _{\xi}\widehat{u}(0,\eta,t_1)
    = - \int_{t_1}^{t_2} \sin((1\mp\eta^2)(t_2-t'))\widehat{u^2}(0,\eta,t')\, dt',
    \end{aligned}
\end{equation}
for almost every $\eta \in \mathbb{R}$. In particular, if  $u(\cdot,t_1)\in Z_{s,2^{+},2^{+}}(\mathbb{R}^2)$ it follows
\begin{equation}\label{identw2}
    2\sin(t_2-t_1)\int x u(x,y,t_1)\, dx dy=(\cos(t_2-t_1)-1)\int u^2_0(x,y)\, dx dy. 
\end{equation}
\end{theorem}

\begin{remarks}
\begin{itemize}
\item[(i)] Since the weight $|x|$ does not satisfy the $A_2(\mathbb{R})$ condition (see \cite{JavierHarmo,SteinThe}) the assumption $\mathcal{H}_x u_0\in L^2(|x|\, dx dy)$ in the space $ZH_{s,1/2,r_2}(\mathbb{R}^2)$ is necessary for our arguments. Notice that for $u_0 \in Z_{s,1/2,r_2}(\mathbb{R}^2)$ the condition $\widehat{u_0}(0,\eta)=0$ does not make sense in general, for this reason, we have distinguished between part (ii) and (iii) of Theorem \ref{localweigh}. Besides, by inspecting our arguments in Lemma \ref{interpo} below and employing \cite[Theorem 4.3]{Yafaev}, the hypothesis $\mathcal{H}_x u_0\in L^2(|x|\, dx dy)$ can be replaced by the assumption that for a.e. $\eta$, the map $\xi \mapsto \widehat{u}_0(\xi,\eta)$ belongs to the $L^2(\mathbb{R})$-closure of the space of square integrable continuous odd functions. 

\item[(ii)] Theorem \ref{sharpdecay} establishes that for arbitrary initial data in $ Z_{s,r_1,r_2}(\mathbb{R}^2)$ with $r_2\geq r_1$ and $r_1\neq 1/2$, $(1/2)^{-}$ is the largest possible decay for solutions of the IVP \eqref{EQBO} on the $x$-spatial variable. Consequently, for this regimen of indexes $r_1,r_2$,  Theorem \ref{localweigh} (i) is sharp. However, it still remains an open problem to derive a similar conclusion for the cases $0\leq r_2<r_1$. Moreover, Theorem \ref{sharpdecay} shows that if $u_0 \in \mathbb{Z}_{s,r_1,r_2}(\mathbb{R}^2)$ with $r_2 \geq r_1=(1/2)^{+}$, $s\geq \max\{ \frac{ 2r_1}{(4r_1-1)^{-}},r_2\}$ and $\widehat{u_0}(0,\eta)\neq 0$ for almost every $\eta$, then the corresponding solution $u=u(x,t)$ of the IVP \eqref{EQBO} satisfies
$$|x|^{(1/2)^{-}}u\in L^{\infty}([0,T];L^2(\mathbb{R}^2)), \hspace{0.2cm} T>0.$$
Although, there does not exist a non-trivial solution $u$ corresponding to data $u_0$ with $\widehat{u_0}(0,\eta) \neq 0$ a.e. with 
$$|x|^{1/2}u\in L^{\infty}([0,T'];L^2(\mathbb{R}^2)), \hspace{0.1cm} \text{ for some } T'>0.$$
\item[(iii)] The condition $u(\cdot,t_1) \in Z_{s,2^{+},2^{+}}(\mathbb{R}^2)$ in Theorem \ref{sharpdecay1} can be relaxed assuming for instance that $u(\cdot,t_1)\in Z_{s,(3/2)^{+},r_2}(\mathbb{R}^2)$ and $xu (x,y,t_1) \in L^{1}(\mathbb{R}^2)$. In addition, \eqref{identw2} provides some unique continuation principles for solutions of the equation in \eqref{EQBO}. Indeed, if $(t_2-t_1)=k\pi$ for some positive odd integer number $k$, then it must be the case that $u\equiv 0$. Besides, if there exists three times $t_1<t_2<t_3$ such that $u(\cdot,t_1)\in Z_{s,2^{+},2^{+}}(\mathbb{R}^2)$,  $u(\cdot,t_j)\in Z_{s,3/2,r_2}(\mathbb{R}^2)$, $j=2,3$ and 
\begin{equation*}
    \sin(t_2-t_1)(1-\cos(t_3-t_1))\neq \sin(t_3-t_1)(1-\cos(t_2-t_1)),
\end{equation*}
then $u\equiv 0$. Accordingly, Theorem \ref{sharpdecay1} establishes that for any initial data $u_0 \in Z_{s,r_1,r_2}(\mathbb{R}^2)$, $r_2\geq r_1 > 2$, (or $u_0 \in Z_{s,(3/2)^{+},r_2}(\mathbb{R}^2)$ with $xu_0\in L^{1}(\mathbb{R}^2)$), $s> \max\{3,r_2\}$ the decay $(3/2)^{-}$ is the largest possible in the $x$-spatial decay. More precisely, if $u_0 \in Z_{s,r_1,r_2}(\mathbb{R}^2)$, $r_2\geq r_1 > 2$, $s> \max\{3,r_2\}$, then the corresponding solution  $u=u(x,t)$ of the IVP \eqref{EQBO} satisfies
$$|x|^{(3/2)^{-}} u\in L^{\infty}([0,T];L^2(\mathbb{R}^2)), \hspace{0.5cm} T>0$$
and there does not exist a non-trivial solution with initial data $u_0$ such that
$$|x|^{3/2} u\in L^{\infty}([0,T'];L^2(\mathbb{R}^2)), \hspace{0.4cm} \text{for some } \, T'>0.$$
\end{itemize}
\end{remarks}
All of the previous well-posedness conclusions were addressed by compactness method. As a matter of fact, we have that the local Cauchy problem for the equation in \eqref{EQBO} cannot be solved for initial data in any isotropic or anisotropic spaces by a direct contraction principle based on its integral formulation. 
\begin{prop}\label{illpos}
Let $s_1,s_2\in \mathbb{R}$ (resp. $s\in \mathbb{R}$). Then there does not exist a time $T>0$ such that the Cauchy problem \eqref{EQBO} admits a unique solution on the interval $[0,T]$ and such that the flow-map data-solution $u_0 \mapsto u(t)$ is $C^2$-differentiable from $H^{s_1,s_2}(\mathbb{R}^2)$ to $H^{s_1,s_2}(\mathbb{R}^2)$ (resp. from $X^s(\mathbb{R}^2)$ to $X^s(\mathbb{R}^2)$). 
\end{prop}

We remark that a similar conclusion was derived before for \eqref{EQSH} in \cite{AminPastor}. Following these arguments or the ideas in \cite[Theorem 1.4]{linaO} for instance, it is not difficult to deduce Proposition \ref{illpos}. For the sake of brevity, we omit its proof.

Finally, we present our conclusions on the Shrira equation: 

\begin{theorem}\label{SCeqThm}
Let $s>3/2$, then the IVP  \eqref{EQSH} is LWP in $H^s(\mathbb{K}^2)$, $\mathbb{K}\in \{\mathbb{R},\mathbb{T}\}$ and in the space $\widetilde{X}^s(\mathbb{R}^2)$ given by the norm
\begin{equation*}
    \|f\|_{\widetilde{X}^s}=\|J_x^sf\|_{L^2_{xy}}+\|D_x^{-1/2}\partial_y f\|_{L^2_{xy}}.
\end{equation*}
In addition, the results of Theorems \ref{localweigh} and \ref{sharpdecay} hold for the IVP \eqref{EQSH}. Moreover, the conclusion of Theorem \ref{sharpdecay1} is also valid considering
\begin{equation}\label{identw3}
    2i\sin(\eta^2(t_2-t_1))\partial_{\xi}\widehat{u}(0,\eta,t_1)=-\int_{t_1}^{t_2}\sin(\eta^2(t_2-t'))\widehat{u^2}(0,\eta,t')\, dt'
\end{equation}
instead of \eqref{identw1}. In particular, if $\partial_{\xi}\widehat{u}(0,\eta,t_1)=0$ for a.e. $\eta$, then $u \equiv 0$. 
\end{theorem}

As a result of Theorem \ref{SCeqThm}, we derive new well-posedness conclusions in the spaces $\widetilde{X}^s(\mathbb{R}^2)$ where the energy \eqref{Energy2} makes sense. Besides, in the periodic setting, we obtain the same well-posedness result stated for the two-dimensional case in the work of R. Schippa \cite[Theorem 1.2]{SCHIPPAShreq}, that is, we deduced that \eqref{EQSH} is LWP in $H^s(\mathbb{T}^2)$, $s>3/2$. We remark that our results are provided by rather different considerations than those given in \cite{SCHIPPAShreq}, where the author employed the setting of the periodic $U^p$-/$V^p$-spaces (\cite{HADAC2009}) combined with key short-time bilinear Strichartz estimates (see Section 3 of the aforementioned reference). Certainly, we believe that these considerations can be adapted to \eqref{EQBO}.

Regarding weighted spaces, our conclusions extend the results in \cite{JulioLi}, since here we deal with less regular solutions, and we improve the $x$-spatial decay allowed by \eqref{SCeqThm} to the interval $[0,3/2)$. Actually, by increasing the required regularity, it is not difficult to adapt our result to solutions in anisotropic spaces $H^{s_1,s_2}(\mathbb{R}^2)$. We remark that our proof of well-posedness in  $Z_{s,r_1,r_2}(\mathbb{R}^2)$ is applied directly to solutions in the space $H^s(\mathbb{R}^2)$, in contrast, in \cite{JulioLi} the author first derive well-posedness in weighted spaces for solutions with the additional property $\partial_x^{-1} u \in H^s(\mathbb{R}^2)$. 
\\ \\
We will begin by introducing some notation and preliminaries. Sections \ref{SectionReal} and \ref{sectionPeri} are devoted to prove Theorem \ref{Improwellp} and Theorem \ref{LocalwellTorus} respectively. Theorems \ref{localweigh}, \ref{sharpdecay} and \ref{sharpdecay1}  will be deduced in Section \ref{Secweights}. Section \ref{SecSh} is aimed to prove Theorem \ref{SCeqThm}.  We conclude the paper with an appendix where we show Proposition \ref{CalderonCom}.

\section{Notation and preliminaries}\label{Sectiprel}

Given two positive quantities $a$ and $b$,  $a\lesssim b$ means that there exists a positive constant $c>0$ such that $a\leq c b$. We write $a\sim b$ to symbolize that $a\lesssim b$ and $b\lesssim a$. The Fourier variables of $(x,y,t)$ are denoted $(\xi,\mu,\tau)$ and in the periodic case as $(m,n,\tau)$. 

$[A,B]$ denotes the commutator between the operators $A$ and $B$, that is
$$[A,B]=AB-BA.$$
Given $p\in [1,\infty]$ and $d\geq 1$ integer, we define the Lebesgue spaces $L^p(\mathbb{K}^d)$, $\mathbb{K}\in \{\mathbb{R},\mathbb{T}\}$ by its norm as $\|f\|_{L^p(\mathbb{K}^d)}=\left\|f\right\|_{L^p}=\left(\int_{\mathbb{K}^d} |f(x)|^p\,dx\right)^{1/p},$ with the usual modification when $p=\infty$. To emphasize the dependence on the variables when $d=2$, we will denote by $\|f\|_{L^p(\mathbb{K}^2)}=\|f\|_{L^p_{xy}(\mathbb{K}^2)}$.  We denote by $C_c^{\infty}(\mathbb{R}^d)$ the spaces of smooth functions of compact support and $\mathcal{S}(\mathbb{R}^d)$ the space of Schwarz functions. The  Fourier transform is defined by $$\widehat{f}(\xi)=\mathcal{F}f(\xi)=\int_{\mathbb{R}^d}  f(x) e^{-i x\cdot \xi}\, dx.$$
For a given number $s\in \mathbb{R}$, the operators  $J_x^s$, $J_y^s$ and $J^s$ are defined via the Fourier transform according to $\widehat{J_x^s \phi}(\xi,\eta)=\langle \xi \rangle^s \widehat{f}(\xi,\eta)$, $  \widehat{J_y^s \phi}(\xi,\eta)=\langle \eta \rangle^s \widehat{f}(\xi,\eta)$ and $ \widehat{J^s \phi}(\xi,\eta)=\langle |(\xi,\eta)| \rangle^s \widehat{f}(\xi,\eta)$, respectively, where $\langle x\rangle=(1+x^2)^{1/2}$. The Sobolev spaces $H^s(\mathbb{R}^2)$ consist of all tempered distributions such that $\left\|f\right\|_{H^s}=\left\|J^{s}f \right\|_{L^2}=\|\langle |(\xi,\eta)| \rangle^{s}\widehat{f}(\xi,\eta) \|_{L^2}<\infty$. 

Since we will deal with the periodic and real equation in different sections, we will employ the same notation for the norm of the Sobolev spaces $H^s(\mathbb{T}^2)$ which  consists of the periodic distributions such that  $\|f\|_{H^s}=\|J^s f\|_{L^2(\mathbb{T})}=\|\langle |(m,n)| \rangle^s\widehat{f}(m,n)\|_{L^2(\mathbb{Z}^2)}<\infty$. Recalling the spaces \eqref{EberSpa}, we will denote by $H^{\infty}(\mathbb{K}^2)=\bigcap_{s\geq 0} H^s(\mathbb{K}^2)$ for $\mathbb{K}=\mathbb{R}$ or $\mathbb{K}=\mathbb{T}$, and $X^{\infty}(\mathbb{R}^2)=\bigcap_{s\geq 0} X^s(\mathbb{R}^2)$.

Now, if $A$ denotes a functional space (for instance those introduced above), we define the spaces $L^p_TA$ and $L^{p}_t A$ according to the norms
\begin{equation*}
\|f\|_{L^{p}_T A}=\|\|f(\cdot,t)\|_{A}\|_{L^{p}([0,T])}\,  \text{ and } \, \|f\|_{L^{p}_t A}=\|\|f(\cdot,t)\|_{A}\|_{L^{p}(\mathbb{R})},
\end{equation*}
respectively, for all $1\leq p\leq \infty$.

We define the unitary group of solutions of the linear problem determined by \eqref{EQBO} by
\begin{equation}\label{SemG}
    S(t)u_0(x,y)=\int e^{it\omega(\xi,\eta)+ix\xi+iy\eta} \widehat{u_0}(\xi,\eta)\, d\xi d\eta
\end{equation}
where 
\begin{equation}\label{lieareqsym}
    \omega(\xi,\eta)=\sign(\xi)+\sign(\xi)\xi^2\mp\sign(\xi)\eta^2.
\end{equation}
The resonant function is given by
\begin{equation}\label{resonF}
\begin{aligned}
\Omega(\xi_1,\eta_1,\xi_2,\eta_2):=\omega(\xi_1+\xi_2,\eta_1+\eta_2)-\omega(\xi_1,\eta_1)-\omega(\xi_2,\eta_2).
\end{aligned}
\end{equation}
The variable $N$ is assumed to be dyadic, i.e., $N\in\left\{2^{l}\, :\, l\in \mathbb{Z}\right\}$. We will mostly use the dyadic numbers $N\geq 1$, then we set  $\mathbb{D}=\left\{2^{l}\, :\,l\in \mathbb{Z}^{+}\cup\left\{0\right\}\right\}$. Let $\psi_1\in C^{\infty}_c(\mathbb{R})$ even function, $0\leq \psi_1 \leq 1$ with $\supp{\psi_1}\subset [-2,2]$ and $\psi_1=1$ in $[-1,1]$. For each $N\in \mathbb{D}\setminus\{1\}$, we let $\psi_N(\xi)=\psi_1(\xi/N)-\psi_1(2\xi/N)$ and $\psi_{\leq N}(\xi)=\psi_1(\xi/N)$.  We define the projector operators in $L^2(\mathbb{R}^2)$ by the relations
\begin{equation}\label{proje1}
\begin{aligned}
\mathcal{F}(P_N^x(u))(\xi,\eta)&=\psi_{N}(\xi)\mathcal{F}(u)(\xi,\eta), \\
\mathcal{F}(P_{\leq N}^x(u))(\xi,\eta)&=\psi_{\leq N}(\xi)\mathcal{F}(u)(\xi,\eta).
\end{aligned}
\end{equation}
With a slightly abuse of notation, we will employ the same notation for the operators $P_N^x$ and $P_{\leq N}^x$ defined in $L^2(\mathbb{R})$. We also require the following projections in our estimates 
\begin{equation}\label{proje2}
    \begin{aligned}
\mathcal{F}(P_N(u))(\xi,\eta)&=\psi_{N}(|(\xi,\eta)|)\mathcal{F}(u)(\xi,\eta), \\
\mathcal{F}(P_{\leq N}(u))(\xi,\eta)&=\psi_{\leq N}(|(\xi,\eta)|)\mathcal{F}(u)(\xi,\eta).
    \end{aligned}
\end{equation}

To obtain estimates for the nonlinear term the following Leibniz rules for fractional derivatives will be implemented in our arguments.

\begin{lemma}\label{conmKP}
If $s>0$ and $1<p<\infty$, then
\begin{equation*}
    \left\|[J^s,f]g\right\|_{L^p(\mathbb{R}^d)} \lesssim  \left\|\nabla f\right\|_{L^{\infty}(\mathbb{R}^d)} \left\|J^{s-1}g\right\|_{L^p(\mathbb{R}^d)}+ \left\|J^s f\right\|_{L^p(\mathbb{R}^d)} \left\|g\right\|_{L^{\infty}(\mathbb{R}^d)}.
\end{equation*}
\end{lemma}
Lemma \ref{conmKP} was proved by Kato and Ponce in \cite{KP}. We also need the following lemma whose proof can be consulted in \cite{FRAC}.

\begin{lemma}\label{fraLR}
Given $d\in \mathbb{Z}^{+}$ and $s>0$, it holds that
\begin{align}
    \left\|D^s(fg)\right\|_{L^2(\mathbb{R}^d)} &\lesssim  \left\|D^s f\right\|_{L^{p_1}(\mathbb{R}^d)} \left\|g\right\|_{L^{q_1}(\mathbb{R}^d)}+ \left\|f\right\|_{L^{p_2}(\mathbb{R}^d)} \left\|D^sg\right\|_{L^{q_2}(\mathbb{R}^d)}, \label{libRule} \\
     \left\|J^s(fg)\right\|_{L^2(\mathbb{R}^d)} & \lesssim  \left\|J^s f\right\|_{L^{p_1}(\mathbb{R}^d)} \left\|g\right\|_{L^{q_1}(\mathbb{R}^d)}+ \left\|f\right\|_{L^{p_2}(\mathbb{R}^d)} \left\|J^s g\right\|_{L^{q_2}(\mathbb{R}^d)}, \label{eqfraLR}
\end{align}
with $\frac{1}{p_j}+\frac{1}{q_j}=\frac{1}{2}$, $1<p_1,p_2,q_1,q_2  \leq \infty$.
\end{lemma}

\begin{lemma}\label{commtwovar} Let $\sigma, \beta \in (0,1)$, then
\begin{equation*}
    \begin{aligned}
    \|D^{\sigma}_xD^{\beta}_y(fg)\|_{L^2_{xy}(\mathbb{R}^2)} \lesssim & \|f\|_{L^{p_1}_{xy}(\mathbb{R}^2)}\|D^{\sigma}_xD^{\beta}_yg\|_{L^{q_1}_{xy}(\mathbb{R}^2)}+\|D^{\sigma}_xD^{\beta}_yf\|_{L^{p_2}_{xy}(\mathbb{R}^2)}\|g\|_{L^{q_2}_{xy}(\mathbb{R}^2)} \\
    &+\|D^{\beta}_yf\|_{L_{xy}^{p_3}(\mathbb{R}^2)}\|D^{\sigma}_xg\|_{L^{q_3}_{xy}(\mathbb{R}^2)}+\|D_x^{\sigma}f\|_{L^{p_4}_{xy}(\mathbb{R}^2)}\|D^{\beta}_yg\|_{L^{q_4}_{xy}(\mathbb{R}^2)},
    \end{aligned}
\end{equation*}
where $\frac{1}{p_j}+\frac{1}{q_j}=\frac{1}{2}$, $1<p_j,q_j  \leq \infty$, $j=1,2,3,4$.
\end{lemma}

Lemma \ref{commtwovar} was deduced by Muscalu, Pipher, Tao and Thiele in \cite{muscalu}. The following commutator estimate will be useful in our considerations.

\begin{prop}\label{CalderonComGU}
Let $1<p<\infty$ and $l,m \in \mathbb{Z}^{+}\cup \{0\}$, $l+m\geq 1$ then
\begin{equation}\label{Comwellprel1}
    \|\partial_x^l[\mathcal{H}_x,g]\partial_x^{m}f\|_{L^p(\mathbb{R})} \lesssim_{p,l,m} \|\partial_x^{l+m} g\|_{L^{\infty}(\mathbb{R})}\|f\|_{L^p(\mathbb{R})}.
\end{equation}
\end{prop}
The estimate \eqref{Comwellprel1} was established in \cite[Lemma 3.1]{DawsonMCPON} and it was extended to the BMO spaces in \cite[Proposition 3.8]{Dli}.

To deduce the LWP result in Theorem \ref{Improwellp} on the space $X^s(\mathbb{R}^2)$, we require the following set of inequalities, which were deduced in the proof of \cite[Lemma 2.1]{KenigKP} (see equations (2.5), (2.6) and (2.7) in this reference). See also \cite[Lemma 4.6]{LinarFKP}.  
\begin{lemma}\label{lemmainterine}
\begin{itemize}
    \item[(i)] Let $0<\delta<1/2$, then
    \begin{equation}\label{Interp1}
    \begin{aligned}
    \|D_x^{1/2+\delta}u\|_{L^{\infty}_{xy}} \lesssim \|u\|_{L^{\infty}_{xy}}+\|\partial_xu\|_{L^{\infty}_{xy}},
    \end{aligned}
\end{equation}
\item[(ii)]  If $\delta_0$ is a positive constant chosen small enough, then the following holds true. There exist
\begin{equation*}
    \left\{ \begin{aligned}
    2<p_1,q_1<\infty \\
    1<r_1,s_1<\infty
    \end{aligned}
    \right.\hspace{0.5cm} \text{ with } \hspace{0.5cm} \frac{1}{p_1}+\frac{1}{q_1}=\frac{1}{2}, \hspace{0.2cm} \frac{1}{r_1}+\frac{1}{s_1}=1,
\end{equation*}
$0<\theta<1$ and $0<\delta_1=\delta_1(\delta_0,\theta)\ll 1$ such that
\begin{equation}\label{Interp2.0}
    \|D_x^{1/2+\delta} u\|_{L_T^{s_1}L_{xy}^{q_1}}\lesssim \|u\|_{L^1_T L_{xy}^{\infty}}^{\theta}\|J_x^{1/2+\delta_0}u\|_{L_T^{\infty}L^2_{xy}}^{1-\theta},
\end{equation}
\begin{equation}\label{Interp2}
    \|\partial_xD_x^{1/2+\delta} u\|_{L_T^{s_1}L_{xy}^{q_1}}\lesssim \|\partial_xu\|_{L^1_T L_{xy}^{\infty}}^{\theta}\|J_x^{3/2+\delta_0}u\|_{L_T^{\infty}L^2_{xy}}^{1-\theta},
\end{equation}
and
\begin{equation}\label{Interp3}
    \|D_y^{\delta}u\|_{L_T^{r_1}L_{xy}^{p_1}} \lesssim \big(\|u\|_{L^1_TL^{\infty}_{xy}}\big)^{1-\theta}\big(\|D_y^{1/2}u\|_{L^{\infty}_TL^2_{xy}}+\|u\|_{L^{\infty}_TL^2_{xy}}\big)^{\theta},
\end{equation}
for all $0<\delta<\delta_1$.
\end{itemize}
\end{lemma}


\section{ Well-posedness for real solutions  }\label{SectionReal}

This section is devoted to establish LWP for \eqref{EQBO} in the spaces $H^s(\mathbb{R}^2)$ and $X^s(\mathbb{R}^2)$. Since this conclusion in the former space can be deduced by the same reasoning applied to $X^s(\mathbb{R}^2)$, or by following the ideas in \cite[Theorem 1.3]{linaO}, we will restrict our considerations to prove Theorem \ref{Improwellp} in  $X^s(\mathbb{R}^2)$.

However, given that the LWP result in $H^s(\mathbb{R}^2)$, $s>3/2$ will be employed to deduce Theorem \ref{localweigh}, we will present some remarks on this conclusion in the Subsection \ref{subHs} below.

\subsection{Linear Estimates}

This section summarized some space-time estimates for the unitary group $\{S(t)\}$ defined by \eqref{SemG}. 

\begin{lemma}\label{STEST}
The following estimate holds
\begin{equation*}
    \left\|S(t)f\right\|_{L^{q}_tL^{p}_{xy}} \lesssim \left\|f\right\|_{L^2},
\end{equation*}
whenever $2\leq p,q \leq \infty$, $q>2$ and $\frac{1}{p}+\frac{1}{q}=\frac{1}{2}$.
\end{lemma}

\begin{proof}
The proof follows a similar reasoning to \cite[Proposition 4.8]{LinarFKP} and the case $\alpha=1$ in \cite{Paisas1}. 
\end{proof}
Notice that the endpoint Strichartz estimate corresponding to $(q,p)=(2,\infty)$ is not stated in the preceding lemma, as a consequence we need to lose a little bit of regularity to control this norm.

\begin{corol}\label{liestCoro}
For each $T>0$ and $\delta>0$, there exists $\kappa_{\delta}\in(0,1/2)$ such that
\begin{equation*}
    \left\|S(t)f\right\|_{L^2_T L^{\infty}_{xy}}\lesssim T^{\kappa_{\delta}}\left\|J^{\delta}f\right\|_{L^2}
\end{equation*}
where the implicit constant depends on $\delta$.
\end{corol}

\begin{proof}
Taking $p$ sufficiently large such that $\delta> 2/p$, Sobolev embedding and \eqref{STEST} yield
\begin{equation*}
    \left\|S(t)f\right\|_{L^2_T L^{\infty}_{xy}} \lesssim_{\delta} T^{\frac{q-2}{2q}} \left\|S(t)J^{\delta}f\right\|_{L^q_T L^p_{xy}} \lesssim_{\delta} T^{\frac{q-2}{2q}} \left\|J^{\delta}f\right\|_{L^2_{xy}}.
\end{equation*}
This completes the proof. 
\end{proof}

Also, we require the following refined Strichartz estimate, which has been proved in different contexts (see \cite{Paisas1,KenigKP,LinarFKP}). 

\begin{lemma} \label{refinStri}
Let $0<\delta \leq 1$ and $T>0$. Then there exist $\kappa_{\delta}\in (\frac{1}{2},1)$ and $\delta>0$ such that
\begin{equation}\label{liest3.1}
\begin{aligned}
    \left\| v\right\|_{L^1_TL^{\infty}_{xy}} \lesssim_{\delta} T^{\kappa_{\delta}} \big(\sup_{[0,T]}&\|J_x^{1/2+ 2\delta}v(t)\|_{L^2_{xy}}+\sup_{[0,T]}\|J_x^{1/2+ \delta}D_y^{\delta}v(t)\|_{L^2_{xy}} \\
    &+\int_0^T ( \|J^{-1/2+2\delta}_xF(\cdot,t') \|_{L^2_{xy}}+\|J^{-1/2+\delta}_x D_y^{\delta}F(\cdot,t') \|_{L^2_{xy}} \, )dt'\big),
\end{aligned}
\end{equation}
whenever $v$ solves 
\begin{equation}\label{liest4}
    \partial_t v +\mathcal{H}_xv-\mathcal{H}_x\partial_x^2v\pm \mathcal{H}_x\partial_y^2v=F.
\end{equation}
\end{lemma}

\begin{proof}
In view of Corollary \ref{liestCoro}, \eqref{liest3.1} is deduced following the same reasoning in  \cite[Lemma 4.11]{LinarFKP}. See also \cite[Lemma 1.7]{KenigKP}. 
\end{proof}


\subsection{Energy Estimates}

\begin{lemma}\label{apriEST}
Let $s>0$. Consider $T>0$ and $u\in C([0,T];X^{\infty}(\mathbb{R}^d))$ be a solution of the IVP \eqref{EQBO}, then there exists a constant $c_0>0$ such that
\begin{equation}\label{twoEnergyES2}
    \left\|u\right\|_{L^{\infty}_T X^s}^2 \leq \left\|u_0\right\|_{X^s}^2+c_0 (\left\| u\right\|_{L^1_T L^{\infty}_{xy}}+\left\|\partial_x u\right\|_{L^1_T L^{\infty}_{xy}})\left\|u\right\|_{L^{\infty}_T X^s}^2.
\end{equation}
\end{lemma}

\begin{proof}
The estimates of the norm $\|J^s_x(\cdot)\|_{L^2_{xy}}$ in the space $X^{s}(\mathbb{R}^2)$ is deduced by applying standard energy estimates and Lemma \ref{conmKP}. For a more detailed discussion, we refer to \cite[Lemma 1.3]{KenigKP} .

To deal with the component $\|D_x^{-1/2}(\cdot)\|_{L^2_{xy}}$ in the $X^s(\mathbb{R}^2)$-norm, we apply $D_x^{-1/2}$ to the equation in \eqref{EQBO}, we multiply then by $D_x^{-1/2}u$ and integrate in space to deduce
\begin{equation*}
\begin{aligned}
 \frac{1}{2}\frac{d}{dt}\left\|D_x^{-1}u(t)\right\|_{L^2_{xy}}^2=-\frac{1}{2}\int D_x^{-1/2}\partial_x(u^2)D_x^{-1/2}u \, dx dy=-\frac{1}{2} \int D^{-1}_x \partial_x(u^2) u \, dx dy,
\end{aligned}
\end{equation*}
where we have used that the operator $\mathcal{H}_x-\mathcal{H}_x\partial_x^2\pm \mathcal{H}_x\partial_y^2$ is skew-symmetric and $D_x^{-1/2}$ is self-adjoint. Hence, by writing  $\partial_x=-\mathcal{H}_x D_x$ and using that $\mathcal{H}_x$ defines a bounded operator in $L^2_{xy}(\mathbb{R}^2)$, we get
\begin{equation*}
    \frac{d}{dt}\|D_x^{-1}u(t)\|_{L^2_{xy}}^2 \lesssim \|u\|_{L^{\infty}_{xy}}\|u\|_{X^s}^2.
\end{equation*}
To control the norm $\|D_x^{-1/2}\partial_y(\cdot)\|_{L^2_{xy}}$, we apply $D_x\partial_y^{-1/2}$ to the equation in \eqref{EQBO}, multiplying the resulting expression by $D_x^{-1/2}\partial_y u$ and integrating in space it is seen that
\begin{equation*}
\begin{aligned}
 \frac{1}{2}\frac{d}{dt}\left\|D_x^{-1}\partial_yu(t)\right\|_{L^2_{xy}}^2=-\frac{1}{2}\int D_x^{-1/2}\partial_y\partial_x(u^2)D_x^{-1/2}\partial_yu \, dx dy.
\end{aligned}
\end{equation*}
Once again, decomposing $\partial_x=-\mathcal{H}_x D_x$ and using that $\mathcal{H}_x$ is skew-symmetric, we get 
\begin{equation}\label{twoeqenerg1}
\begin{aligned}
  \int D_x^{-1/2}\partial_y\partial_x(u^2)D_x^{-1/2}\partial_y u \, dx dy&=-\int \mathcal{H}_x \partial_y(u^2)\partial_yu \, dx dy \\
  &=-\int ([\mathcal{H}_x,u]\partial_y u) \partial_y u \, dxdy \\
  &=\int (D_x^{1/2}[\mathcal{H}_x,u]D_x^{1/2}(D^{-1/2}_x\partial_yu)) D_x^{-1/2}\partial_yu \, dxdy.
\end{aligned}
\end{equation}
Then the Cauchy-Schwarz inequality and Proposition \ref{CalderonCom} yield
\begin{equation}\label{twoeqenerg2}
\begin{aligned}
&\left|\int (D_x^{1/2}[\mathcal{H}_x,u]D_x^{1/2}(D^{-1/2}_x\partial_yu)) D_x^{-1/2}\partial_yu \, dxdy\right|\\
&\hspace{3cm}\lesssim \|\|D_x^{1/2}[\mathcal{H}_x,u]D_x^{1/2}(D^{-1/2}_x\partial_yu)\|_{L^2_x}\|_{L^2_y}\|D_x^{-1/2}\partial_yu\|_{L^2_{xy}} \\
&\hspace{3cm}\lesssim \|\partial_xu\|_{L^{\infty}_{xy}}\|D_x^{-1/2}\partial_yu\|_{L^2_{xy}}^2,
\end{aligned}
\end{equation}
 and so we arrive at
\begin{equation*}
\begin{aligned}
 \frac{d}{dt}\|D_x^{-1}\partial_yu(t)\|_{L^2_{xy}}^2 \lesssim \|\partial_x u\|_{L^{\infty}_{xy}}\|D_x^{-1/2}\partial_y u\|_{L^2_{xy}}^2.
\end{aligned}
\end{equation*}
Integrating in time the previous estimates yield the desired conclusion.
\end{proof}

Next  we derive \emph{a priori} estimates for the norms $\left\|u \right\|_{L^1_T L_{xy}^{\infty}}$ and $\left\|\partial_x u \right\|_{L^1_T L_{xy}^{\infty}}$ in $X^s(\mathbb{R}^2)$, whenever $s>3/2$.

\begin{lemma}\label{apriEstS} 
Let $s>3/2$ fixed.  Consider $u\in C([0,T];X^{\infty}(\mathbb{R}^2))$ solution of the IVP \eqref{EQBO}. Then, there exist $\kappa_{\delta}\in(\frac{1}{2},1)$ and $c_s>0$ such that
\begin{equation}\label{eqapriEst2}
    (\left\|u\right\|_{{L^1_TL^{\infty}_{xy}}}+ \left\|\partial_x u\right\|_{L^1_TL^{\infty}_{xy}})\leq c_sT^{\kappa_{\delta}}(1+\left\|u\right\|_{{L^1_TL^{\infty}_{xy}}}+ \left\|\partial_x u\right\|_{L^1_TL^{\infty}_{xy}})\left\|u\right\|_{L^{\infty}_T X^s}.
\end{equation}
\end{lemma}

\begin{proof}
We will follow the arguments in \cite{KenigKP} and \cite{LinarFKP}. By applying Lemma \ref{refinStri} with $F=-\partial_x(u\partial_x u)$ we find
\begin{equation}\label{twoeqenerg6}  
\begin{aligned}
\|\partial_x u\|_{L^1_TL^{\infty}_{xy}}\lesssim_{\delta} T^{\kappa_{\delta}}\big(\sup_{ [0,T]}&\|J_x^{3/2+2\delta}u(t)\|_{L^2_{xy}}+\sup_{ [0,T]}\|J_x^{3/2+\delta}D_y^{\delta}u(t)\|_{L^2_{xy}} \\
&+\int_0^T ( \|J^{1/2+2\delta}_x(u\partial_x u)(t') \|_{L^2_{xy}}+\|J^{1/2+\delta}_x D_y^{\delta}(u\partial_x u)(t') \|_{L^2_{xy}} )\, dt'\big).
\end{aligned}
\end{equation}
 Taking $\delta>0$ small such that $\frac{3}{2}\big(\frac{1+\delta}{1-\delta}\big)<s$, Young's inequality yields
\begin{equation}\label{twoeqenerg7} 
    (1+|\xi|)^{3/2+\delta}|\eta|^{\delta}\lesssim \big((1+|\xi|)^{3/2+\delta} |\xi|^{\delta/2}\big)^{1/(1-\delta)}+|\eta||\xi|^{-1/2} \lesssim (1+|\xi|)^{s}+|\eta||\xi|^{-1/2}.
\end{equation}
Hence the previous display and Plancherel's identity show
\begin{equation}\label{twoeqenerg7.0}
    \sup_{[0,T]}\big(\|J_x^{3/2+2\delta}u(t)\|_{L^2_{xy}}+\|J_x^{3/2+\delta}D_y^{\delta}u(t)\|_{L^2_{xy}}\big) \lesssim \sup_{[0,T]}\big(\|J_x^s u(t)\|_{L^2_{xy}}+\|D_x^{-1/2}\partial_y u(t)\|_{L^2_{xy}} \big)\lesssim \|u\|_{L^{\infty}_T X^s} .
\end{equation}
This completes the estimate for the first two terms on the right-hand side (r.h.s) of \eqref{twoeqenerg6}. Next, we deal with the third factor on the r.h.s of \eqref{twoeqenerg6}. An application of \eqref{eqfraLR} allows us to deduce
\begin{equation}\label{twoeqenerg7.2}
\begin{aligned}
\|J^{1/2+2\delta}_x(u\partial_x u) \|_{L^1_TL^2_{xy}} &=\|\,\|\,\|J^{1/2+2\delta}_x(u\partial_x u)(t)\|_{L^2_x}\|_{L^2_y}\|_{L^1_T}\\
&\lesssim\|\, ( \|\,\|u(t)\|_{L^{\infty}_x}\|J_x^{1/2+2\delta}\partial_x u(t)\|_{L^2_x}+\|\partial_x u(t)\|_{L^{\infty}_x}\|J_x^{1/2+2\delta}u(t)\|_{L^2_x}\|_{L^2_y})\|_{L^1_T} \\
&\lesssim (\|u\|_{L^1_TL^{\infty}_{xy}}+\|\partial_x u\|_{L^1_TL^{\infty}_{xy}})\|J_x^{s}u\|_{L^{\infty}_TL^2_{xy}},
\end{aligned}
\end{equation}
which holds for $0<\delta<\min\{1/2,s/2-3/4\}$. Since $\|J_x^{s}u\|_{L^{\infty}_TL^2_{xy}} \leq \|u\|_{L_T^{\infty}X^s}$, the previous inequality completes the study of third term in \eqref{twoeqenerg6}. Next, we decompose the remaining factor in \eqref{twoeqenerg6} as follows
\begin{equation*}
\begin{aligned}
\int_0^T\|J^{1/2+\delta}_x D_y^{\delta}(u\partial_x u)(t') \|_{L^2_{xy}} \, dt'&\lesssim \|D_y^{\delta}(u\partial_x u) \|_{L^1_TL^2_{xy}}+\|D^{1/2+\delta}_x D_y^{\delta}(u\partial_x u) \|_{L^1_TL^2_{xy}}=:\mathcal{I}+\mathcal{II}.   
\end{aligned}
\end{equation*}
To deal with $\mathcal{I}$, we employ the point-wise inequality
\begin{equation} \label{twoeqenerg7.1}
\begin{aligned}
|\xi|^{l}|\eta|^{\delta}&=|\xi|^{l+\delta/2}(|\xi|^{-1/2}|\eta|)^{\delta}\lesssim (1+|\xi|)^{\frac{2l+\delta}{2(1-\delta)}}+|\xi|^{-1/2}|\eta|, \end{aligned}
\end{equation}
valid for $l=0,1$ and $0<\delta<1$ small satisfying $\frac{2l+\delta}{2(1-\delta)}<s$. Hence the fractional Leibniz's rule \eqref{libRule}, Plancherel's identity and \eqref{twoeqenerg7.1} show
\begin{equation*}
\begin{aligned}
\mathcal{I} =\|\,\|\, \|D_y^{\delta}&(u\partial_x u)(t)\|_{L^2_y(\mathbb{R})}\|_{L^{2}_x(\mathbb{R})}\|_{L^1_T}\\
& \lesssim \| (\|u(t)\|_{L^{\infty}_{xy}}\|D_y^{\delta}\partial_xu(t)\|_{L^{2}_{xy}}+\|\partial_xu(t)\|_{L^{\infty}_{xy}}\|D_y^{\delta}u(t)\|_{L^{2}_{xy}})\|_{L^1_T} \\
& \lesssim  (\|u\|_{L^1_TL^{\infty}_{xy}}+\|\partial_xu\|_{L^1_TL^{\infty}_{xy}}) (\|J_x^su\|_{L_T^{\infty}L^{2}_{xy}}+\|D_x^{-1/2}\partial_y u\|_{L_T^{\infty}L^2_{xy}}).
\end{aligned}
\end{equation*}
This completes the analyze of $\mathcal{I}$. On the other hand, employing Lemma \ref{commtwovar}, we further decompose $\mathcal{II}$
\begin{equation}\label{twoeqenerg9}
\begin{aligned}
    \mathcal{II}\lesssim &\|\, \|u(t)\|_{L^{\infty}_{xy}}\|D_x^{3/2+\delta}D_y^{\delta} u(t)\|_{L^2_{xy}}\|_{L^1_T}+\|\,\|\partial_x u(t)\|_{L^{\infty}_{xy}}\|D_x^{1/2+\delta}D_y^{\delta}u(t)\|_{L^2_{xy}}\|_{L^1_T}\\
    &+\|\,\| D_x^{1/2+\delta}u(t)\|_{L^{\infty}_{xy}}\|\partial_xD_y^{\delta}u(t)\|_{L^2_{xy}}\|_{L^1_T}+\|\,\|\partial_x D_x^{1/2+\delta}u(t)\|_{L^{q_1}_{xy}}\|D_y^{\delta}u(t)\|_{L^{p_1}_{xy}}\|_{L^1_T}  \\
    &=:\mathcal{II}_1+\mathcal{II}_2+\mathcal{II}_3+\mathcal{II}_4,
\end{aligned}
\end{equation}
where $\frac{1}{p_1}+\frac{1}{q_1}=\frac{1}{2}$. Since the norms $\|D_x^{3/2+\delta}D_y^{\delta}u(t)\|_{L^2_{xy}},\|D_x^{1/2+\delta}D_y^{\delta}u(t)\|_{L^2_{xy}}\leq \|J_x^{3/2+\delta}D_y^{\delta}u(t)\|_{L^2_{xy}}$, we use \eqref{twoeqenerg7} with $\frac{3}{2}\big(\frac{1+\delta}{1-\delta}\big)<s$ and Plancherel's identity to infer $\mathcal{II}_1+\mathcal{II}_2\lesssim (\|u\|_{L^1_TL^{\infty}_{xy}}+\|\partial_x u\|_{L^1_TL^{\infty}_{xy}}) \|u\|_{L_T^{\infty}X^s}$.

To deal with $\mathcal{II}_3$, we let $0<\delta<1/2$ small satisfying $\frac{2+\delta}{2(1-\delta)}<s$, then we employ \eqref{Interp1} to control the norm $\|D_x^{1/2+\delta}u(t)\|_{L^{\infty}_{xy}}$. The estimate for $\|\partial_xD_y^{\delta}u(t)\|_{L^2_{xy}}$ is a consequence of Plancherel's identity and \eqref{twoeqenerg7.1} with $l=1$. This yields the desired bound for $\mathcal{II}_3$.

Next,  by employing \eqref{Interp2}, \eqref{Interp3} in Lemma \ref{lemmainterine}, it is seen 
\begin{equation*}
\begin{aligned}
    \mathcal{II}_4 &\lesssim \|\partial_xD_x^{1/2+\delta}u\|_{L_T^{s_1}L_{xy}^{q_1}}\|D_y^{\delta}u\|_{L_T^{r_1}L^{p_1}_{xy}} \\
    &\lesssim \|\partial_xu\|_{L^1_TL^{\infty}_{xy}}^{\theta}\|J_x^{3/2+\delta_0}u\|_{L_T^{\infty}L_{xy}^{2}}^{1-\theta}\|u\|^{1-\theta}_{L^1_TL^{\infty}_{xy}}\big(\|D_y^{1/2}u\|_{L_T^{\infty}L^{2}_{xy}}+\|u\|_{L_T^{\infty}L^{2}_{xy}}\big)^{\theta},
\end{aligned}
\end{equation*}
for some $0<\delta \ll 1$ and $0<\delta_0 < s-3/2$ fixed and where $\frac{1}{s_1}+\frac{1}{r_1}=1$. Given that
\begin{equation*}
    |\eta|^{1/2}=|\xi|^{1/4}(|\xi|^{-1/2}|\eta|)^{1/2} \lesssim |\xi|^{1/2}+|\xi|^{-1/2}|\eta|\lesssim (1+|\xi|)^{s}+|\xi|^{-1/2}|\eta|.
\end{equation*} 
Plancherel's identity yields
\begin{equation}\label{twoeqenerg9.1}
    \|D_y^{1/2}u\|_{L_T^{\infty}L^{2}_{xy}}+\|u\|_{L_T^{\infty}L^{2}_{xy}} \lesssim \|u\|_{L_T^{\infty}X^s}.
\end{equation}
From this we get $ \mathcal{II}_4 \lesssim (\|u\|_{L^1_TL^{\infty}_{xy}}+\|\partial_x u\|_{L^1_TL^{\infty}_{xy}}) \|u\|_{L_T^{\infty}X^s}$.
According to \eqref{twoeqenerg9}, this completes the estimate of $\mathcal{II}$. Collecting the bounds derived for $\mathcal{I}$ and $\mathcal{II}$, we obtain
\begin{equation*}
    \|\partial_x u\|_{L^1_TL^{\infty}_{xy}}\lesssim T^{\kappa_{\delta}}(1+\|u\|_{L^1_T L^{\infty}_{xy}}+\|\partial_x u\|_{L^1_T L^{\infty}_{xy}})\|u\|_{L^{\infty}_TX^s}.
\end{equation*}
On the other hand, the estimate concerning $\|u\|_{L^1_TL^{\infty}_{xy}}$ is obtained by applying Lemma \ref{refinStri} with $F=-u\partial_x u=-\frac{1}{2}\partial_x(u^2)$, estimate \eqref{eqfraLR}, \eqref{twoeqenerg7.0} and the inequality
\begin{equation}\label{twoeqenerg10.1}
    |\eta|^{1/2+2\delta}=|\xi|^{(1+4\delta)/4}(|\xi|^{-1/2}|\eta|)^{(1+4\delta)/2} \lesssim (1+|\xi|)^{\frac{1+4\delta}{2(1-4\delta)}}+|\xi|^{-1/2}|\eta|,
\end{equation}
valid for $0<\delta <1/16$. To avoid repetition we shall omit its proof. However, we emphasized that this estimate does not require to implement Lemma \ref{lemmainterine}. The proof is complete.
\end{proof}

Additionally, we require to control the norm $\|\partial_x^2u\|_{L^1_TL^{\infty}_{xy}}$. This estimate will be useful to close the argument leading to the proof of Theorem \ref{Improwellp} in the space $X^s(\mathbb{R}^2)$.

\begin{lemma}\label{lemmaSecondD}
Let $T>0$ and $u\in C([0,T];X^{\infty}(\mathbb{R}^2))$ be a solution of the IVP \eqref{EQBO}. Then for all $s>3/2$, there exist $\kappa_{\delta}\in (\frac{1}{2},1)$ and $c_s$ such that
\begin{equation*}
    \|\partial_x^2 u\|_{L^1_T L^{\infty}_{xy}} \leq c_sT^{\kappa_s}(1+h(T))\|u\|_{L^{\infty}_T X^{s+1}}+c_sT^{\kappa_s}\|\partial_x^2u\|_{L^1_T L^{\infty}_{xy}}\|u\|_{L^{\infty}_T X^{s}},
\end{equation*}
where $h(T)=\left\|u\right\|_{{L^1_TL^{\infty}_{xy}}}+ \left\|\partial_x u\right\|_{L^1_TL^{\infty}_{xy}}$.
\end{lemma}

\begin{proof}
Applying Lemma \ref{refinStri}  with $F=-\partial_x\big(\partial_x u \partial_x u+u \partial_x^2u\big)$, the proof of Lemma \ref{lemmaSecondD} follows the same arguments in the deduction of \eqref{twoEnergyES2}. 
\end{proof}

\subsection{Proof of Theorem \ref{Improwellp}}

Our results relay on existence of smooth solutions for the IVP \eqref{EQBO}. To achieve this conclusion in the spaces $X^s(\mathbb{R}^2)$, we require the following lemma.
\begin{lemma}\label{lemaexismooth}
Let $s\geq 4$. Then it holds 
\begin{align}
&  (\|u\|_{L^{\infty}_{xy}}+\|\partial_x u\|_{L^{\infty}_{xy}})\lesssim \|u\|_{X^s}, \label{twoeqenerg12} \\
&\left|\int D_x^{-1/2}\partial_y^{l}(u\partial_x u)D_x^{-1/2}\partial_y^{l}u\, dxdy\right|\lesssim \|\partial_x u\|_{L^{\infty}_{xy}}\|D_x^{-1/2}\partial_y^{l}u\|_{L^2_{xy}}^2,  \label{twoeqenerg13}
\end{align}
for every $l=0,1$.
\end{lemma}

\begin{proof}
We first notice that \eqref{twoeqenerg13} is deduced applying the same reasoning in \eqref{twoeqenerg1} and \eqref{twoeqenerg2}, which mostly depends on Proposition \ref{CalderonCom}. 

Next, to deduce \eqref{twoeqenerg12}, we use Sobolev embedding in the variables $x$ and $y$ to get
\begin{equation*}
    \begin{aligned}
    \|\partial_xu\|_{L^{\infty}_{xy}}\lesssim \|J_x^{1/2+\epsilon}J_y^{1/2+\epsilon}\partial_xu\|_{L^{2}_{xy}} &\lesssim \|J_x^{3/2+\epsilon}u\|_{L^2_{xy}}+\|J_x^{3/2+\epsilon}D_y^{1/2+\epsilon}u\|_{L^2_{xy}} \lesssim \|u\|_{X^s},
    \end{aligned}
\end{equation*}
for any $0<\epsilon\ll 1$ and $s\geq 4$, where we have used a similar estimate as in \eqref{twoeqenerg10.1} and Plancherel's identity to estimate $\|J_x^{3/2+\epsilon}D_y^{1/2+\epsilon}u\|_{L^2}$. Since this same reasoning also applies to $\|u\|_{L^{\infty}_{xy}}$, \eqref{twoeqenerg12} follows. The proof is complete. 
\end{proof}
 
 Whenever $s>2$, local well-posedness in $H^s(\mathbb{R}^2)$ for the IVP \eqref{EQBO} follows from a parabolic regularization argument. Roughly speaking, an additional term $-\mu \Delta u$ is added to the equation, after which the limit $\mu \to 0$ is taken. This technique was applied in \cite{Omarths} for the IVP \eqref{EQBO} establishing LWP in $H^s(\mathbb{R}^2)$ for all $s>2$. 
 
Furthermore, employing Lemma \ref{lemaexismooth}, it is possible to apply a parabolic regularization argument adapting the ideas in \cite{Omarths} or \cite[Section 6.2]{Ioriobook} (see also \cite[Theorem 9.2]{linaresBook}), to obtain local well-posedness for the IVP \eqref{EQBO} in $X^s(\mathbb{R}^2)$, $s\geq 4$. Summarizing the preceding discussion we have:
 
\begin{lemma}\label{comwellp}
Let $s\geq 4$ and $\mathfrak{X}^s(\mathbb{R}^2)$ be any (fixed) of the spaces $H^s(\mathbb{R}^2)$ and $X^s(\mathbb{R}^2)$. Then for any $u_0 \in \mathfrak{X}^s(\mathbb{R}^2)$, there exist $T=T(\left\|u_0\right\|_{\mathfrak{X}^s})>0$ and a unique solution $u\in C([0,T]; \mathfrak{X}^s(\mathbb{R}^d))$ of the IVP \eqref{EQBO}. In addition, the flow-map $u_0 \mapsto u(t)$ is continuous in the $\mathfrak{X}^s$-norm.
\end{lemma}

The proof of Lemma \ref{comwellp} also provides existence of smooth solutions and a blow-up criterion. More precisely, let $u_0\in \mathfrak{X}^{\infty}(\mathbb{R}^2)$, where $\mathfrak{X}^{\infty}(\mathbb{R}^2)$ is any (fixed) of the spaces $H^{\infty}(\mathbb{R}^2)$ and $X^{\infty}(\mathbb{R}^2)$, then there exists a solution $u\in C([0,T^{\ast});\mathfrak{X}^{\infty}(\mathbb{R}^2))$ to the IVP \eqref{EQBO}, where $T^{\ast}$ is the maximal time of existence of $u$ satisfying $T^{\ast}>T(\|u\|_{\mathfrak{X}^4})>0$ and the following blow-up alternative holds true
\begin{equation}\label{blowupal}
    \lim_{t \to T^{\ast}} \left\|u(t)\right\|_{\mathfrak{X}^4}=\infty,
\end{equation}
if $T^{\ast}<\infty$.

We require of some additional \emph{a priori} estimates.
\begin{lemma}\label{apriESTlower}
Let $s\in (3/2,4]$. then there exists $A_s>0$ such that for all $u_0 \in X^{\infty}(\mathbb{R}^2)$  there is a solution $u\in C([0,T^{\ast});X^{\infty}(\mathbb{R}^2))$ of the IVP \eqref{EQBO} where $T^{\ast}=T^{\ast}(\left\|u_0\right\|_{X^{4}})>(1+A_s \left\|u_0\right\|_{H^s})^{-2}$. Moreover, there exists a constant $K_0>0$ such that
\begin{equation*}
    \left\|u\right\|_{L^{\infty}_T X^{s}} \leq 2 \left\|u_0\right\|_{X^s}, 
\end{equation*}
and
\begin{equation}\label{aprio1}
      \left\|u\right\|_{L^1_T L^{\infty}_{xy}}+\left\|\partial_x u\right\|_{L^1_T L^{\infty}_{xy}} \leq K_0,
\end{equation}
whenever $0<T\leq (1+A_s\|u_0\|_{H^s})^{-2}$. 
\end{lemma}
\begin{proof}
In view of Lemmas \ref{apriEST}, \ref{apriEstS}, \ref{comwellp} and the blow-up criteria \eqref{blowupal} applied to the $X^{4}$-norm, the proof is obtained by arguing as in \cite[Lemma 5.3]{LinarFKP}.
\end{proof}

Now we can prove the existence of solutions. 

\subsubsection{Existence of solution}

We first establish some auxiliary estimates involving the projectors introduced in \eqref{proje1} and \eqref{proje2}.
\begin{lemma}
Let $0\leq \sigma \leq s$ and $M,N\in \mathbb{D}=\{2^l\, :\, l\in \mathbb{Z}^{+}\cup \{0\}\}$ such that $M\geq N$. Assume that $u_0\in X^s(\mathbb{R}^2)$, then
\begin{equation}\label{eqexist2}
    N^{\sigma}\left\|J_x^{s-\sigma}(P_{\leq N}^xu_{0}-P_{\leq M}^xu_{0})\right\|_{L^2_{xy}}+  \left\|D_x^{-1/2}\partial_y^{l}(P_{\leq N}^xu_{0}-P_{\leq M}^xu_{0})\right\|_{L^2_{xy}}\underset{N\to \infty}{\rightarrow} 0, 
\end{equation}
for each $0\leq \sigma \leq s$ and $l=0,1$.
\end{lemma}
\begin{proof}
By support considerations we observe
\begin{equation*}
\begin{aligned}
    |\langle \xi\rangle^{s-\sigma} (\psi_1(\xi/N)-\psi_1(\xi/M))\widehat{u}_0(\xi,\eta)|^2 \lesssim N^{-2\sigma}|\langle \xi \rangle^s \widehat{u}_0(\xi,\eta)|^2.
\end{aligned}
\end{equation*}
Integrating the above expression, we use Plancherel's identity and Lebesgue dominated convergence theorem to verify that the first norm on the left-hand side (l.h.s) of \eqref{eqexist2} satisfy the desired limit. A similar argument provides the required limit for the second norm on the l.h.s of \eqref{eqexist2}.
\end{proof}

Now, we gather the previous result to derive some conclusion for the smooth solutions generated by some approximations of the initial data.

Let $u_0\in X^s(\mathbb{R}^2)$, $s \in (3/2,4]$. For each dyadic number $N\in \mathbb{D}$, Lemma \ref{apriESTlower} assures the existence of a time $0<T\leq (1+A_s \left\|u_0\right\|_{X^s})^{-2}$ (for some constant $A_s>0$) independent of $N$ and smooth solutions $v_N \in C([0,T];X^{\infty}(\mathbb{R}^2))$ of the IVP \eqref{EQBO} with initial data $P_{\leq N}^xu_0$ such that
\begin{equation}\label{eqexist7}
    \left\|v_N\right\|_{L^{\infty}_T X^s} \leq 2 \left\|u_0\right\|_{X^s}
\end{equation}
and
\begin{equation}\label{eqexist8}
    K_1:=\sup_{N\in \mathbb{D}}\left\{\left\|v_N\right\|_{L^{1}_TL^{\infty}_{xy}} +\left\|\partial_x v_N\right\|_{L^{1}_TL^{\infty}_{xy}} \right\} <\infty.
\end{equation}
Additionally, we combine Lemma \ref{lemmaSecondD}, \eqref{eqexist7} and \eqref{eqexist8} to infer
\begin{equation}\label{eqexist8.1}
    \|\partial_x^2v_N\|_{L^1_TL^{\infty}_{xy}} \lesssim \|v_N\|_{L^{\infty}_T X^{s+1}},
\end{equation}
provided that $A_s$ is chosen large enough. Now, let $M,N \in \mathbb{D}$, $M\geq N$, and $w_{N,M}=v_N-v_M$, so $w_{N,M}$ satisfies the equation
\begin{equation}\label{eqexiscauchy}
    \partial_t w_{N,M} +\mathcal{H}_xw_{N,M}-\mathcal{H}_x\partial_x^2w_{N,M}\pm \mathcal{H}_x\partial_y^2w_{N,M}+\frac{1}{2}\partial_x((v_N+v_M)w_{N,M})=0,
\end{equation}
with initial condition $w_{N,M}(0)=P_{\leq N}^xu_{0}-P_{\leq M}^{x}u_{0}$. Thus, by employing similar energy estimates leading to \eqref{twoEnergyES2}, together with \eqref{eqexist2}, we deduce
\begin{equation}\label{eqexist0}
    N^{s-\sigma}\, \left\|J^{\sigma}_x(v_N-v_{M})\right\|_{L^{\infty}_T L^2_{xy}} \underset{N\to \infty}{\rightarrow}0,
\end{equation}
whenever $0\leq\sigma<s$. 

Accordingly, we shall prove that $\{v_N\}_{N\in \mathbb{D}}$ is a Cauchy sequence in $C([0,T];X^s(\mathbb{R}^2))\cap L^1([0,T],W_x^{1,\infty}(\mathbb{R}^2))$. Let us first estimate the sequence $\{v_N\}$ in $L^1([0,T],W_x^{1,\infty}(\mathbb{R}^2))$. 

\begin{lemma}\label{Cseq1}
Let $M,N \in \mathbb{D}$, $M\geq N$. If $u_0\in X^s(\mathbb{R}^2)$, $s\in (3/2,4]$, then 
    \begin{equation}\label{eqexist0.2}
    \|v_{N}-v_M\|_{L^1_TL^{\infty}_{xy}} \underset{N\to \infty}{=} o(N^{-1})+O(T^{1/2}N^{-1}\|D_x^{-1/2}\partial_y(v_{N}-v_M)\|_{L^{\infty}_T L^2_{xy}})
\end{equation}
and 
\begin{equation}\label{eqexist0.3}
    \|\partial_x(v_{N}-v_M)\|_{L^1_TL^{\infty}_{xy}} \underset{N\to \infty}{=} o(1)+O(T^{1/2}\|D_x^{-1/2}\partial_y(v_{N}-v_M)\|_{L^{\infty}_T L^2_{xy}}),
\end{equation}
provided that $0<T\leq (1+A_s\|u_0\|_{X^s})^{-2}$ with $A_s$ large enough.
\end{lemma}

\begin{proof}
Since \eqref{eqexist0.2} and \eqref{eqexist0.3} are inferred as in the proof of Lemma \ref{apriEstS}, we will only deduce \eqref{eqexist0.2}. Recalling that $w_{N,M}=v_N-v_M$ satisfies \eqref{eqexiscauchy}, we apply Lemma \ref{refinStri} with $F=-\frac{1}{2}\partial_x((v_N+v_M)w_{N,M}))$ to get
\begin{equation}\label{eqexist14.0}
\begin{aligned}
\|v_N-v_M\|_{L^1_TL^{\infty}_{xy}}&\lesssim_{\delta} T^{1/2}\big(\sup_{[0,T]}\|J_x^{1/2+2\delta}w(t)\|_{L^2_{xy}}+\sup_{[0,T]}\|J_x^{1/2+\delta}D_y^{\delta}w(t)\|_{L^2_{xy}} \\
&+\int_0^T ( \|J^{1/2+2\delta}_x((v_N+v_M)w_{N,M})(t') \|_{L^2_{xy}}+\|J^{1/2+\delta}_x D_y^{\delta}((v_N+v_M)w_{N,M})(t') \|_{L^2_{xy}} )\, dt'\big) \\
&=: T^{1/2}\big(\mathcal{I}_1+\mathcal{I}_2+\mathcal{I}_3+\mathcal{I}_4),
\end{aligned}
\end{equation}
for some $0<\delta<\delta_0$ with $\delta_0$ to be determined along the proof. Now, we proceed to estimate each of the factors $\mathcal{I}_j$. 

In view of \eqref{eqexist0}, it follows that $N \, \mathcal{I}_1 \underset{N\to \infty}{\rightarrow} 0,$ whenever $0<\delta<\delta_0<s/2-3/4$. To study $\mathcal{I}_2$, we employ Young's inequality to derive
\begin{equation}\label{eqexist14.1}
    (1+|\xi|)^{1/2+\delta}|\eta|^{\delta}\lesssim N^{\frac{\delta}{1-\delta}}(1+|\xi|)^{\frac{1+3\delta}{2(1-\delta)}}+N^{-1}|\eta||\xi|^{-1/2}. 
\end{equation}
Plancherel's identity shows
\begin{equation}\label{eqexist15}
   \mathcal{I}_2 \lesssim N^{\frac{\delta}{1-\delta}}\|J_x^{\frac{1+3\delta}{2(1-\delta)}} w_{N,M}\|_{L^{\infty}_TL^2_{xy}}+N^{-1}\|D_x^{-1/2}\partial_y w_{N,M}\|_{L^{\infty}_TL^2_{xy}}.
\end{equation}
Therefore, choosing $0<\delta<\delta_0<1$, where $\delta_0$ is small satisfying $\frac{1+5\delta_0}{2(1-\delta_0)}<s-1$, we have from \eqref{eqexist0} and \eqref{eqexist15} that
\begin{equation*}
     \mathcal{I}_2 \underset{N\to \infty}{=} o(N^{-1})+O(N^{-1}\|D_x^{-1/2}\partial_y (v_N-v_{M})\|_{L^{\infty}_TL^2_{xy}}).
\end{equation*}
Next, by employing \eqref{eqfraLR}, we follow the arguments in \eqref{twoeqenerg7.2} to deduce
\begin{equation*}
\begin{aligned}
&\mathcal{I}_3 \lesssim   (\|J_x^{1/2+2\delta}v_N\|_{L^{\infty}_TL^2_{xy}}+\|J_x^{1/2+2\delta}v_M\|_{L^{\infty}_TL^2_{xy}})\|v_N-v_M\|_{L^1_T L^{\infty}_{xy}}\\
& \hspace{0.7cm}+(\|v_N\|_{L^{1}_TL^{\infty}_{xy}}+\|v_M\|_{L^{1}_TL^{\infty}_{xy}})\|J_x^{1/2+2\delta}(v_N-v_M)\|_{L^{\infty}_T L^{2}_{xy}} \\
&  \hspace{0.1cm} \underset{N\to \infty}{=}  O(\|u_0\|_{X^s}\|v_N-v_M\|_{L^1_T L^{\infty}_{xy}})+o(N^{-1}),
\end{aligned}
\end{equation*}
for all $0<\delta<\delta_0$, with $\delta_0<s/2-3/4$, where we have used \eqref{eqexist7}, \eqref{eqexist8} and \eqref{eqexist0}. Now, we divide the remaining term $\mathcal{I}_4$ as follows  
\begin{equation*}
\begin{aligned}
\mathcal{I}_4 &\lesssim \|\,\|D_y^{\delta}((v_N+v_M)w_{N,M} )(t) \|_{L^2_{xy}}\|_{L^1_T} +\|\,\|D^{1/2+\delta}_x D_y^{\delta}((v_N+v_M)w_{N,M})(t) \|_{L^2_{xy}}\|_{L^1_T} \\
&=:\mathcal{I}_{4,1}+\mathcal{I}_{4,2}.   
\end{aligned}
\end{equation*}
By employing the fractional Leibniz's rule \eqref{libRule} in the $y$-variable, \eqref{eqexist7} and a similar argument to \eqref{eqexist14.1}, it is seen
\begin{equation*}
\begin{aligned}
\mathcal{I}_{4,1} \underset{N\to \infty}{=} O(\|u_0\|_{X^s}\|v_N-v_M\|_{L^1_T L^{\infty}_{xy}})+o(N^{-1})+O(N^{-1}\|D_x^{-1/2}\partial_y(v_N-v_M)\|_{L^{\infty}_{T}L^2_{xy}}),
\end{aligned}
\end{equation*}
for all $0<\delta<\delta_0<1$ such that $\frac{3\delta_0}{2(1-\delta_0)}<s$. On the other hand, from Lemma \ref{commtwovar} it is deduced
\begin{equation*}
\begin{aligned}
    \mathcal{I}_{4,2}\lesssim & (\|v_N\|_{L^1_TL^{\infty}_{xy}}+\|v_M\|_{L^1_TL^{\infty}_{xy}})\|D_x^{1/2+\delta}D_y^{\delta} (v_N-v_M)\|_{L_{T}^{\infty}L^2_{xy}}\\ &+(\|D_x^{1/2+\delta}D_y^{\delta}v_N\|_{L^{\infty}_{T}L^2_{xy}}+\|D_x^{1/2+\delta}D_y^{\delta}v_M\|_{L^{\infty}_TL^2_{xy}})\|v_N-v_M\|_{L^1_T L^{\infty}_{xy}}\\
    &+(\| D_x^{1/2+\delta}v_N\|_{L^1_TL^{\infty}_{xy}}+\| D_x^{1/2+\delta}v_M\|_{L^1_TL^{\infty}_{xy}})\|D_y^{\delta}(v_N-v_M)\|_{L^{\infty}_TL^2_{xy}}\\ 
    &+(\|D_y^{\delta}v_N\|_{L^{r_1}_TL^{p_1}_{xy}}+\|D_y^{\delta}v_M\|_{L^{r_1}_TL^{p_1}_{xy}})\| D_x^{1/2+\delta}(v_N-v_M)\|_{L_{T}^{s_1}L^{q_1}_{xy}} \\
    =&:\mathcal{I}_{4,2,1}+\mathcal{I}_{4,2,2}+\mathcal{I}_{4,2,3}+\mathcal{I}_{4,2,4},
\end{aligned}
\end{equation*}
where $1<r_1,s_1<\infty$ and $2<p_1,q_1<\infty$ satisfy the conditions in Lemma \ref{lemmainterine} (ii). An application of \eqref{eqexist14.1} shows 
\begin{equation*}
     \mathcal{I}_{4,2,1} \underset{N\to \infty}{=} o(N^{-1})+O(N^{-1}\|D_x^{-1/2}\partial_y (v_N-v_M)\|_{L^{\infty}_TL^2_{xy}}),
\end{equation*}
for each $0<\delta<\delta_0<1$, where $\delta_0$ is small satisfying $\frac{1+5\delta_0}{2(1-\delta_0)}<s-1$. Now, we combine estimate  \eqref{twoeqenerg7.0} and\eqref{eqexist7} to derive
\begin{equation*}
    \mathcal{I}_{4,2,2}\lesssim \|u_0\|_{X^s}\|v_N-v_M\|_{L^1_TL^{\infty}_{xy}}.
\end{equation*}
Additionally, employing \eqref{Interp1} and identity \eqref{eqexist15}, it is not difficult to see 
\begin{equation*}
    \mathcal{I}_{4,2,3}\underset{N\to \infty}{=} o(N^{-1})+O(N^{-1}\|D_x^{-1/2}\partial_y (v_N-v_M)\|_{L^{\infty}_TL^2_{xy}}),
\end{equation*}
for all $0<\delta<\delta_0<1$ and $\frac{1+5\delta_0}{2(1-\delta_0)}<s-1$. Finally, gathering together estimates \eqref{Interp2.0}, \eqref{Interp3}, \eqref{eqexist7} and \eqref{eqexist8}, we deduce
\begin{equation*}
\begin{aligned}
    \mathcal{II}_{4,2,4} 
\lesssim K_1^{1-\theta}\|u_0\|_{X^s}^{\theta}\|v_N-v_M\|_{L^1_T L^{\infty}_{xy}}^{\theta}\|J_x^{1/2+\delta_0}(v_N-v_M)\|_{L^1_T L^{\infty}_{xy}}^{1-\theta},
\end{aligned}
\end{equation*}
so that Young's inequality and \eqref{eqexist0} yield
\begin{equation*}
    \mathcal{I}_{4,2,4}\underset{N\to \infty}{=} o(N^{-1})+O(\|u_0\|_{X^s}\|v_N-v_M\|_{L^{1}_TL^{\infty}_{xy}}),
\end{equation*}
for all $0<\delta \leq \delta_0$ and $0<\delta_0 \ll 1$ ($\delta_0 < s-3/2$) given by Lemma \ref{lemmainterine}.  Collecting all the previous estimates
\begin{equation*}
    \mathcal{I}_4 \underset{N\to \infty}{=} o(N^{-1})+O(\|u_0\|_{X^s}\|v_N-v_M\|_{L^{1}_TL^{\infty}_{xy}})+O(N^{-1}\|D_x^{-1/2}\partial_y(v_N-v_M)\|_{L_T^{\infty}L^2_{xy}}).
\end{equation*}
Plugging the bounds obtained for the terms $\mathcal{I}_j$, $j=1,\dots,4$ in  \eqref{eqexist14.0}, we obtain
\begin{equation*}
    \|v_N-v_M\|_{L_T^1 L^{\infty}_{xy}} \underset{N\to \infty}{=} o(N^{-1})+O(T^{1/2}\|u_0\|_{X^s}\|v_N-v_M\|_{L^{1}_TL^{\infty}_{xy}})+O(T^{1/2}N^{-1}\|D_x^{-1/2}\partial_y(v_N-v_M)\|_{L_T^{\infty}L^2_{xy}}).
\end{equation*}
This completes the deduction of \eqref{eqexist0.2} provided that $0<T\leq (1+A_s\|u_0\|_{X^s})^{-2}$ and $A_s$ is chosen sufficiently large.
\end{proof}

Next, we shall prove that $\{v_N\}$ is a Cauchy sequences in the space $C([0,T];X^{s}(\mathbb{R}^2))$. 

\begin{prop}\label{Cseq2}
Let $M,N\in \mathbb{D}$, $M\geq N$. If $u_0\in X^s(\mathbb{R}^2)$ $s\in (3/2,4]$, then 
    \begin{equation}\label{eqexis17}
    \|J_x^s(v_N-v_M)\|_{L^{\infty}_TL^2_{xy}}+\|D_x^{-1/2}(v_N-v_M)\|_{L^{\infty}_TL^2_{xy}}+\|D_x^{-1/2}\partial_y(v_{N}-v_M)\|_{L^{\infty}_TL^{2}_{xy}} \underset{N\to \infty}{\rightarrow} 0.
\end{equation}
\end{prop}

\begin{proof}
We apply $J^s_x$ to \eqref{eqexiscauchy} and then multiplying by $J^s_xw_{N,M}$ and integrating the resulting expression in space, we deduce
\begin{equation}\label{differINEQ}
\begin{aligned}
\frac{1}{2}\frac{d}{dt}\|J^s_x(v_N-v_M)(t)\|_{L^2_{xy}}^2=&-\int J^s_x(v_M\partial_x w_{N,M})J^s_xw_{N,M}\, dxdy-\int J^s_x(\partial_x v_N w_{N,M})J^s_xw_{N,M}\, dxdy \\
=:& -(\mathcal{I}+\mathcal{II}).
\end{aligned}
\end{equation}
Integrating by parts,
\begin{equation*}
    \begin{aligned}
    \mathcal{I}=\int [J^s_x,v_M]\partial_x w_{N,M} J^s_xw_{N,M}\, dx dy-\frac{1}{2}\int \partial_x v_M(J^s_x w_{N,M})^2 \, dxdy,
    \end{aligned}
\end{equation*}
which together with Lemma \ref{conmKP} yield
\begin{equation}\label{eqexis18}
    \begin{aligned}
    |\mathcal{I}|&\lesssim \|\|[J^s_x,v_M]\partial_x w_{N,M}\|_{L^2_x}\|_{L^2_{y}}\|J^s_x w_{N,M}\|_{L^2_{xy}}+\|\partial_x v_M\|_{L^2_{xy}}\|J^s_x w_{N,M}\|_{L^2_{xy}}^2 \\
    &\lesssim \|J^s_xv_M\|_{L^2_{xy}}\|\partial_x w_{N,M}\|_{L^{\infty}_{xy}}\|J^s_x w_{N,M}\|_{L^2_{xy}}+\|\partial_x v_M\|_{L^{\infty}_{xy}}\|J^s_x w_{N,M}\|_{L^2_{xy}}^2.
    \end{aligned}
\end{equation}
On the other hand, 
\begin{equation*}
    \begin{aligned}
    \mathcal{II}=\int [J^s_x,w_{N,M}]\partial_x v_N J^s_x w_{N,M} \, dx dy +\int w_{N,M}(\partial_x J^s_x v_N)J^s_x w_{N,M}\, dx dy,
    \end{aligned}
\end{equation*}
then Lemma \ref{conmKP} gives
\begin{equation}
    \begin{aligned}\label{eqexis19}
    |\mathcal{II}|\lesssim & \|\|[J^s_x, w_{N,M}]\partial_x v_N\|_{L^{2}_x}\|_{L^2_{y}}\|J^s_x w_{N,M}\|_{L^2_{xy}}+\|w_{N,M}\|_{L^{\infty}_{xy}}\|J^{s+1}_xv_N\|_{L^2_{xy}}\|J^s w_{N,M}\|_{L^2_{xy}}\\
    \lesssim & \|\partial_x w_{N,M}\|_{L^{\infty}_{xy}}\|J^s_x v_{N}\|_{L^2_{xy}}\|J^s_x w_{N,M}\|_{L^2_{xy}}+\|\partial_x v_N\|_{L^{\infty}_{xy}}\|J^s_x w_{N,M}\|_{L^2_{xy}}^2 \\
    &+\|w_{N,M}\|_{L^{\infty}_{xy}}\|J^{s+1}_xv_N\|_{L^2_{xy}}\|J^s_x w_{N,M}\|_{L^2_{xy}}.
    \end{aligned}
\end{equation}
To control $\|J^{s+1}_xv_N\|_{L^2_{xy}}$, we employ the fact that $v_N$ solves the equation in \eqref{EQBO} to apply energy estimates with Lemma \ref{conmKP} to find
\begin{equation}\label{eqexis20.1}
    \|J^{s+1}_x v_N\|_{L^{\infty}_TL^2_{xy}}\lesssim e^{c(\|v_N\|_{L^{1}_TL^{\infty}_{xy}}+\|\partial_x v_N\|_{L^{1}_TL^{\infty}_{xy}})}\|J^{s+1}_xP_{\leq N}^x u_0\|_{L^2_{xy}}\lesssim N\|J^{s}_xu_0\|_{L^2_{xy}},
\end{equation}
where we have also used Gronwall's inequality, \eqref{eqexist7} and \eqref{eqexist8}. Therefore, gathering \eqref{eqexis18}-\eqref{eqexis20.1}, and  \eqref{eqexist7} and \eqref{eqexist8} in \eqref{differINEQ}, after Gronwall's inequality we get
\begin{equation*}
\begin{aligned}
    \|J^s_x(v_N-v_M)\|_{L^{\infty}_TL^2_{xy}} &\lesssim \|J^s_x(P_{\leq N}^x u_0-P_{\leq M}^xu_0)\|_{L^2_{xy}}+\|\partial_x w_{N,M}\|_{L^1_TL^{\infty}_{xy}}+N\| w_{N,M}\|_{L^1_TL^{\infty}_{xy}} \\
    &\underset{N\to \infty}{=} o(1)+O(\|D_x^{-1/2}\partial_y(v_{N}-v_M)\|_{L^{\infty}_T L^2_{xy}}),
\end{aligned}
\end{equation*}
which holds in virtue of Lemma \ref{Cseq1} and \eqref{eqexist2}. Once we have established that $\|D_x^{-1/2}\partial_y(v_{N}-v_M)\|_{L^{\infty}_T L^2_{xy}} \to 0$ as $N\to \infty$ the estimate for the $J^s_x$ norm of $w_{N,M}$ will be completed.
\\ \\
Now, applying $D_x^{-1/2}$ to \eqref{eqexiscauchy}, and then multiplying by $D_x^{-1/2}w_{N,M}$ and integrating in space, we have
\begin{equation*}
\begin{aligned}
    \frac{1}{2}\frac{d}{dt}\|D_x^{-1/2}w_{N,M}(t)\|_{L^2_{xy}}^2&=-\int D_x^{-1/2}\partial_x((v_N+v_M)w_{N,M})D_x^{-1/2}w_{N,M}\, dx dy\\
    &=\int \mathcal{H}_x((v_N+v_M)w_{N,M})w_{N,M}\, dxdy.
\end{aligned}
\end{equation*}
Now, given that $\mathcal{H}_x$ determines a skew-symmetric operator, it is seen
    \begin{equation}\label{eqexis20.01}
\begin{aligned}
    \int \mathcal{H}_x((v_N+v_M)w_{N,M})&w_{N,M}\, dx dy \\
    &=\frac{1}{2}\int [\mathcal{H}_x,v_N+v_M]w_{N,M} w_{N,M}\, dx dy \\
    &=\frac{1}{2}\int( D_x^{1/2}[\mathcal{H}_x,v_N+v_M]D_x^{1/2}(D_x^{-1/2}w_{N,M}) )D_x^{-1/2}w_{N,M}\, dx dy,
\end{aligned}
\end{equation}
so that Proposition \ref{CalderonCom} applied to the $x$-variable gives
\begin{equation*}
    \begin{aligned}
    \frac{1}{2}\frac{d}{dt}\|D_x^{-1/2}w_{N,M}(t)\|_{L^2_{xy}}^2 &\lesssim \|D_x^{1/2}[\mathcal{H}_x,v_N+v_M]D_x^{1/2}(D_x^{-1/2}w_{N,M})\|_{L^2_{xy}}\|D_x^{-1/2}w_{N,M}\|_{L^2_{xy}} \\
    &\lesssim \|\partial_x(v_N+v_M)\|_{L^{\infty}_{xy}} \|D_x^{-1/2}w_{N,M}\|_{L^2_{xy}}^{2}. 
    \end{aligned}
\end{equation*}
Therefore, the preceding differential inequality, Gronwall's lemma, \eqref{eqexist8} and \eqref{eqexist2} imply
\begin{equation*}
    \|D_x^{-1/2}(v_N-v_M)\|_{L^{\infty}_TL^2_{xy}}\lesssim e^{cK_1}\|D_x^{-1/2}(P_{\leq N}^xu_0-P_{\leq M}^x u_0)\|_{L^2_{xy}} \underset{N\to \infty}{\rightarrow} 0. 
\end{equation*}
Finally, we proceed to estimate the $\|D_x^{-1/2}\partial_y(v_N-v_M)\|_{L^{\infty}_T L^2_{xy}}$ norm. Since $w_{N,M}=v_N-v_M$ solves \eqref{eqexiscauchy}, we apply $D_x^{-1/2}\partial_y$ to this equation, multiplying by $D_x^{-1/2}\partial_y w_{N,M}$, then integrating in space, we deduce
\begin{equation}\label{eqexis20.10} 
\begin{aligned}
  \frac{1}{2}\frac{d}{dt}\|D_x^{-1/2}&\partial_yw_{N,M}(t)\|_{L^2_{xy}}^2 \\
  &=-\int D_x^{-1/2}\partial_y\partial_x((v_N+v_M)w_{N,M}) D_x^{-1/2}\partial_y w_{N,M} \, dx dy \\ 
  &=\int \mathcal{H}_x((v_N+v_M)\partial_yw_{N,M})\partial_y w_{N,M}\, dx dy+\int \mathcal{H}_x(\partial_y(v_N+v_M)w_{N,M})\partial_y w_{N,M}\, dx dy \\
  &=:\mathbb{I}+\mathbb{II},
\end{aligned}
\end{equation}
where we have employed the decomposition $\partial_x=-\mathcal{H}_xD_x$. Arguing as in \eqref{eqexis20.01},  Proposition \ref{CalderonCom} shows
\begin{equation}\label{eqexis20.11}
\begin{aligned}
  |\mathbb{I}| &\lesssim \|D_x^{1/2}[\mathcal{H}_x,v_N+v_M]D_x^{1/2}(D_x^{-1/2}\partial_y w_{N,M})\|_{L^2_{xy}}\|D_x^{-1/2}\partial_y w_{N,M}\|_{L^2_{xy}} \\
    & \lesssim \|\partial_x(v_N-v_M)\|_{L^{\infty}_{xy}}\|D_x^{-1/2}\partial_yw_{N,M}\|_{L^2_{xy}}^2.
\end{aligned}
\end{equation}
On the other hand, we use H\"older's inequality to find
\begin{equation}\label{eqexis20.2}
    |\mathbb{II}|\lesssim \|w_{N,M}\|_{L^{\infty}_{xy}}\|\partial_yw_{N,M}\|_{L^2_{xy}}^2+\|\partial_y v_N\|_{L^2_{xy}}\|w_{N,M}\|_{L^{\infty}_{xy}}\|\partial_yw_{N,M}\|_{L^2_{xy}}.
\end{equation}
According to the above estimate, we are led to bound the norms $\|\partial_y v_N\|_{L^2_{xy}}$ and $\|\partial_y w_{N,M}\|_{L^2_{xy}}$. Thus, given that $v_N$ satisfies the equation in \eqref{EQBO}, integrating by parts it follows that
\begin{equation*}
    \frac{1}{2}\frac{d}{dt}\|\partial_y v_N(t)\|_{L^2_{xy}}^2 =-\int \partial_y(v_N\partial_x v_N)\partial_y v_N \, dx dy=-\frac{1}{2}\int \partial_x v_N (\partial_y v_N)^2 \, dx dy.
\end{equation*}
Hence, Gronwall's inequality and \eqref{eqexist8} yield
\begin{equation}\label{eqexis21}
    \|\partial_y v_N\|_{L^{\infty}_{T}L^2_{xy}} \leq e^{cK_1} \|\partial_y P_{\leq N}^xu_0\|_{L^2_{xy}}\lesssim N^{1/2}\|D_x^{-1/2}\partial_y u_0\|_{L^2_{xy}}=O(N^{1/2}).
\end{equation}
Now, from the fact that $w_{N,M}$ solves \eqref{eqexiscauchy} and integrating by parts we find
\begin{equation}\label{eqexis22}
\begin{aligned}
  \frac{1}{2}\frac{d}{dt}\|\partial_y w_{N,M}\|_{L^2_{xy}}^2=&-\frac{1}{2}\int \partial_x \partial_y ((v_N+v_M)w_{N,M})\partial_y w_{N,M} \, dx dy \\
  =&-\int \partial_x \partial_yv_Nw_{N,M}\partial_y w_{N,M}\, dx dy-\int \partial_y v_N\partial_x w_{N,M} \partial_y w_{N,M} \, dx dy \\
  &-\frac{1}{2}\int \partial_x v_M(\partial_y w_{N,M})^2 \, dx dy \\
  =:& \mathbb{III}_1+\mathbb{III}_2+\mathbb{III}_3.
\end{aligned}
\end{equation}
To estimate $\mathbb{III}_1$, we employ that $v_N$ solves the equation in \eqref{EQBO} to get
\begin{equation*}
    \frac{1}{2}\frac{d}{dt}\|\partial_y \partial_x v_N(t)\|_{L^2_{xy}}^2=-\frac{3}{2}\int \partial_x v_N(\partial_y \partial_x v_N)^2\, dx dy-\int \partial_x^2 v_N \partial_y v_N \partial_y \partial_x v_N \, dx dy.
\end{equation*}
From this estimate and \eqref{eqexis21}, it is seen
\begin{equation*}
    \frac{1}{2}\frac{d}{dt}\|\partial_y \partial_x v_N(t)\|_{L^2_{xy}}^2\lesssim \|\partial_x v_N\|_{L^{\infty}_{xy}}\|\partial_y \partial_x v_N\|_{L^2_{xy}}^2+\|\partial_x^2 v_N\|_{L^{\infty}_{xy}}\|\partial_y v_N\|_{L^2_{xy}}\|\partial_y \partial_x v_N\|_{L^2_{xy}}.
\end{equation*}
Then, in view of \eqref{eqexist7}-\eqref{eqexist8.1}, \eqref{eqexis20.1}, \eqref{eqexis21} and Gronwall's inequality
\begin{equation*}
    \|\partial_y \partial_x v_N\|_{L^{\infty}_TL^2_{xy}} \lesssim e^{cK_1}\big(\|\partial_y\partial_x P_{\leq N}^x u_0\|_{L^2_{xy}}+N^{1/2}\|\partial_x^2 v_N\|_{L^1_T L^{\infty}_{xy}} \big) =O(N^{3/2}),
\end{equation*}
where we have used that $\|\partial_y\partial_x P_{\leq N}^x u_0\|_{L^2_{xy}}\lesssim N^{3/2}\|D_x^{-1/2}\partial_y u_0\|_{L^2_{xy}}$. Consequently,  the previous estimate allows us to deduce 
\begin{equation*}
    \mathbb{III}_1 \lesssim N^{3/2}\|w_{N,M}\|_{L^{\infty}_{xy}}\|\partial_y w_{N,M}\|_{L^{2}_{xy}}.
\end{equation*}
Now, by using \eqref{eqexis21} and H\"older's inequality,
\begin{equation*}
    \begin{aligned}
    \mathbb{III}_2+\mathbb{III}_3 \lesssim N^{1/2}(\|\partial_x v_N\|_{L^{\infty}_{xy}}+\|\partial_x v_M\|_{L^{\infty}_{xy}})\|\partial_y w_{N,M}\|_{L^{2}_{xy}}+\|\partial_x v_M\|_{L^{\infty}_{xy}}\|\partial_y w_{N,M}\|_{L^2_{xy}}^2. 
    \end{aligned}
\end{equation*}
Thus, inserting the above estimates in \eqref{eqexis22}, applying Gronwall's inequality together with \eqref{eqexist7}, \eqref{eqexist8} and \eqref{eqexist0.2} reveal
\begin{equation}\label{eqexis23}
\begin{aligned}
   \|\partial_y w_{N,M}\|_{L^{\infty}_TL^2_{xy}} \lesssim e^{cK_1}\big(&\|\partial_y(P_{\leq N}^x u_0-P_{\leq M}^x u_0)\|_{L^2_{xy}}+N^{3/2}\|w_{N,M}\|_{L^1_TL^{\infty}_{xy}} \\
   & +N^{1/2}(\|\partial_x v_N\|_{L^1_TL^{\infty}_{xy}}+\|\partial_x v_M\|_{L^1_TL^{\infty}_{xy}})\big) \lesssim N^{1/2}+N^{3/2}\|w_{N,M}\|_{L^1_T L^{\infty}_{xy}}.
\end{aligned}
\end{equation}
Going back to $\mathbb{II}$, we plug \eqref{eqexis21} and \eqref{eqexis23} into \eqref{eqexis20.2} to obtain
\begin{equation}\label{eqexis24}
    |\mathbb{II}|\lesssim N\|w_{N,M}\|_{L^{\infty}_{xy}}+N^2\|w_{N,M}\|_{L^1_T L^{\infty}_{xy}}\|w_{N,M}\|_{L^{\infty}_{xy}}.
\end{equation}
Now, collecting \eqref{eqexis20.11}, \eqref{eqexis24} in \eqref{eqexis20.10},
\begin{equation*}
\begin{aligned}
\frac{1}{2}\frac{d}{dt}\|D_{x}^{-1/2}\partial_y w_{N,M}(t)\|_{L^2_{xy}}^2 \lesssim & (\|v_N+v_M\|_{L^{\infty}_{xy}}+\|\partial_x(v_N+v_M)\|_{L^{\infty}_{xy}})\|D_x^{-1/2}\partial_y w_{N,M}\|_{L^{2}_{xy}}^2\\
&+N\|w_{N,M}\|_{L^{\infty}_{xy}}+N^2\|w_{N,M}\|_{L^1_T L^{\infty}_{xy}}\|w_{N,M}\|_{L^{\infty}_{xy}}.
\end{aligned}
\end{equation*}
Then, applying Gronwall's inequality to the preceding inequality, together with \eqref{eqexist7}, \eqref{eqexist8}  yield
\begin{equation*}
    \begin{aligned}
    \|D_{x}^{-1/2}\partial_y w_{N,M}\|_{L^{\infty}_TL^2_{xy}} \lesssim e^{cK_1}\big(\|D_x^{-1/2}\partial_y(P_{\leq N}^x u_0-P_{\leq M}^x u_0)\|_{L^2_{xy}}+o(1)\big)\underset{N\to \infty}{\rightarrow} 0.
    \end{aligned}
\end{equation*}
where we have used \eqref{eqexist0.2} with $0<T<(1+A_s\|u_0\|_{X^s})^{-2}$ and $A_s$ large enough. This completes the proof of \eqref{eqexis17}. 
\end{proof}
Consequently, Lemma \ref{Cseq1} and Proposition \ref{Cseq2} establish that $\{v_N\}$ has a limit $v$ in the class
\begin{equation*}
    C([0,T];X^s(\mathbb{R}^2))\cap L^1([0,T];W^{1,\infty}_x(\mathbb{R}^2)).
\end{equation*}
Thus, since $v_N$ solve the integral equation
\begin{equation}\label{inteqN}
    v_N(t)=S(t)P_{\leq N}^x u_0-\frac{1}{2}\int_0^t S(t-t')\partial_x(v_N(t'))^2 \, dt',
\end{equation}
taking the limit $N\to \infty$ in the class $C([0,T];J^2X^s(\mathbb{R}^2))$, where $J^2X^s(\mathbb{R}^2)=\{f\in S'(\mathbb{R}^2): J^{-2}f \in X^{s}(\mathbb{R}^2)\}$ with norm $\|f\|_{J^2X^s}=\|J^{-2}f\|_{X^s}$, we deduce that $v$ solves the IVP \eqref{EQBO}. This completes the existence part of Theorem \ref{Improwellp}.

\subsubsection{Uniqueness and Continuous Dependence}

Uniqueness of solution in the classes  $$ C([0,T];X^s(\mathbb{R}^2))\cap L^1([0,T];W_x^{1,\infty}(\mathbb{R}^2))$$ 
can be easily obtained by applying energy estimates at the $L^2$-level following the same ideas in Lemma \ref{apriEST}. For the sake of brevity, we omit these computations. Continuous dependence on the spaces $X^s(\mathbb{R}^2)$, follows from the continuity of the flow-map data solution in Lemma \ref{comwellp} and the ideas in \cite{LinarFKP}.

\subsubsection{Solutions in \texorpdfstring{$H^s(\mathbb{R}^2)$}{}}\label{subHs}

With the aim of Lemmas \ref{refinStri}, \ref{comwellp} and the blow-up alternative \eqref{blowupal}, the proof of local well-posedness in $H^{s}(\mathbb{R}^2)$, $s>3/2$ follows the same ideas in \cite[Theorem 1.3]{linaO}.
 
Similarly, the proof of local well-posedness in $H^{s}(\mathbb{R}^2)$ can also be deduced from the arguments employed above to estimate the $\|J_x^s(\cdot)\|_{L^2}$-norm of the space $X^s(\mathbb{R}^2)$. However, when replacing $J^s_x$ by $J^s$ in our estimates, we require to employ Lemmas \ref{conmKP} and \ref{fraLR} in the full spatial variables, as a consequence the estimate of $\|\nabla u\|_{L^{1}_T L^{\infty}_{xy}}$ for solutions of the IVP \eqref{EQBO} is essential in this part.

Summarizing, let $u_0\in H^{s}(\mathbb{R}^2)$, $s>3/2$, then for all dyadic $N \in \mathbb{D}$ there exist a time
\begin{equation}\label{eqexistime1}
0<T\leq (1+A_s \left\|u_0\right\|_{H^s})^{-2}
\end{equation}
independent of $N$ and a solution $u_N \in C([0,T];H^{\infty}(\mathbb{R}^2))$ of the IVP \eqref{EQBO} with initial data $P_{\leq N}u_0$, such that
\begin{equation}\label{eqexist4}
   \left\|u_N\right\|_{L^{\infty}_T H^s} \leq 2 \left\|u_0\right\|_{H^s}
\end{equation}
and
    \begin{equation}\label{aproxres}
   u_N \to u \hspace{0.2cm} \text{ in the sense of } \hspace{0.2cm} C\big([0,T];H^s(\mathbb{R}^2))\cap L^1([0,T];W^{1,\infty}(\mathbb{R}^2)\big).
\end{equation}
This conclusion will be useful to perform rigorously weighted energy estimates at the $H^s(\mathbb{R}^2)$-level stated in Theorem \ref{localweigh}.


\section{Periodic Solutions}\label{sectionPeri}

\subsection{Function spaces and additional notation}

We will follow the notation in \cite{IonescuKeniTata} (see also, \cite{RibaVento,RoberTtFKP,RobertS,ZhangKP}). We recall that $\mathbb{D}=\left\{2^{l}\, :\,l\in \mathbb{Z}^{+}\cup\left\{0\right\}\right\}$. The variable $L$ is presumed to be dyadic. Given $N_1,N_2 \in \mathbb{D}$, we define $N_1\vee N_2=\max(N_1,N_2)$ and  $N_1\wedge N_2=\min(N_1,N_2)$. For each $N\in \mathbb{D}\setminus\left\{1\right\}$, we set $I_N = \left\{m \in  \mathbb{Z}^{+} : N/2 \leq |m| < N\right\}$ and $I_1=\left\{0\right\}$. Let $N,L\in \mathbb{D}$, we set
\begin{equation}\label{Dset}
    D_{N,L}=\left\{(m,n,\tau)\in \mathbb{Z}^2\times \mathbb{R}: |(m,n)|\in I_{N} \text{ and } |\tau-\omega(m,n)|\leq L \right\},
\end{equation}
where $\omega(m,n)=\sign(m)+\sign(m)m^2\mp \sign(m)n^2$ is defined as in \eqref{lieareqsym}.

Now, we introduce some family of projectors required for our arguments. To simplify notation, we will employ the same symbols used in \eqref{proje2} restricted to this section. We define the projector operators in $L^2(\mathbb{T}^2\times \mathbb{R})$ by the relation
\begin{equation*}
    \mathcal{F}(P_N(u))(m,n,\tau)=\mathbbm{1}_{I_N}(|(m,n)|)\mathcal{F}(u)(m,n,\tau),
\end{equation*}
for all $m,n \in \mathbb{Z}$ and $\tau \in \mathbb{R}$, here $\mathbbm{1}_{I_N}$ stands for the indicator function on the set $I_N$. Given a dyadic number $N$, we define the operator $P_{\leq N}u$ by the Fourier multiplier $\mathbbm{1}_{I_{\leq N}}(|(m,n)|)$, where $I_{\leq N}=\bigcup_{M\leq N} I_M$ with $M$ dyadic. We also set $P_{>M}u=(I-P_{\leq M})u$. 

For a time $T_0\in (0,1)$, let $N_0\in \mathbb{D}$ be the greatest dyadic number such that $N_0\leq 1/T_0$. Let $N\in \mathbb{D}$ and $b\in [0,1/2]$, we define the dyadic $X^{s,b}$-type normed spaces
\begin{equation*}
    \begin{aligned} X_N^b=X_N^b(\mathbb{Z}^2\times \mathbb{R})= \{f \in & L^2(\mathbb{Z}^2\times \mathbb{R}):  \mathbbm{1}_{I_N}(|(m,n)|)f=f \text{ and }  \\
    &\left\|f\right\|_{X_N^b}=N_0^{b}\|\psi_{\leq N_0}(\tau-\omega(m,n))\cdot f\|_{L_{m,n,\tau}^2}\\ 
    & \hspace{1.5cm}+\sum_{L> N_0}L^{b} \left\|\psi_L(\tau-\omega(m,n))\cdot f\right\|_{L^2_{m,n,\tau}}<\infty \}.
    \end{aligned}
\end{equation*}
where the functions $\psi_L$ and $\psi_{\leq N_0}$ are defined as in Section \ref{Sectiprel}. We will denote by $X_N$ the space $X_N^{1/2}$. Next we introduce the spaces $F^b_N$ according to $X_N^b$ uniformly on time intervals of size $N^{-1}$:
\begin{equation*}
\begin{aligned}
    F^b_N:=\{f\in C(\mathbb{R};L^2(\mathbb{T}^2)):\, P_Nf&=f, \, \, \|f\|_{F_N^b}:=\sup_{t_N\in \mathbb{R}}\|\mathcal{F}(\psi_1(N(\cdot-t_N))f)\|_{X_N^b}<\infty\}
\end{aligned}
\end{equation*}
and 
\begin{equation*}
\begin{aligned}
    \mathcal{N}_N:=\{f\in C(\mathbb{R};L^2(\mathbb{T}^2))\,:\, P_Nf&=f, \, \, \|f\|_{\mathcal{N}_N}:=\sup_{t_N\in \mathbb{R}}\||\tau+\omega(m,n)+iN|^{-1}\mathcal{F}(\psi_1(N(\cdot-t_N))f)\|_{X_N}\}.
\end{aligned}
\end{equation*}
Let $T\in (0,T_0]$ and $Y_N$ be any of the spaces $F_N^{b}$ or $\mathcal{N}_N$, we set
$$Y_N(T):=\{f\in C([0,T];L^2(\mathbb{T}^2)) : \, \|f\|_{Y_N(T)}<\infty\}$$
equipped with the norm:
\begin{equation*}
    \|f\|_{Y_N(T)}:=\inf\{\|\widetilde{f}\|_{Y_N} : \, \, \widetilde{f}\in Y_N, \, \, \widetilde{f}\equiv f \, \, \text{on }\, \, [0,T] \}.
\end{equation*}
Then for a given $s\geq 0$, we define the spaces $F^{s,b}(T)$ and $\mathcal{N}^s(T)$ from their frequency localized version $F_N^{b}(T)$ and $\mathcal{N}(T)$ by using the Littlewood-Paley decomposition as follows
\begin{equation}\label{ResoSpaF}
    F^{s,b}(T):=\{f\in C([0,T];H^{s}(\mathbb{T}^2)), \, \, \|f\|_{F^{s,b}(T)}^2=\sum_{N\in \mathbb{D}} (N^{2s}+N_0^{2s})\|P_Nf\|^2_{F^b_N(T)}<\infty\}
\end{equation}
and
\begin{equation*}
    \mathcal{N}^{s}(T):=\{f\in C([0,T];H^{s}(\mathbb{T}^2)), \, \, \|f\|_{\mathcal{N}^s(T)}^2=\sum_{N\in \mathbb{D}} (N^{2s}+N_0^{2s})\|P_Nf\|^2_{\mathcal{N}_N(T)}<\infty\}.
\end{equation*}
Next, we define the associated energy spaces $B^s(T)$ endowed with the norm
\begin{equation*}
    B^{s}(T):=\{f\in C([0,T];H^{s}(\mathbb{T}^2)), \, \, \|f\|_{B^{s}(T)}^2=\|P_{\leq N_0}f(0)\|_{H^s}^2+\sum_{N> N_0}\sup_{t_N\in [0,T]} \|P_Nf(t_N)\|^2_{H^s}<\infty\}.
\end{equation*}
In the subsequent considerations $F_N$ and $F^s(T)$ will denote the spaces above with parameter $b=1/2$. 

\subsubsection{Basic Properties}

Now we collect some basic properties of the spaces $X_N^b$ and $F_N^b(T)$. These results have been deduced in different contexts in \cite{GuoTadahiro,IonescuKeniTata,RoberTtFKP,RobertCKP,ZhangKP} for instance. 


 \begin{lemma}\label{FUNLEMM2}
Let $N\in\mathbb{D}$, $b\in(0,1/2]$, $f_N\in X_N^b$ and $h\in L^2(\mathbb{R})$ satisfying
\begin{equation*}
    |\widehat{h}(\tau)|\lesssim \langle \tau \rangle^{-4}.
\end{equation*}
Then for any $\widetilde{N_0}\in\mathbb{D}$, $\widetilde{N_0}\geq N_0$ and $t_0\in \mathbb{R}$, 
\begin{equation}\label{FEEQ2}
    \begin{aligned}
    \widetilde{N_0}^{b}\left\|\psi_{\leq \widetilde{N_0}}(\tau-\omega(m,n)) \mathcal{F}(h(\widetilde{N_0}(t-t_0))\mathcal{F}^{-1}(f_N)) \right\|_{L^2_{m,n,\tau}}\lesssim \left\|f_N\right\|_{X_N^b},
    \end{aligned}
\end{equation}
and
\begin{equation}\label{FEEQ3}
    \begin{aligned}
    \sum_{L>\widetilde{N_0}}L^{b}\left\|\psi_{L}(\tau-\omega(m,n)) \mathcal{F}(h(\widetilde{N_0}(t-t_0))\mathcal{F}^{-1}(f_N)) \right\|_{L^2_{m,n,\tau}}\lesssim \left\|f_N\right\|_{X_N^b}.
    \end{aligned}
\end{equation}
The implicit constants above are independent of $N_0\geq 1$, and in consequence of the definition of the spaces $X_N^b$.
\end{lemma}

Additionally, we require the next conclusion:
\begin{lemma}\label{FUNLEMM3} Let $N\in \mathbb{D}$, $b\in(0,1/2]$ and $I \subset \mathbb{R}$ a bounded interval. Then
\begin{equation*}
\sup_{L\in \mathbb{D}} L^{b}\|\psi_{L}(\tau-\omega(m,n))\mathcal{F}(\mathbbm{1}_{I}(t)f)\|_{L^2_{m,n,\tau}}\lesssim \|\mathcal{F}(f)\|_{X_N^b}
\end{equation*}
for all $f$ whose Fourier transform is in $X_N^b$ and the implicit constant is independent of $N_0\geq 1$.
\end{lemma}

The following lemma will be useful to obtain time factors in the energy estimates.
\begin{lemma}\label{timelemma}
Let $T\in (0,T_0)$ and $0\leq b <1/2$. Then for any $f\in F_N(T)$,
\begin{equation*}
 \|f\|_{F_N^b(T)} \lesssim T^{(1/2-b)^{-}}\|f\|_{F_N(T)}
\end{equation*}
where the implicit constant is independent of $N,N_0$ and $T$, and in consequence of the definition of the spaces $F_N^b$.
\end{lemma}
\begin{proof}
The proof follows the same arguments in \cite[Lemma 3.4]{GuoTadahiro}.
\end{proof}

We recall the embedding $F^s(T)\hookrightarrow C([0,T];H^s(\mathbb{T}^2))$, $s>0$, $T\in (0,T_0]$ established in \cite[Lemma 3.1]{IonescuKeniTata} and \cite{ZhangKP}.

\begin{lemma}\label{LemEmbedding}
Let $T\in (0,T_0]$, then
\begin{equation*}
    \sup_{t\in [0,T]}\|u(t)\|_{H^s} \lesssim \|u\|_{F^s(T)},
\end{equation*}
whenever $u\in F^s(T)$ and the implicit constant is independent of $N_0\geq 1$.
\end{lemma}

We also need the following linear estimate which is deduced in much the same way as in \cite[Proposition 3.2]{IonescuKeniTata} (see also \cite[Proposition 6.2]{ZhangKP}).

\begin{prop}
Assume that $T\in (0,T_0]$, $s\geq 0$ and $u,v \in C([0,T]; H^{\infty}(\mathbb{T}^2))$ with
\begin{equation*}
    \partial_t u +\mathcal{H}_xu-\mathcal{H}_x\partial_x^2u\pm \mathcal{H}_x\partial_y^2u=v,\hskip 15pt \text{ on } \mathbb{T}^2\times [0,T).
\end{equation*}
Then
\begin{equation}\label{FEEQ6}
    \|u\|_{F^s(T)}\lesssim \|u\|_{B^s(T)}+\|v\|_{\mathcal{N}^s(T)},
\end{equation}
where the implicit constant is independent of $N_0$, and in consequence of the definition of the spaces $F^s(T)$, $B^s(T)$ and $\mathcal{N}^s(T)$.
\end{prop}
To obtain a priori estimates for smooth solutions we need the following lemma. 
\begin{lemma}
Let $s\geq 0$, $v\in C([0,T_0];H^{\infty}(\mathbb{T}^2))$. Then the mapping $T\rightarrow \|v\|_{\mathcal{N}^s(T)}$ is increasing and continuous on $[0,T_0]$ and
\begin{equation}\label{FFEQ6}
    \lim_{T\to 0} \|v\|_{\mathcal{N}^s(T)}\rightarrow 0.
\end{equation}
\end{lemma}
\begin{proof}
The proof follows the same line of arguments in \cite[Lemma 6.3]{ZhangKP}.
\end{proof}


\subsection{\texorpdfstring{$L^2$}{} Bilinear estimates}

Next, we obtain the crucial $L^2$ bilinear estimates, which will be applied in both the short time estimates and energy estimates in the subsequent subsections. Recalling the notation introduced in \eqref{Dset}, we have:
\begin{prop}\label{lembilinEST}
Assume that $N_j,L_j\in \mathbb{D}$ and $f_j:\mathbb{Z}^2\times \mathbb{R}\rightarrow \mathbb{R}^{+}$ functions supported in $D_{N_j,L_j}$ for $j=1,2,3$.
\begin{itemize}
\item[(i)] It holds that
         \begin{equation}\label{eqbili1}
        \int_{\mathbb{Z}^2\times \mathbb{R}} (f_1\ast f_2)\cdot f_3 \lesssim N_{min}L_{min}^{1/2}\left\|f_1\right\|_{L^2}\left\|f_2\right\|_{L^2}\left\|f_3\right\|_{L^2}.
    \end{equation}
    \item[(ii)] Suppose that $N_{min}\ll N_{max}$. If $(N_j,L_j)=(N_{min},L_{max})$ for some $j\in\{1,2,3\}$, then
    \begin{equation}\label{eqbili2}
        \int_{\mathbb{Z}^2\times \mathbb{R}} (f_1\ast f_2)\cdot f_3 \lesssim N_{max}^{-1/2}N_{min}^{1/2}L_{max}^{1/2} (N_{max}^{1/2}\vee L_{min}^{1/2})\left\|f_1\right\|_{L^2}\left\|f_2\right\|_{L^2}\left\|f_3\right\|_{L^2},
    \end{equation}
    otherwise
     \begin{equation}\label{eqbili3}
        \int_{\mathbb{Z}^2\times \mathbb{R}} (f_1\ast f_2)\cdot f_3 \lesssim N_{max}^{-1/2}N_{min}^{1/2}L_{med}^{1/2} (N_{max}^{1/2}\vee L_{min}^{1/2})\left\|f_1\right\|_{L^2}\left\|f_2\right\|_{L^2}\left\|f_3\right\|_{L^2}.
    \end{equation}
    \item[(iii)] If $N_{min}\sim N_{max}$, \begin{equation}\label{eqbili3.1}
        \int_{\mathbb{Z}^2\times \mathbb{R}} (f_1\ast f_2)\cdot f_3 \lesssim L_{max}^{1/2} (N_{max}^{1/2}\vee L_{med}^{1/2})\left\|f_1\right\|_{L^2}\left\|f_2\right\|_{L^2}\left\|f_3\right\|_{L^2}.
    \end{equation}
\end{itemize}
\end{prop}
Before proving Proposition \ref{lembilinEST}, we require the following elementary result.
\begin{lemma}\label{cardinal}
Let $I,J$ be two intervals in $\mathbb{R}$, and $\varphi:J\rightarrow \mathbb{R}$ a $C^1$ function with $\inf_{x\in J}|\varphi'(x)|>0$. Suppose that $\{n\in J\cap \mathbb{Z} : \varphi(n)\in I\}\neq \emptyset$. Then
\begin{equation*}
    \#\{n\in J\cap \mathbb{Z}:  \varphi(n)\in I\} \lesssim 1+\frac{|I|}{\inf_{x\in J}|\varphi'(x)|}.
\end{equation*}
\end{lemma}
\begin{proof}[Proof of Proposition \ref{lembilinEST}]
We notice that
\begin{equation}\label{eqbili4}
    \int_{\mathbb{Z}^2\times \mathbb{R}} (f_1\ast f_2)\cdot f_3=\int_{\mathbb{Z}^2\times \mathbb{R}} (\widetilde{f}_1\ast f_3)\cdot f_2 =\int_{\mathbb{Z}^2\times \mathbb{R}} (\widetilde{f}_2\ast f_3)\cdot f_1=: \mathcal{I},
\end{equation}
where $\widetilde{f}_j(m,n,\tau)=f_j(-m,-n,-\tau)$. Let us first establish (i). In view of the above display, we can assume that $L_1=L_{min}$. Let $f^{\#}_j(m,n,\tau)=f_j(m,n,\tau+\omega(m,n))$, then $f_j^{\#}$ is supported in
\begin{equation*}
    D^{\#}_{N_j,L_j}=\left\{(m,n,\tau)\in \mathbb{R}^3: |(m,n)|\in I_{N_j} \text{ and } |\tau|\leq L_j \right\},
\end{equation*}
and $\|f_j^{\#}\|_{L^2}=\|f_j\|_{L^2}$, $j=1,2,3$, so that
\begin{equation}\label{eqbili5}
\begin{aligned}
 \mathcal{I}=\sum_{m_1,n_1,m_2,n_2}\int f^{\#}_1(m_1,n_1,\tau_1)f^{\#}_2(m_2,n_2,\tau_2)f^{\#}_3(m_1+m_2,n_1+n_2,\tau_1+\tau_2+\Omega(m_1,n_1,m_2,n_2))\, d\tau_1 d\tau_2,
\end{aligned}
\end{equation}
where $\Omega(m_1,n_1,m_2,n_2)$ is defined as in  \eqref{resonF}. Thus, by applying the Cauchy-Schwarz inequality in $\tau_2$ and then in $\tau_1$ we get
\begin{equation}\label{eqbili6}
    \begin{aligned}
    \mathcal{I}&\leq \sum_{m_1,n_1,m_2,n_2} \int |f^{\#}_1(m_1,n_1,\tau_1)|\|f^{\#}_2(m_2,n_2,\cdot)\|_{L^2_{\tau}}\|f^{\#}_3(m_1+m_2,n_1+n_2,\cdot)\|_{L^2_{\tau}} d\tau_1 \\
    &\leq L_1^{1/2} \sum_{m_1,n_1,m_2,n_2}  \|f^{\#}_1(m_1,n_1,\cdot)\|_{L^2_\tau}\|f^{\#}_2(m_2,n_2,\cdot)\|_{L^2_{\tau}}\|f^{\#}_3(m_1+m_2,n_1+n_2,\cdot)\|_{L^2_{\tau}}.
    \end{aligned}
\end{equation}
In this manner, the same procedure applied above now to the spatial variables $m_1,n_1,m_2,n_2$ on the r.h.s of \eqref{eqbili6} yields \eqref{eqbili1}.

Next, we deduce (ii). By \eqref{eqbili4}, we shall assume that $N_{min}\sim N_2$ and $L_{1}\geq L_3$, that is, $N_2\ll N_1\sim N_3$. We consider the sets:
\begin{equation}\label{eqbili6.1}
    \begin{aligned}
      A_1&=(\mathbb{Z}^4\times \mathbb{R}^2) \setminus \bigcup_{j=2}^4 A_j,\\
      A_2&=\left\{(m_1,n_1,m_2,n_2,\tau_1,\tau_2)\in \mathbb{Z}^4\times \mathbb{R}^2\, :\, m_1m_2<0 \text{ and } |m_1|>|m_2| \right\}, \\
      A_3&=\left\{(m_1,n_1,m_2,n_2,\tau_1,\tau_2)\in \mathbb{Z}^4\times\mathbb{R}^2\, :\, m_1m_2<0 \text{ and } |m_1|=|m_2| \right\}, \\
      A_4&=\left\{(m_1,n_1,m_2,n_2,\tau_1,\tau_2)\in \mathbb{Z}^4\times \mathbb{R}^2\, :\, m_1=0 \text{ or } m_2=0 \right\}.
    \end{aligned}
\end{equation}
Accordingly, we divide $\mathcal{I}$ given by \eqref{eqbili5} as
\begin{equation}\label{eqbili6.2}
    \mathcal{I}=\mathcal{I}_1+\mathcal{I}_2+\mathcal{I}_3+\mathcal{I}_4,
\end{equation}
where $\mathcal{I}_j$ corresponds to the restriction of $\mathcal{I}$ to the domain $A_j$. Now, we proceed to estimate each of the factor $\mathcal{I}_j$, $j=1,2,3,4$.
\\ \\
{\bf Estimate for $\mathcal{I}_1$}. By support considerations, it must follow that $m_2(m_1+m_2)>0$, or equivalently,  $\sign(m_2)=\sign(m_1+m_2)$. Thus, recalling \eqref{resonF}, the resonant function for this case satisfies
\begin{equation}\label{eqbili6.3}
\begin{aligned}
\Omega(m_1,n_1,m_2,n_2) 
=&\sign(m_2)(m_1^2+2m_1m_2)\mp \sign(m_2)(n_1^2+2n_1n_2)\\
&-\sign(m_1)-\sign(m_1)m_1^2\pm \sign(m_1)n_1^2. 
\end{aligned}
\end{equation}
So, we divide $A_1=A_{1,1}\cup A_{1,2}$, where $A_{1,1}$ consists of the elements in $A_1$ satisfying that $m_2>0$ and $A_{1,2}$ those for which $m_2<0$. Thus, we find
\begin{equation}\label{eqbili7}
    \big|\frac{\partial}{\partial m_2}\Omega(m_1,n_1,m_2,n_2)\big|\sim |m_1| \text{ and } \big|\frac{\partial}{\partial n_2}\Omega(m_1,n_1,m_2,n_2)\big|\sim |n_1|
\end{equation}
in each of the regions $A_{1,1}$ and $A_{1,2}$. Now, since $|(m_1,n_1)|\sim N_1$ in the support of $\mathcal{I}_1$, we further divide the region of integration according to the cases where $|\frac{\partial}{\partial m_2}\Omega(m_1,n_1,m_2,n_2)|\sim N_1$ and $|\frac{\partial}{\partial n_2}\Omega(m_1,n_1,m_2,n_2)|\sim N_1$, namely 
\begin{equation}\label{eqbili8}
    \mathcal{I}_1=\sum_{k=1}^2\int_{A_{1,k}\cap \{|m_1|\sim N_1\}} (f_1\ast f_2)\cdot f_3+\int_{A_{1,k}\cap \{|m_1|\ll N_1, \, |n_1|\sim N_1\}} (f_1\ast f_2)\cdot f_3=\mathcal{I}_{1,1}+\mathcal{I}_{1,2},
\end{equation}
To estimate $\mathcal{I}_{1,1}$, we use that $|\tau_1+\tau_2+\Omega(m_1,n_1,m_2,n_2)|\leq L_3$, \eqref{eqbili7} and Lemma \ref{cardinal}, together with the Cauchy-Schwarz inequality in the $m_2$ variable to find
\begin{equation}\label{eqbili9}
\begin{aligned}
\mathcal{I}_{1,1}&\lesssim \sum_{|m_1|\sim N_1,n_1,n_2} (1+L_{3}^{1/2}/N_1^{1/2})\int |f_1^{\#}(m_1,n_1,\tau_1)|\\
&\hspace{3cm}\times\|f^{\#}_2(m_2,n_2,\tau_2)f^{\#}_3(m_1+m_2,n_1+n_2,\tau_1+\tau_2+\Omega(m_1,n_1,m_2,n_2))\|_{L^2_{m_2}}d\tau_1 d\tau_2 \\
&\lesssim \sum_{n_2}(1+L_{3}^{1/2}/N_1^{1/2})\int \|f_1^{\#}\|_{L^2}\|f_3^{\#}\|_{L^2}\|f^{\#}_2(\cdot,n_2,\tau_2)\|_{L^2_{m}} d\tau_2 \\
& \lesssim L_2^{1/2}N_2^{1/2}(1+L_{3}^{1/2}/N_1^{1/2})\|f_1^{\#}\|_{L^2}\|f_2^{\#}\|_{L^2}\|f_3^{\#}\|_{L^2},
\end{aligned}
\end{equation}
where we have employed the Cauchy-Schwarz inequality in $m_1,n_1, \tau_1,$, and the last line is obtained by the same inequality in $n_2, \tau_2$. The estimate for $\mathcal{I}_{1,2}$ is deduced changing the roles of $m_2$ by $n_2$ in the preceding argument. This completes the study of $\mathcal{I}_1$. 
\\ \\
{\bf Estimate for $\mathcal{I}_2$}. In this case $\sign(m_1)=-\sign(m_2)$ and $\sign(m_1)=\sign(m_1+m_2)$, then
\begin{equation*}
\begin{aligned}
\Omega(m_1,n_1,m_2,n_2) 
=&\sign(m_1)(2m_1m_2+2m_2^2)\mp \sign(m_1)(2n_1n_2+2n_2^2)+\sign(m_1).
\end{aligned}
\end{equation*}
We write $A_2=A_{2,1}\cup A_{2,2}$, where $A_{2,1}=A_2 \cap \{m_1>0\}$ and $A_2\cap \{m_1<0\}$. Consequently, in each of the sets $A_{2,1}, A_{2,2}$, it holds
\begin{equation}\label{eqbili10}
    \big|\frac{\partial}{\partial m_2}\Omega(m_1,n_1,m_2,n_2)\big|\sim |2m_1+4m_2| \hspace{0.5cm} \text{ and } \hspace{0.5cm} \big|\frac{\partial}{\partial n_2}\Omega(m_1,n_1,m_2,n_2)\big|\sim |2n_1+4n_2|.
\end{equation}
Now, since $|(m_1,n_1)|\sim N_1$, $|(m_2,n_2)|\sim N_2$ with $N_2\ll N_1$, \eqref{eqbili10} establishes that in each of the regions defined by $\mathcal{I}_2$ restricted to $A_{2,1}, A_{2,2}$, either $|\frac{\partial}{\partial m_2}\Omega(m_1,n_1,m_2,n_2)|\sim N_1$ or  $|\frac{\partial}{\partial n_2}\Omega(m_1,n_1,m_2,n_2)|\sim N_1 $. In consequence, we can further divide $\mathcal{I}_2$ restricted to each $A_{2,j}$, $j=1,2$ as in \eqref{eqbili8} to apply a similar argument to \eqref{eqbili9}, which ultimately leads to the desired estimate.
\\ \\
{\bf Estimate for $\mathcal{I}_3$ and $\mathcal{I}_4$}. In these cases both regions of integration can be bounded directly by means of the Cauchy-Schwarz inequality without any further consideration on the resonant function. Indeed, in the support of $\mathcal{I}_3$, we have that $m_2=-m_1$ and so
\begin{equation}\label{eqbili11}
    \begin{aligned}
    \mathcal{I}_3&=\sum_{m_1\neq 0,n_1,n_2} \int f_1^{\#}(m_1,n_1,\tau_1)f_2^{\#}(-m_1,n_2,\tau_2)f_3^{\#}(0,n_1+n_2,\tau_1+\tau_2+\Omega(m_1,-m_1,n_1,n_2))\, d \tau_1 d\tau_2 \\
    &\lesssim  \sum_{n_1,n_2} \int \|f_1^{\#}(\cdot,n_1,\tau_1)\|_{L^2_{m}}\|f_2^{\#}(\cdot,n_2,\tau_2)\|_{L^2_m}|f_3^{\#}(0,n_1+n_2,\tau_1+\tau_2+\Omega(m_1,-m_1,n_1,n_2))|\, d \tau_1 d\tau_2 \\
    &\lesssim  \sum_{n_2} \int \|f_2^{\#}(\cdot,n_2,\tau_2)\|_{L^2_m}\|f_1^{\#}\|_{L^2}\|f_3^{\#}\|_{L^2}\, d\tau_2 \\
    &\lesssim L_2^{1/2}N_2^{1/2} \|f_2^{\#}\|_{L^2}\|f_1^{\#}\|_{L^2}\|f_3^{\#}\|_{L^2},
    \end{aligned} 
\end{equation}
where we have employed that $L^{2}(\mathbb{Z}^2)\subset L^{\infty}(\mathbb{Z}^2)$, together with consecutive applications of the Cauchy-Schwarz inequality. On the other hand, to estimate $\mathcal{I}_4$, we split the region of integration in two parts for which at least one of the variables among $m_1$ and $m_2$ is not considered in the summation. This in turn allows us to perform some simple modifications to the previous argument dealing with $\mathcal{I}_3$ to bound $\mathcal{I}_4$ by the r.h.s of \eqref{eqbili11}.  

Collecting the estimates for $\mathcal{I}_{j}$, $j=1,2,3.4$, we complete the deduction of (ii). 
\\ \\
Next, we  consider (iii). In virtue of \eqref{eqbili4}, we shall assume that $L_2=L_{min}$ and $L_3=L_{max}$. As before, we decompose $\mathcal{I}=\widetilde{\mathcal{I}}_1+\widetilde{\mathcal{I}}_2+\mathcal{\mathcal{I}}_3+\widetilde{\mathcal{I}}_4$, where $\widetilde{\mathcal{I}}_j$ corresponds to the restriction of $\mathcal{I}$ (given by \eqref{eqbili5}) to the domain $A_j$ determined by \eqref{eqbili6.1}. 

Since $N_1\sim N_2, \sim N_3$, \eqref{eqbili7} allows us to estimate $\widetilde{\mathcal{I}}_1$ exactly as in \eqref{eqbili9}. The estimate for $\widetilde{\mathcal{I}}_{j}$ is obtained without considering the resonant function as in the study of $\mathcal{I}_j$ above for each $j=3,4$.  For the sake of brevity, we omit these estimates.

In the case of $\widetilde{\mathcal{I}}^{2}$, we notice that \eqref{eqbili10} shows that $\partial_{m_2} \Omega$ and $\partial_{n_2}\Omega$ could vanish in the support of the integral. Instead, we split $A_2=A_{2,1}\cup A_{2,2}$, where $A_{2,1}=A_2\cap \{m_1>0\}$, $A_{2,1}=A_2\cap \{m_1<0\}$, we have
\begin{equation}\label{eqbili12}
    \big|\frac{\partial}{\partial m_1}\Omega(m_1,n_1,m_2,n_2)\big|\sim |m_2| \hspace{0.5cm} \text{ and } \hspace{0.5cm} \big|\frac{\partial}{\partial n_1}\Omega(m_1,n_1,m_2,n_2)\big|\sim |n_2|,
\end{equation}
in each of the regions $A_{2,1}$ and $A_{2,2}$. Thus, \eqref{eqbili12} and similar considerations in the deduction of \eqref{eqbili9} yield
\begin{equation*}
 \widetilde{\mathcal{I}}_{2} \lesssim N_{max}^{1/2}L_{med}^{1/2}(1+L_{max}^{1/2}/N_{max}^{1/2})\|f_1^{\#}\|_{L^2}\|f_2^{\#}\|_{L^2}\|f_3^{\#}\|_{L^2}.
\end{equation*}
This completes the deduction of (iii). The proof is complete.
\end{proof}

By duality and Proposition \ref{lembilinEST}, we obtain the following $L^2$ bilinear estimates.

\begin{corol}
Let $N_1,N_2,N_3,L_1,L_2,L_3 \in \mathbb{D}$ be dyadic numbers and $f_j:\mathbb{R}^3\rightarrow \mathbb{R}_{+}$ supported in $D_{N_j,L_j}$ for $j=1,2$.
\begin{itemize}
\item[(1)] It holds that
         \begin{equation}\label{eqbili13}
        \| \mathbbm{1}_{D_{N_3,L_3}}(f_1\ast f_2)\|_{L^2} \lesssim N_{min}L_{min}^{1/2}\left\|f_1\right\|_{L^2}\left\|f_2\right\|_{L^2}.
    \end{equation}
    \item[(2)] Suppose that $N_{min}\ll N_{max}$. If $(N_j,L_j)=(N_{min},L_{max})$ for some $j\in\{1,2,3\}$, then
    \begin{equation}\label{eqbili14}
        \|\mathbbm{1}_{D_{N_3,L_3}}(f_1\ast f_2)\|_{L^2} \lesssim N_{max}^{-1/2}N_{min}^{1/2}L_{max}^{1/2} (N_{max}^{1/2}\vee L_{min}^{1/2})\left\|f_1\right\|_{L^2}\left\|f_2\right\|_{L^2},
    \end{equation}
    otherwise
     \begin{equation}\label{eqbili15}
         \|\mathbbm{1}_{D_{N_3,L_3}}(f_1\ast f_2)\|_{L^2} \lesssim N_{max}^{-1/2}N_{min}^{1/2}L_{med}^{1/2} (N_{max}^{1/2}\vee L_{min}^{1/2})\left\|f_1\right\|_{L^2}\left\|f_2\right\|_{L^2}.
    \end{equation}
    \item[(3)] If $N_{min}\sim N_{max}$, 
    \begin{equation}\label{eqbili16}
        \|\mathbbm{1}_{D_{N_3,L_3}}(f_1\ast f_2)\|_{L^2}\lesssim L_{max}^{1/2} (N_{max}^{1/2}\vee L_{med}^{1/2})\left\|f_1\right\|_{L^2}\left\|f_2\right\|_{L^2}.
    \end{equation}
\end{itemize}
\end{corol}

\subsection{Short time bilinear estimates}

In this section, we derive the crucial key bilinear estimates for the equation and the difference of solutions.

\begin{prop}\label{propShortBi1}
Let $s\geq s_0\geq 1$, $T\in (0,T_0]$, then
\begin{align}
\|\partial_x(uv)\|_{\mathcal{N}^s(T)}&\lesssim T_0^{1/4}\big( \|u\|_{F^{s_0}(T)}\|v\|_{F^{s}(T)}+\|v\|_{F^{s_0}(T)}\|u\|_{F^{s}(T)} \big), \label{ShortimEst1} \\ 
\|\partial_x(uv)\|_{\mathcal{N}^0(T)}&\lesssim T_0^{1/4}\|u\|_{F^{0}(T)}\|v\|_{F^{s_0}(T)}, \label{ShortimEst2}
\end{align}
for all $u,v\in F^s(T)$ and where the implicit constants are independent of $T_0$, and the definition of the spaces involved.
\end{prop}
We split the proof of Proposition \ref{propShortBi1} in the following technical lemmas.

\begin{lemma}[$Low\times High \rightarrow High$]\label{lemmaShortBi1}
Let $N,N_1,N_2 \in \mathbb{D}$ satisfying $N_1\ll N\sim N_2$. Then,
\begin{equation*}
    \|P_N(\partial_x(u_{N_1}v_{N_2}))\|_{\mathcal{N}_N} \lesssim N_1^{1/2}\|u_{N_1}\|_{F_{N_1}}\|v_{N_2}\|_{F_{N_2}},
\end{equation*}
whenever $u_{N_1}\in F_{N_1}$ and $v_{N_2}\in F_{N_2}$. 
\end{lemma}

\begin{proof}
We use the definition of the space $\mathcal{N}_N$ to find
\begin{equation*}
\begin{aligned}
    \|P_N(\partial_x(u_{N_1}v_{N_2}))\|_{\mathcal{N}_N}\lesssim \sup_{t_N\in \mathbb{R}}\||\tau+\omega(m,n)+iN|^{-1}N \mathbbm{1}_{\{|(m,n)|\sim N\}}f_{N_1}\ast g_{N_2}\|_{X_N}
\end{aligned}
\end{equation*}
where
\begin{equation}\label{shortimeequ0} 
\begin{aligned}
f_{N_1}&=|\mathcal{F}(\psi_1(N(\cdot-t_N))u_{N_1})|, \\
g_{N_2}&=|\mathcal{F}(\widetilde{\psi}_1(N(\cdot-t_N))v_{N_2})|,
\end{aligned}
\end{equation}
with $\widetilde{\psi}_1\psi_1=\psi_1$. Now, we define 
\begin{equation}\label{shortimeequ0.1} 
\begin{aligned}
f_{N_1,(N\vee N_0)}&=\psi_{\leq (N\vee N_0)}(\tau-\omega(m,n))f_{N_1}(m,n,\tau), \\
f_{N_1,L}&=\psi_{L}(\tau-\omega(m,n))f_{N_1}(m,n,\tau),
\end{aligned}
\end{equation}
for $L> (N\vee N_0)$, and we set similarly $g_{N_2,(N\vee N_0)}$ and $g_{N_2,L}$. Therefore, from the definition of the spaces $X_N$, \eqref{eqbili14} and \eqref{eqbili15}, we find
\begin{equation}\label{shortimeequ1}
\begin{aligned}
    \|P_N(\partial_x(u_{N_1}&v_{N_2}))\|_{\mathcal{N}_N} \\
    \lesssim &\sup_{t_N\in \mathbb{R}} \sum_{L,L_1,L_2\geq (N\vee N_0)} NL^{-1/2}\|\mathbbm{1}_{D_{N,L}}\cdot(f_{N_1,L_1}\ast g_{N_2,L_2})\|_{L^2} \\
    \lesssim &\sup_{t_N\in \mathbb{R}} \sum_{L,L_1,L_2\geq (N\vee N_0), \, L_1=L_{max}}N L^{-1/2}N^{-1/2}N_1^{1/2}L_1^{1/2}L_{min}^{1/2}\|f_{N_1,L_1}\|_{L^2}\|g_{N_2,L_2}\|_{L^2} \\
    &+\sup_{t_N\in \mathbb{R}} \sum_{L,L_1,L_2\geq (N\vee N_0), \, L_1<L_{max}}NL^{-1/2}N^{-1/2}N_{1}^{1/2}L_{med}^{1/2} L_{min}^{1/2}\|f_{N_1,L_1}\|_{L^2}\|g_{N_2,L_2}\|_{L^2} \\
    \lesssim &\sup_{t_N\in \mathbb{R}} N_1^{1/2}\sum_{L\geq N} (N/L)^{1/2}  \big(\sum_{L_1\geq (N\vee N_0)}L_1^{1/2}\|f_{N_1,L_1}\|_{L^2}\big)\big(\sum_{L_2\geq (N\vee N_0)}L_2^{1/2}\|g_{N_2,L_2}\|_{L^2}\big),
\end{aligned}
\end{equation}
since $|\tau+\omega(m,n)+iN|^{-1}\leq N^{-1}$, in the first line above we have used that the sum over $N_0\leq L<(N\vee N_0)$ on the left-hand side of \eqref{shortimeequ1} can be controlled by the right-hand side of this inequality. Therefore, the above expression and Lemma \ref{FUNLEMM2} yield the deduction of the lemma.
\end{proof}

\begin{lemma}[$High\times High \rightarrow High$]\label{lemmaShortBi2}
Let $N,N_1,N_2 \in \mathbb{D}$ satisfying  $N\sim N_1\sim N_2 \gg 1$. Then,
\begin{equation*}
    \|P_N(\partial_x(u_{N_1}v_{N_2}))\|_{\mathcal{N}_N} \lesssim N^{(1/2)^{+}}\|u_{N_1}\|_{F_{N_1}}\|v_{N_2}\|_{F_{N_2}},
\end{equation*}
whenever $u_{N_1}\in F_{N_1}$ and $v_{N_2}\in F_{N_2}$. 
\end{lemma}

\begin{proof}
Following the same arguments and notation as in the proof of Lemma \ref{lemmaShortBi1}, we write
\begin{equation}\label{eqshortim1}
\begin{aligned}
  \|P_N(\partial_x(u_{N_1} v_{N_2}))\|_{\mathcal{N}_N} \lesssim &\sup_{t_N\in \mathbb{R}} \sum_{L,L_1,L_2\geq (N\vee N_0)}NL^{-1/2}\|\mathbbm{1}_{D_{N,L}}\cdot(f_{N_1,L_1}\ast g_{N_2,L_2})\|_{L^2} \\
 &= \sup_{t_N\in \mathbb{R}}\big(\sum_{\substack{L,L_1,L_2 \geq (N\vee N_0) \\  L\leq (L_1\wedge L_2)}}NL^{-1/2}(\cdots)+\sum_{\substack{L,L_1,L_2 \geq (N\vee N_0) \\ L> (L_1\wedge L_2)}} NL^{-1/2}(\cdots)\big).
\end{aligned}
\end{equation}
To estimate the first term on the right-hand side of \eqref{eqshortim1}, we employ \eqref{eqbili16} and the restrictions $(N\vee N_0)\leq L\leq (L_1\wedge L_2)$ to find
\begin{equation}\label{eqshortim1.1}
\begin{aligned}
 NL^{-1/2}\|\mathbbm{1}_{D_{N,L}}\cdot(f_{N_1,L_1}\ast g_{N_2,L_2})\|_{L^2} \lesssim  N^{1/2}(N/L)^{-1/2}(L_1^{1/2}\|f_{N_1,L_1}\|_{L^2})(L_2^{1/2}\|g_{N_2,L_2}\|_{L^2}).
    \end{aligned}
\end{equation}
Thus, we add the above expression over $L,L_1,L_2\geq (N\vee N_0)$ with $L\leq (L_1\wedge L_2)$, then we apply Lemma \ref{FUNLEMM2} to the resulting inequality to obtain the desired bound. Next, we deal with the second sum on the right-hand side of \eqref{eqshortim1}.  Interpolating \eqref{eqbili13} and \eqref{eqbili16}, it is seen
\begin{equation}\label{eqshortim2}
\begin{aligned}
    NL^{-1/2}\|&\mathbbm{1}_{D_{N,L}}\cdot(f_{N_1,L_1}\ast g_{N_2,L_2})\|_{L^2} \\
    &\lesssim  N^{2-\theta} L^{-1/2}L_{min}^{(1-\theta)/2}L_{max}^{\theta/2}L_{med}^{\theta/2}L_1^{-1/2}L_2^{-1/2}(L_1^{1/2}\|f_{N_1,L_1}\|_{L^2})(L_2^{1/2}\|g_{N_2,L_2}\|_{L^2}), 
    \end{aligned}
\end{equation}
for all $\theta\in [0,1]$ and $L>(L_1\wedge L_2)$. Under these considerations, either $L_1=L_{min}$ or $L_2=L_{min}$, which implies  $$L^{-1/2}L_{min}^{(1-\theta)/2}L_{max}^{\theta/2}L_{med}^{\theta/2}L_1^{-1/2}L_2^{-1/2} \leq L_{min}^{-\theta/2}L_{max}^{-1/2+\theta/2}L_{med}^{-1/2+\theta/2}.$$
Then, plugging the previous estimate in \eqref{eqshortim2} and recalling that $N\leq L_j,N\leq L$, we get
\begin{equation}\label{eqshortim3}
\begin{aligned}
    NL^{-1/2}\|\mathbbm{1}_{D_{N,L}}&\cdot(f_{N_1,L_1}\ast g_{N_2,L_2})\|_{L^2} \\
    &\lesssim N^{1-\theta/2}(N/L)^{1/2-\theta/2} (L_1^{1/2}\|f_{N_1,L_1}\|_{L^2})(L_2^{1/2}\|g_{N_2,L_2}\|_{L^2}). 
    \end{aligned}
\end{equation}
Therefore, taking $\theta$ sufficiently close to $1$, we sum \eqref{eqshortim3} over $L,L_1,L_2\geq (N\vee N_0)$ with $ L\geq (L_1\wedge L_2)$ and then we apply Lemma \ref{FUNLEMM2} to derive the desired estimate for the second term on the r.h.s of \eqref{eqshortim1}. 
\end{proof}

\begin{lemma}[$High\times High \rightarrow Low$]\label{lemmaShortBi3}
Let $N,N_1,N_2 \in \mathbb{D}$ satisfying $N\ll N_1\sim N_2$. Then,
\begin{equation*}
    \|P_N(\partial_x(u_{N_1}v_{N_2}))\|_{\mathcal{N}_N} \lesssim N^{(1/2)^{+}}\log(N_{max})\|u_{N_1}\|_{F_{N_1}}\|v_{N_2}\|_{F_{N_2}},
\end{equation*}
whenever $u_{N_1}\in F_{N_1}$ and $v_{N_2}\in F_{N_2}$. 
\end{lemma}

\begin{proof}
Following the same notation employed in the proof of Lemma \ref{lemmaShortBi1}, we have 
\begin{equation}\label{eqshortim4}
\begin{aligned}
  \|P_N(\partial_x(u_{N_1}& v_{N_2}))\|_{\mathcal{N}_N} \\
  \lesssim &\sup_{t_N\in \mathbb{R}} \sum_{L,L_1,L_2\geq (N\vee N_0)}NL^{-1/2}\|\mathbbm{1}_{D_{N,L}}\cdot(f_{N_1,L_1}\ast g_{N_2,L_2})\|_{L^2} \\
 =& \sup_{t_N\in \mathbb{R}}\big(\sum_{L,L_1,L_2\geq (N\vee N_0), \, L=L_{max}}NL^{-1/2}(\cdots)+\sum_{L,L_1,L_2\geq (N\vee N_0), \, L<L_{max}} NL^{-1/2}(\cdots)\big).
\end{aligned}
\end{equation}
To estimate the first term on the r.h.s of \eqref{eqshortim4}, we use \eqref{eqbili14} to deduce
\begin{equation}\label{eqshortim4.1}
    \begin{aligned}
    NL^{-1/2}\|\mathbbm{1}_{D_{N,L}}\cdot(f_{N_1,L_1}\ast g_{N_2,L_2})\|_{L^2}\lesssim N^{3/2}N_1^{-1/2}(N_1^{1/2}\vee L_{min}^{1/2})\|f_{N_1,L_1}\|_{L^2}\|g_{N_2,L_2}\|_{L^2},
    \end{aligned}
\end{equation}
where $L,L_1,L_2\geq (N\vee N_0), \, L=L_{max}$. These restrictions imply, $N_1^{-1/2}(N_1^{1/2}\vee L_{min}^{1/2})L_1^{-1/2}L_2^{-1/2}\lesssim N^{-1/2}L_{med}^{-1/2}$,  then when $L_{med} \sim L=L_{max}$, we have 
\begin{equation}\label{eqshortim5}
\begin{aligned}
N^{3/2}N_1^{-1/2}(N_1^{1/2}\vee L_{min}^{1/2})&\|f_{N_1,L_1}\|_{L^2}\|g_{N_2,L_2}\|_{L^2} \\
&\lesssim N^{1/2}(N/L)^{1/2}(L^{1/2}_1\|f_{N_1,L_1}\|_{L^2})(L^{1/2}_2\|g_{N_2,L_2}\|_{L^2}).
\end{aligned}
\end{equation}
Now, when $L_{med} \ll L$, we use instead
\begin{equation}\label{eqshortim6}
\begin{aligned}
 N^{3/2}N_1^{-1/2}(N_1^{1/2}\vee L_{min}^{1/2})&\|f_{N_1,L_1}\|_{L^2}\|g_{N_2,L_2}\|_{L^2} \lesssim N^{1/2}(L^{1/2}_1\|f_{N_1,L_1}\|_{L^2})(L^{1/2}_2\|g_{N_2,L_2}\|_{L^2}).
\end{aligned}
\end{equation}
By support considerations, it must follow that $L\sim |\Omega|\lesssim N_1^{2}$, whenever $L_{med}\ll L$. This implies that summing over $L$ in \eqref{eqshortim6} yields a factor of order $\log(N_1)$. This remark completes the estimates for the first sum in \eqref{eqshortim4}. The remaining sum in \eqref{eqshortim4} is bounded directly by \eqref{eqbili15} and arguing as above. The proof of the lemma is now complete. 
\end{proof}

\begin{lemma}[$Low\times Low \rightarrow Low$]\label{lemmaShortBi4}
Let $N,N_1,N_2 \in \mathbb{D}$ satisfying $N, N_1,N_2\ll 1$. Then,
\begin{equation*}
    \|P_N(\partial_x(u_{N_1}v_{N_2}))\|_{\mathcal{N}_N} \lesssim \|u_{N_1}\|_{F_{N_1}}\|v_{N_2}\|_{F_{N_2}},
\end{equation*}
whenever $u_{N_1}\in F_{N_1}$ and $v_{N_2}\in F_{N_2}$. 
\end{lemma}

\begin{proof}
By following similar reasoning as in the proof of Lemma \ref{lemmaShortBi1}, we notice that it is enough to establish 
\begin{equation}
    \begin{aligned}\label{eqshortime7}
    NL^{-1/2}\|\mathbbm{1}_{D_{N,L}}\cdot(f_{N_1,L_1}\ast g_{N_2,L_2})\|_{L^2}\lesssim L^{-1/2}(L^{1/2}_1\|f_{N_1,L_1}\|_{L^2})(L^{1/2}_2\|g_{N_2,L_2}\|_{L^2}),
    \end{aligned}
\end{equation}
for $L,L_1,L_2\geq N_0$. Thus, \eqref{eqshortime7} is a direct consequence of \eqref{eqbili13} and the fact that $N,N_1,N_2\lesssim 1$. 
\end{proof}

We are in conditions to prove Proposition \ref{propShortBi1}. 

\begin{proof}[Proof of Proposition \ref{propShortBi1}]  We will adapt the ideas in \cite{RibaVento} for our considerations.  We will only deduce \eqref{ShortimEst1}, since \eqref{ShortimEst2} is obtained by a similar reasoning.  For each $N_1,N_2 \in \mathbb{D}$, we choose extensions $u_{N_1}$, $v_{N_2}$ of $P_{N_1}u$ and $P_{N_2}v$ satisfying, $\|u_{N_1}\|_{F^s}\leq 2\|P_{N_1}u\|_{F_{N_1}^s(T)}$ and $\|v_{N_2}\|_{F^s}\leq 2\|P_{N_2}v\|_{F^s_{N_2}(T)}$. By the definition of the space $\mathcal{N}^s(T)$ and Minkowski inequality we have
\begin{equation*}
\begin{aligned}
\|\partial_x(uv)\|_{\mathcal{N}_N(T)}&\lesssim \sum_{j=1}^5\Big(\sum_{N\geq 1}(N^{2s}+N_0^{2s})\Big(\sum_{(N_1,N_2)\in A_j} \|P_N(\partial_x(u_{N_1}v_{N_2}))\|_{\mathcal{N}_N}\Big)^2\Big)^{1/2}=:\sum_{j=1}^5 S_j,
\end{aligned}
\end{equation*}
where
\begin{equation*}
    \begin{aligned}
    A_1&=\{(N_1,N_2)\in \mathbb{D}^2: \, N_1\ll N\sim N_2\}, \\
    A_2&=\{(N_1,N_2)\in \mathbb{D}^2: \, N_2\ll N\sim N_1\}, \\
    A_3&=\{(N_1,N_2)\in \mathbb{D}^2: \, N\sim N_1\sim N_2\gg 1\}, \\
    A_4&=\{(N_1,N_2)\in \mathbb{D}^2: \, N\ll N_1\sim N_2\}, \\
    A_5&=\{(N_1,N_2)\in \mathbb{D}^2: \, N\sim N_1\sim N_2\lesssim 1\}.
    \end{aligned}
\end{equation*}
To estimate $S_1$, we use Lemma \ref{lemmaShortBi1}, the fact that $N_1^{1/2+\epsilon}\lesssim T_0^{1/4}(N_1^{3/4+{\epsilon}}+N_0^{3/4+{\epsilon}})$ for $0<\epsilon \ll 1$ small enough and the definition of $F^s(T)$ to derive
\begin{equation*}
\begin{aligned}
S_1 &\lesssim T_0^{1/4}\Big(\sum_{N\geq 1}(N^{2s}+N_0^{2s})\Big(\sum_{N_1\ll N} N_1^{-\epsilon}(N_1^{3/4+\epsilon}+N_0^{3/4+\epsilon})\|u_{N_1}\|_{F_{N_1}}\|v_{N}\|_{F_{N}}\Big)^2\Big)^{1/2} \\
&\lesssim T_0^{1/4} \|u\|_{F^{s_0}(T)}\|v\|_{F^{s}(T)}.
\end{aligned}
\end{equation*}
The estimate for $S_2$ is obtained symmetrically as above. 
Next, we use Lemma \ref{lemmaShortBi2} and that $N^{(1/2)^{+}}\lesssim T_0^{1/4}(N^{(3/4)^{+}}+N_0^{(3/4)^{+}})$ to obtain
\begin{equation*}
\begin{aligned}
S_3 \lesssim T_0^{1/4}\Big(\sum_{N\geq 1}(N^{2s}+N_0^{2s})(N^{(3/4)^{+}}+N_0^{(3/4)^{+}})\|u_{N}\|_{F_{N}}^2\|v_{N}\|_{F_{N}}^2\Big)^{1/2} \lesssim T^{1/4}_0\|u\|_{F^{s_0}(T)}\|v\|_{F^{s}(T)}.
\end{aligned}
\end{equation*}
Let $0<\epsilon\ll 1$ fixed, then Lemma \ref{lemmaShortBi3} and the Cauchy-Schwarz inequality yield 
\begin{equation*}
\begin{aligned}
S_4 &\lesssim T^{1/4}_0\Big(\sum_{N\geq 1}N^{-\epsilon}\Big(\sum_{N\ll N_1,N_2}N_{1}^{-\epsilon/2}N_{2}^{-\epsilon/2}(N^{s}+N_0^{s})(N_{max}^{3/4+4\epsilon}+N_0^{3/4+4\epsilon})\|u_{N_1}\|_{F_{N_1}}\|v_{N_2}\|_{F_{N_2}}\Big)^2\Big)^{1/2}\\
&\lesssim T^{1/4}_0\|u\|_{F^{s_0}(T)}\|v\|_{F^{s}(T)},
\end{aligned}
\end{equation*}
which holds given that $N^{1/2+2\epsilon}\log(N_{max})\lesssim T^{1/4}_0N_{1}^{-\epsilon/2}N_{2}^{-\epsilon/2}(N_{max}^{3/4+4\epsilon}+N_0^{3/4+4\epsilon})$. The estimate for $S_5$ follows from Lemma \ref{lemmaShortBi4} and similar considerations as above. This concludes the deduction of \eqref{ShortimEst1}. 
\end{proof}

\subsection{Energy estimates.}

This section is devoted to derive the estimates required to control the $B^s$-norm of regular solutions and the difference of solutions.
\begin{lemma}\label{lemmaEner1}
Let $s_0>1/2$, then there exists $\nu>0$ small enough such that for $T\in (0,T_0]$ it holds that
\begin{equation}\label{eqenerg0}
    \left|\int_{\mathbb{T}^2\times[0,T]} u_1u_2u_3\right| \lesssim T^{\nu}N^{s_0}_{min}\prod_{j=1}^3 \|u_j\|_{F_{N_j}(T)},
\end{equation}
for each function  $u_j\in F_{N_j}(T)$, $j=1,2,3$.
\end{lemma}
\begin{proof}
In view of \eqref{eqbili4}, we will assume that $N_1 \leq N_2 \leq N_3$. Let $\widetilde{u}_j\in F_{N_j}$ be an extension of $u_j$ to $\mathbb{R}$ such that $\|\widetilde{u}_j\|_{F_{N_j}}\leq 2\|u_j\|_{F_{N_j}(T)}$ for each $j=1,2,3$. Additionally, let $h:\mathbb{R}\rightarrow \mathbb{R}$ be a smooth function supported in $[-1,1]$ such that
\begin{equation*}
    \sum_{k\in \mathbb{Z}} h^3(x-k)=1, \hspace{0.5cm} \forall x \in \mathbb{R}.
\end{equation*}
Then, we write
\begin{equation}\label{eqenerg1}
\begin{aligned}
\left|\int_{\mathbb{T}^2\times[0,T]} u_1u_2u_3 \right| &\lesssim \sum_{|k|\lesssim N_3} \int_{\mathbb{Z}^2\times \mathbb{R}} |\mathcal{F}(h(N_3t-k)\mathbbm{1}_{[0,T]}\widetilde{u}_3)| \\
&\hspace{0.3cm}\times \big(|\mathcal{F}(h(N_3t-k)\mathbbm{1}_{[0,T]}\widetilde{u}_1)|\big) \ast (|\mathcal{F}(h(N_3t-k)\mathbbm{1}_{[0,T]}\widetilde{u}_2)|\big)=: \sum_{\mathcal{A}}(\cdots)+\sum_{\mathcal{B}}(\cdots),
\end{aligned}
\end{equation}
where 
\begin{equation*}
\begin{aligned}
    \mathcal{A}&=\{k\in \mathbb{Z}\, :\, h(N_3t-k)\mathbbm{1}_{[0,T]}=h(N_3t-k)\}, \\
     \mathcal{B}&=\{k\in \mathbb{Z}\, :\, h(N_3t-k)\mathbbm{1}_{[0,T]}\neq h(N_3t-k)\, \text{ and }\, h(N_3t-k)\mathbbm{1}_{[0,T]}\neq 0 \}.
\end{aligned}
\end{equation*}
Let us estimate the sum over $\mathcal{A}$ in \eqref{eqenerg1}. Recalling the dyadic $N_0$ defining the spaces $X_N^b$, we denote by
\begin{equation*}
\begin{aligned}
f_{N_j,(N_3\vee N_0)}^k&=\psi_{\leq (N_3\vee N_0)}(\tau-\omega(m,n))|\mathcal{F}(h(N_3t-k)\widetilde{u}_j)|, \\
f_{N_j,L}^k&=\psi_{L}(\tau-\omega(m,n))|\mathcal{F}(h(N_3t-k)\widetilde{u}_j)|,
\end{aligned}
\end{equation*}
for each $j=1,2,3$, $L> (N_3\vee N_0)$ and $k\in \mathcal{A}$. Now since there are at most $N_{3}T$ integers in $\mathcal{A}$, we employ \eqref{eqbili2} and \eqref{eqbili3} when $N_{1}\ll N_{3}$, or  \eqref{eqbili3.1} if $N_{1}\sim N_3$ to deduce that
\begin{equation}\label{eqenerg2}
\begin{aligned}
\mathcal{I}_{\mathcal{A}} &\lesssim \sum_{|k|\in \mathcal{A}}\sum_{L_1,L_2,L_3\geq (N_3\vee N_0)}\int_{\mathbb{Z}^2\times \mathbb{R}} (f_{N_1,L_1}^k\ast f_{N_2,L_2}^k)\cdot f_{N_3,L_3}^k \\
&\lesssim  N_{1}^{1/2}T \,\sup_{k\in \mathcal{A}} \,\prod_{j=1}^3 \, \sum_{L_j\geq (N_3\vee N_0)}\, L_k^{1/2} \|f_{N_j,L_j}^k\|_{L^2} \lesssim N^{1/2}_1 T  \prod_{j=1}^3 \|\widetilde{u_j}\|_{F_{N_j}},
\end{aligned}
\end{equation}
where the last line above follows from \eqref{FEEQ2} and \eqref{FEEQ3}.

Next we deal with the sum over $\mathcal{B}$ in \eqref{eqenerg1}. We consider $b\in (0,1/2)$ fixed and let
\begin{equation*}
    g_{N_j,L}^k:=\psi_{L}(\tau-\omega)|\mathcal{F}(h(N_3t-k)\mathbbm{1}_{[0,T]}\widetilde{u}_j|,
\end{equation*}
for each $j=1,2,3$, $L\in \mathbb{D}$ and $k\in \mathcal{B}$. We treat first the case $N_1\ll N_3$. Since $\#\mathcal{B}\lesssim 1$, we have \begin{equation*}
    \begin{aligned}
    \mathcal{I}_{\mathcal{B}}&\lesssim \sup_{k\in \mathcal{B}} \sum_{L_1,L_2,L_3\geq 1}\int_{\mathbb{Z}^2\times \mathbb{R}} (g_{N_1,L_1}^k\ast g_{N_2,L_2}^k)\cdot g_{N_3,L_3}^k \\
    &\lesssim \sup_{k\in \mathcal{B}}\big(\sum_{L_2,L_3 \leq L_1}  \int_{\mathbb{Z}^2\times \mathbb{R}} (\cdots)+\sum_{L_1,L_2,L_3, \, L_{1}<L_{max}}  \int_{\mathbb{Z}^2\times \mathbb{R}} (\cdots) \, \big)=: \sup_{k\in \mathcal{B}} \, \big( \mathcal{I}_{\mathcal{B}}^{1,k}+\mathcal{I}_{\mathcal{B}}^{2,k} \big).
    \end{aligned}
\end{equation*}
From \eqref{eqbili2} and the fact that $N_3^{-1/2}(N_3^{1/2}\vee L_{min}^{1/2})L_{min}^{-1/2}\leq1$, we get
\begin{equation}\label{eqenerg2.1}
\begin{aligned}
    \mathcal{I}_{\mathcal{B}}^{1,k} \lesssim & \sum_{ L_2,L_3 \leq L_1} N_3^{-1/2}N_{1}^{1/2}L_{max}^{1/2}(N_{3}^{1/2}\vee L_{min}^{1/2})\|g_{N_1,L_1}^k\|_{L^2}\|g_{N_2,L_2}^k\|_{L^2}\|g_{N_3,L_3}^k\|_{L^2} \\
    \lesssim & \sum_{L_2,L_3 \leq L_1} N_{1}^{1/2}L_{max}^{1/2}L_{min}^{1/2}\|g_{N_1,L_1}^k\|_{L^2}\|g_{N_2,L_2}^k\|_{L^2}\|g_{N_3,L_3}^k\|_{L^2}.
    \end{aligned}
\end{equation}
In the regions where $L_{med} \sim L_{max}$, we use Lemmas \ref{FUNLEMM3} and \ref{timelemma}, together with the fact that $\|\widetilde{u_j}\|_{F_{N_j}}\leq 2\|u_j\|_{F_{N_j}(T)}$ to deduce 
\begin{equation}\label{eqenerg3}
    \begin{aligned}
    \sup_{k\in \mathcal{B}}\sum_{\substack{ L_2,L_3 \leq L_1, \\ L_{med}\sim L_{max}} }& N_{1}^{1/2}L_{med}^{-b}L_{max}^{1/2}L_{med}^{b}L_{min}^{1/2}\|g_{N_1,L_1}^k\|_{L^2}\|g_{N_2,L_2}^k\|_{L^2}\|g_{N_3,L_3}^k\|_{L^2} \\
 &\lesssim N_1^{1/2}\big(\sum_{L_1,L_2,L_3} L_{max}^{-b} \big) T^{(1/2-b)^{-}} \prod_{j=1}^3 \|u_j\|_{F_{N_j}(T)}.
    \end{aligned}
\end{equation}
Now, we deal with the case $L_{med}\ll L_{max}$. Interpolating the right-hand side of \eqref{eqenerg2.1} with the bound derived for $\mathcal{I}_{\mathcal{B}}^{1,k}$ using \eqref{eqbili1} instead of \eqref{eqbili2}, we find for all $\theta \in [0,1)$ that
\begin{equation}\label{eqenerg4}
    \begin{aligned}
    &\sup_{k\in \mathcal{B}}\sum_{\substack{ L_2,L_3 \leq L_1, \\ L_{med}\ll L_{max}} }N_{1}^{1-\theta/2}L_{max}^{\theta/2}L_{min}^{1/2}\|g_{N_1,L_1}^k\|_{L^2}\|g_{N_2,L_2}^k\|_{L^2}\|g_{N_3,L_3}^k\|_{L^2} \\
&=\sup_{k\in \mathcal{B}}\sum_{\substack{ L_1=L_{max}, \\ L_{med}\ll L_{max}} } N_{1}^{1-\theta/2}L_{med}^{-b}L_{max}^{-(1-\theta)/2}L_{max}^{1/2}L_{med}^{b}L_{min}^{1/2}\|g_{N_1,L_1}^k\|_{L^2}\|g_{N_2,L_2}^k\|_{L^2}\|g_{N_3,L_3}^k\|_{L^2} \\
 &\lesssim N_1^{1-\theta/2}\big(\sum_{L_1,L_2,L_2} L_{max}^{-(1-\theta)/2}L_{med}^{-b} \big) T^{(1/2-b)^{-}} \prod_{j=1}^3 \|u_j\|_{F_{N_j}(T)}.
    \end{aligned}
\end{equation}
Therefore, the estimate for $\sup_{k\in \mathcal{B}}\, \mathcal{I}_{\mathcal{B}}^{1,k}$ is now a consequence of \eqref{eqenerg3} and \eqref{eqenerg4}. On the other hand, we can implement \eqref{eqbili3} and the same ideas dealing with \eqref{eqenerg3} to derive the following bound 
\begin{equation*}
    \begin{aligned}
    \sup_{k\in \mathcal{B}}\, \mathcal{I}_{\mathcal{B}}^{2,k} &\lesssim \sup_{k\in \mathcal{B}}\sum_{\substack{ L_1,L_2,L_3, L_1< L_{max}} } N_{1}^{1/2}L_{max}^{-b}L_{max}^{b}L_{med}^{1/2}L_{min}^{1/2}\|g_{N_1,L_1}^k\|_{L^2}\|g_{N_2,L_2}^k\|_{L^2}\|g_{N_3,L_3}^k\|_{L^2} \\
 &\lesssim N_1^{1/2}\big(\sum_{L_1,L_2,L_2} L_{max}^{-b} \big) T^{(1/2-b)^{-}} \prod_{j=1}^3 \|u_j\|_{F_{N_j}(T)}.
    \end{aligned}
\end{equation*}
This completes the analysis of $\mathcal{I}_{\mathcal{B}}$ in the region $N_1 \ll N_3$. Next we treat the case $N_1 \sim N_2$. Interpolating \eqref{eqbili1} and \eqref{eqbili3.1}, we obtain for all $\theta \in [0,1]$ that
\begin{equation}\begin{aligned}\label{eqenerg5}
   \mathcal{I}_{\mathcal{B}}
    &\lesssim \sup_{k\in \mathcal{B}} \,  \sum_{L_{1},L_2,L_3} \big(L_{max}^{-\theta/2}(N_{1}^{1/2}\vee L_{med}^{1/2})^{1-\theta}N_1^{\theta}L_{med}^{-1/2}L_{min}^{\theta/2-b}\big)L_{max}^{1/2}L_{med}^{1/2}L_{min}^{b}\|g_{N_1,L_1}^k\|_{L^2}\|g_{N_2,L_2}^k\|_{L^2}\|g_{N_3,L_3}^k\|_{L^2} \\
    &\lesssim \sup_{k\in \mathcal{B}} \, N_1^{1/2+\theta/2} \sum_{L_{1},L_2,L_3} L_{max}^{-\theta/2}\big(L_{med}^{-\theta/2}L_{min}^{\theta/2-b}\big)L_{max}^{1/2}L_{med}^{1/2}L_{min}^{b}\|g_{N_1,L_1}^k\|_{L^2}\|g_{N_2,L_2}^k\|_{L^2}\|g_{N_3,L_3}^k\|_{L^2}.
\end{aligned}
\end{equation}
Therefore, taking $0<\theta\ll 1$ and employing a similar reasoning to \eqref{eqenerg4}, the estimate for $\mathcal{I}_{\mathcal{B}}$ when $N_1 \sim N_3$ is a consequence of \eqref{eqenerg5}. Gathering all the previous results, by setting $\nu=1/2-b$ we obtain \eqref{eqenerg0}.
\end{proof}

\begin{lemma}\label{lemmaEner2}
Assume that $s_0 >3/2$, $N_1\ll N$, then there exists $\nu>0$ such that for $T\in(0,T_0]$,
\begin{equation*}
    \left|\int_{\mathbb{T}^2\times \mathbb{R}} P_{N}(\partial_{x}u P_{N_1}v)P_N u \,dx dy dt\right|\lesssim_{s_0} T^{\mu} N_{min}^{s_0}\|v\|_{F_{N_1}(T)}\sum_{N_2\sim N}\|u\|_{F_{N_2}(T)}^2,
\end{equation*}
whenever $v\in F_{N_1}(T)$ and $u\in F_{N_2}(T)$.
\end{lemma}
\begin{proof}
We divide the integral expression in the following manner
\begin{equation}\label{eqenerg5.1}
\begin{aligned}
\int_{\mathbb{T}^2\times \mathbb{R}}& P_{N}(\partial_{x}u P_{N_1}v)P_N u \,dx dy dt \\
&=\int_{\mathbb{T}^2\times \mathbb{R}} \partial_{x}P_{N}u P_{N_1}vP_N u +\int_{\mathbb{T}^2\times \mathbb{R}} P_{N}(\partial_{x}u P_{N_1}v)P_N u-\partial_{x}P_{N}u P_{N_1}vP_N u \\
&=\mathcal{I}+\mathcal{II}.
\end{aligned}
\end{equation}
Integrating by parts and using \eqref{eqenerg0}, the first term on the right-hand side of the above expression satisfies
\begin{equation}\label{eqenerg5.1.1}
|\mathcal{I}|\lesssim T^{\nu}N_{1}^{(3/2)^{+}}\|v\|_{F_{N_1}(T)}\|u\|_{F_{N}(T)}^2. 
\end{equation}
The estimate for $\mathcal{II}$ is deduced arguing as in \cite[Lemma 6.1]{IonescuKeniTata} (see equation (6.10)) and following the same ideas leading to \eqref{eqenerg0}. For the sake of brevity, we omit its proof. 
\end{proof}

\begin{prop}\label{propEner1}
Let $T\in (0,T_0]$ and  $s \geq s_0> 3/2$. Then for any $u\in C([0,T];H^{\infty}(\mathbb{T}^2))$ solution of the IVP \eqref{EQBO} on $[0,T]$,
\begin{equation}\label{EnergyInequ}
    \|u\|_{B^s(T)}^2 \lesssim \|u_0\|_{H^s}^2+T^{\nu}\|u\|_{F^{s_0}(T)}\|u\|_{F^s(T)}^2,
\end{equation}
where the implicit constant above is independent of the definition of the spaces involved.
\end{prop}

\begin{proof}
According to the definition of the spaces $B^s(T)$ and the fact that $u$ solves the IVP \eqref{EQBO}, it is enough to derive a bound for the sum over $N\geq N_0$ of the following expression
\begin{equation}\label{eqenerg6}
    N^{2s}\|P_N u(t_N)\|_{L^2}^2 \lesssim N^{2s}\|P_N u_0\|_{L^2}^2+N^{2s}\left|\int_{\mathbb{T}^2\times [0,T]} P_{N}(\partial_{x}u u)P_N u  \,dx dy dt\right|.
\end{equation}
Now we split the estimate for the integral term above according to the iterations: $High\times Low \rightarrow High$, 
\begin{equation}\label{eqenerg7}
 \int_{\mathbb{T}^2\times [0,T]} P_N( \partial_{x}u P_{N_1}u) P_N u\, dx dy dt, \text{ where } N_1\ll  N,
\end{equation}
$Low\times High \rightarrow High$, 
\begin{equation}\label{eqenerg8}
 \int_{\mathbb{T}^2\times [0,T]}  P_N( \partial_{x}P_{N_1}u P_{N_2}u) P_N u\, dx dy dt, \text{ where } N_1\ll N_2\sim N,
\end{equation}
$High\times High \rightarrow High$, 
\begin{equation}\label{eqenerg9}
 \int_{\mathbb{T}^2\times [0,T]} P_N( \partial_{x}P_{N_1}u P_{N_2}u) P_N u\, dx dy dt, \text{ where } N\sim N_1\sim N_2,
\end{equation}
and 
$High\times High \rightarrow Low$, 
\begin{equation}\label{eqenerg10}
 \int_{\mathbb{T}^2\times [0,T]} P_N( \partial_{x}P_{N_1}u P_{N_2}u) P_N u\,  dx dy dt, \text{ where } N\ll N_1\sim N_2.
 \end{equation}
In view of Lemma \ref{lemmaEner2}, the $High\times Low \rightarrow High$ iteration satisfies
\begin{equation}\label{eqenerg11} 
    \eqref{eqenerg7} \lesssim T^{\nu} N_{1}^{(3/2)^{+}}\|P_{N_1}u\|_{F_{N_1}(T)}\sum_{N_2\sim N}\|P_{N_2}u\|_{F_{N_2}(T)}^2.
\end{equation}
Summing the above expression over $N$ and $N_1\ll N$, we can modify the power of $N_1^{(3/2)^{+}}$ by an arbitrary small factor to apply the Cauchy-Schwarz inequality in the sum over $N_1$. Next, we apply the same inequality for the sum over $N_1$, obtaining  \eqref{EnergyInequ}. Now, recalling \eqref{eqenerg5.1.1} in the proof of  Lemma \ref{lemmaEner2}, we notice that the  $Low\times High \rightarrow High$ iteration satisfies the same estimate in \eqref{eqenerg11}. 
\\ \\
Next we apply  \eqref{eqenerg0} to control the $High\times High \rightarrow High$ iterations as follows 
\begin{equation}\label{eqenerg12}
    \eqref{eqenerg9} \lesssim  T^{\nu}N^{(3/2)^{+}}\|P_Nu\|_{F_{N}(T)}\|P_{N_1}u\|_{F_{N_1}(T)}\|P_{N_2}u\|_{F_{N_2}(T)}.
\end{equation}
Since $N\sim N_1 \sim N_2$, we can increase the power in $N^{(3/2)^{+}}$ by a small factor to apply  the Cauchy-Schwarz inequality separately in each of the sums over $N,N_1,N_2$ to derive the desired result. The estimate for $High\times High \rightarrow Low$ is obtained by \eqref{eqenerg0} and a similar reasoning to the iteration $High\times High \rightarrow High$. This completes the estimate for the r.h.s of \eqref{eqenerg6} and in turn the deduction of \eqref{EnergyInequ}.
\end{proof}

We also require the following result to deal with the difference of solutions.
\begin{prop}
Let $T\in (0,T_0]$, $s\geq s_0>3/2$. Consider $u,v\in C([0,T];H^{\infty}(\mathbb{T}^2))$ solutions of the IVP \eqref{EQBO} with initial data $u_0,v_0\in H^{\infty}(\mathbb{T}^2)$ respectively, then
\begin{equation}\label{eqenerg13.1}
    \|u-v\|^2_{B^0(T)} \lesssim \|u_0-v_0\|_{L^2}^2+T^{\nu}\big( \|u-v\|_{F^{s_0}(T)}\|u-v\|_{F^{0}(T)}^2+\|v\|_{F^{s_0}(T)}\|u-v\|_{F^{0}(T)}^2\big),
\end{equation}
and
\begin{equation}\label{eqenerg13.2}
\begin{aligned}
  \|u-v\|^2_{B^s(T)} \lesssim \|u_0-v_0\|_{H^s}^2+ T^{\nu}\big( \|v\|_{F^{s_0}(T)}\|u-v&\|_{F^{s}(T)}^2+\|u-v\|_{F^{s_0}(T)}\|u-v\|_{F^{s}(T)}\|v\|_{F^{s}(T)} \\
  &+\|v\|_{F^{(s+3/2)^{+}}(T)}\|u-v\|_{F^{s}(T)}\|u-v\|_{F^{0}(T)}\big),
\end{aligned}
\end{equation}
where the implicit constants are independent of $T_0$ and the spaces involved.
\end{prop}
\begin{proof}
We shall argue as in the proof of Proposition \ref{propEner1}. Letting $w=u-v$, we find that $w$ solves the equation:
\begin{equation}\label{DIFEQBO}
  \partial_t w +\mathcal{H}_xw-\mathcal{H}_x\partial_x^2w\pm \mathcal{H}_x\partial_y^2w+\frac{1}{2}\partial_{x}((u+v)w)=0,
\end{equation}
with initial condition $  w(x,0)=u_0-v_0$. Let $\widetilde{s}\in \{0,s\}$, the definition of the $B^{\widetilde{s}}(T)$-norm and the fact that $w$ solves \eqref{DIFEQBO} yield
\begin{equation}\label{eqenerg13.3}
\begin{aligned}
\|w\|_{B^{\widetilde{s}}(T)}^2 &\lesssim \|P_{\leq N_0} w(0)\|_{H^{\widetilde{s}}}^2+\sum_{N> N_0} \sup_{t_N}N^{\widetilde{s}}\|P_N w(t_N)\|_{L^2}^2 \\
&\lesssim \|w(0)\|_{H^{\widetilde{s}}}+\sum_{N> N_0}N^{2\widetilde{s}}\left|\int_{\mathbb{T}^2\times [0,T]} P_{N}(w\partial_{x}w+v \partial_x w+\partial_x v w)P_N w \,dx dy dt\right|.
\end{aligned}
\end{equation}
Then, we are reduced to estimate the integral term on the right-hand side of the last inequality. Arguing as in the proof of Proposition \ref{propEner1}, applying Lemmas \ref{lemmaEner1} and \ref{lemmaEner2}, we obtain
\begin{equation}\label{eqenerg13.4}
 \sum_{N> N_0}   \left|\int_{\mathbb{T}^2\times [0,T]} P_{N}(w\partial_{x}w+v \partial_x w)P_N w \,dx dy dt\right| \lesssim T^{\nu}\big( \|w\|_{F^{(3/2)^{+}}(T)}\|w\|_{F^{0}(T)}^2+\|v\|_{F^{(3/2)^{+}}(T)}\|w\|_{F^{0}(T)}^2\big)
\end{equation}
and 
\begin{equation}\label{eqenerg13.5}
\begin{aligned}
\sum_{N>N_0}&N^{2s} \left|\int_{\mathbb{T}^2\times [0,T]} P_{N}(w\partial_{x}w+v \partial_x w)P_N w \,dx dy dt\right| \\
&\hspace{4cm}\lesssim T^{\nu}\big( \|v\|_{F^{(3/2)^{+}}(T)}\|w\|_{F^{s}(T)}^2+\|w\|_{F^{(3/2)^{+}}(T)}\|w\|_{F^{s}(T)}\|v\|_{F^{s}(T)}\big),
\end{aligned}
\end{equation}
where we emphasize that the last term on the right-hand side of \eqref{eqenerg13.5}  appears from the estimate dealing with the $Low\times High \rightarrow High$ iteration and Lemma \ref{lemmaEner1}, since in this case
\begin{equation*}
    N^{2s}\left|\int_{\mathbb{T}^2\times[0,T]}P_N(\partial_x P_{N_1}wP_{N_2}v) P_Nw \, dx dy dt\right|\lesssim T^{\nu}N_1^{(3/2)^{+}}N^{2s}\|P_{N_1}w\|_{F_{N_1}(T)}\|P_{N_2}v\|_{F_{N_2}(T)}\|P_{N}w\|_{F_{N}(T)},
\end{equation*}
with $N_1\ll N\sim N_2$. It remains to control the term $v\partial_xw$ in the integral in \eqref{eqenerg13.3}. We divide our considerations as in the proof of Proposition \ref{propEner1} according to the iterations: $High\times Low \rightarrow High$, $Low \times High\rightarrow High$, $High\times High \rightarrow High$ and $High \times High \rightarrow Low$. Notice that in this case we cannot apply Lemma \ref{lemmaEner2} to control the $High\times Low\rightarrow High$ iteration. We use instead Lemma \ref{lemmaEner1} to find for $N_1 \ll N$ that
\begin{equation}\label{eqenerg13.6}
\begin{aligned}
    N^{2\widetilde{s}}\Big|\int_{\mathbb{T}^2\times[0,T]}P_N(\partial_x v & P_{N_1}w)P_N w  \, dx dydt\Big|\\
    &\lesssim \sum_{N_2\sim  N}T^{\nu} N_1^{(1/2)^{+}}N_2N^{2\widetilde{s}}\|P_{N_1}w\|_{F_{N_1}(T)}\|P_{N_2}v\|_{F_{N_2}(T)}\|P_N w\|_{F_{N}(T)}.
\end{aligned}
\end{equation}
Summing \eqref{eqenerg13.6} over $N$ and $N_1\ll N$, we use that $N_1^{(1/2)^{+}}N_2N^{2\widetilde{s}}\lesssim N_2^{(3/2)^{+}+\widetilde{s}}N^{\widetilde{s}}N_{1}^{-\epsilon}$ for $0<\epsilon\ll 1$ to apply the Cauchy-Schwarz inequality on the sum over $N_1$ and then on $N$ to control the resulting expression by the r.h.s of \eqref{eqenerg13.1} if $\widetilde{s}=0$, or by the last term on the r.h.s of \eqref{eqenerg13.2} if $\widetilde{s}=s$.

The remaining iterations are treated as in the proof of Proposition \ref{propEner1}, and their resulting bounds are the same displayed on the right-hand sides of \eqref{eqenerg13.4} if $\widetilde{s}=0$ and \eqref{eqenerg13.5} if $\widetilde{s}=s$ respectively. The proof of the proposition is now complete.

\end{proof}


\subsection{Proof of Theorem \ref{LocalwellTorus} }

We follow similar considerations as in \cite{IonescuKeniTata,ZhangKP} to prove Theorem \ref{LocalwellTorus}. We begin by recalling the local well-posedness result for smooth initial data, which can be deduced as in \cite[Theorem 2.1]{IoNu}. 
\begin{theorem}\label{existsmoothsol}
Let $u_0\in H^{\infty}(\mathbb{T}^2)$. Then there exist $T>0$ and a unique $u\in C([0,T];H^3(\mathbb{T}^2))$ solution of the IVP \eqref{EQBO}. Moreover, the existence time $T=T(\|u_0\|_{H^3})$ is a non-increasing function of $\|u_0\|_{H^3}$ and the flow-map is continuous. 
\end{theorem}
We divide the proof of Theorem \ref{LocalwellTorus} in the following main parts.

\subsubsection{A priori estimates for smooth solutions}

\begin{prop}\label{PropaprioE}
 Let $s>3/2$ and $R>0$. Then there exists $T=T(R)>0$, such that for all $u_0\in H^{\infty}(\mathbb{T}^2)$ satisfying  $\|u_0\|_{H^s}\leq R$, then the corresponding solution $u$ of the IVP \eqref{EQBO} given by Theorem \ref{existsmoothsol} is in the space $C([0,T];H^{\infty}(\mathbb{T}^2))$ and satisfies
 \begin{equation}\label{eqTmain0}
    \sup_{t\in [0,T]} \|u(t)\|_{H^s} \lesssim \|u_0\|_{H^s}. 
 \end{equation}
\end{prop}
\begin{proof}
We consider $s>3/2$ fixed and $u_0$ as in the statement of the proposition. In virtue of Theorem \ref{existsmoothsol}, there exist $T'=T'(\|u_0\|_{H^{3}})\in (0,1]$ and $u\in C([0,T'];H^{\infty}(\mathbb{T}^2))$ solution of the IVP \eqref{EQBO} with initial data $u_0$. Then for a given $T_0\in (0,1]$ to be chosen later, we collect the estimates \eqref{FEEQ6}, \eqref{ShortimEst1} and \eqref{EnergyInequ} to find for each $s_1\geq s \geq s_0 >3/2$ that
\begin{equation}\label{eqTmain1}
    \left\{\begin{aligned}
    &\|u\|_{F^{s_1}(T)}\lesssim \|u\|_{B^{s_1}(T)}+\|\partial_x(u^2)\|_{\mathcal{N}^{s_1}(T)} \\
    &\|\partial_x(u^2)\|_{\mathcal{N}^{s_1}(T)}\lesssim T_0^{1/4}\|u\|_{F^{s_0}(T)}\|u\|_{F^{s_1}(T)},  \\
    &\|u\|_{B^{s_1}(T)}\lesssim \|u_0\|_{H^{s_1}}+ T_0^{\nu}\|u\|_{F^{s_0}(T)}^{1/2}\|u\|_{F^{s_1}(T)},
    \end{aligned} \right.
\end{equation}
where $0<T\leq(T'\wedge T_0)$. We emphasize that our arguments indicate that the implicit constants in \eqref{eqTmain1} and $\nu>0$ are independent of $T_0\in (0,1]$ and in consequence of the definition of the spaces involved (which depend on $N_0\leq T_0^{-1}$). 
Letting $s_1=s=s_0$ and $\Gamma_s(T)=\|u\|_{B^s(T)}+\|\partial_x(u^2)\|_{\mathcal{N}^s(T)}$, \eqref{eqTmain1} yields 
\begin{equation}\label{eqTmain2}
    \Gamma_s(T) \lesssim \|u_0\|_{H^s}+T_0^{1/4}\Gamma_s(T)^2+T_0^{\nu}\, \Gamma_s(T)^{3/2}.
\end{equation}
Considering now $s_1=3$, $s_0=s$ in \eqref{eqTmain1}, we also find
\begin{equation}\label{eqTmain3}
    \|u\|_{F^{3}(T)}
    \lesssim \|u_0\|_{H^3}+T_0^{1/4}\Gamma_s(T)\|u\|_{F^{3}(T)}+T_0^{\nu}\Gamma_s(T)^{1/2}\|u\|_{F^{3}(T)}.
\end{equation}
Since the mapping $T\mapsto \|u\|_{B^s(T)}$ is decreasing and continuous with $\lim_{T\to 0}\|u\|_{B^s(T)}\lesssim \|u\|_{H^s}$, from \eqref{FFEQ6} it follows that
 \begin{equation}\label{eqTmain4}
  \lim_{T\to 0}\Gamma_s(T) \lesssim \|u_0\|_{H^s},  
 \end{equation}
where the implicit constant is independent of $T_0$ and the definition of the spaces involved. Thus, we can choose $T_0=T_0(R)>0$ sufficiently small, such that $T_0^{1/4}R+T_0^{\nu}R^{1/2} \ll 1$ according to the constants in \eqref{eqTmain2} and \eqref{eqTmain4}. Then, for this time and the associated spaces $F^s(T), \mathcal{N}^s(T),B^s(T)$, we can apply a bootstrap argument relaying on \eqref{eqTmain2}, \eqref{eqTmain4} and the continuity of $\Gamma_s(T)$, to obtain $\Gamma_s(T)\lesssim \|u_0\|_{H^s}$, for any $0<T\leq T_0$. Consequently, Lemma \ref{LemEmbedding} reveals
\begin{equation*}
    \sup_{t\in[0,(T'\wedge T_0)]}\|u(t)\|_{H^s} \lesssim \|u_0\|_{H^s}.
\end{equation*}
Therefore, up to choosing $T_0$ smaller at the beginning of the argument, from \eqref{eqTmain3}  we infer
\begin{equation*}
    \sup_{t\in[0,(T'\wedge T_0)]}\|u(t)\|_{H^3} \lesssim \|u_0\|_{H^3}.
\end{equation*}
In this manner, the preceding result and Theorem \ref{existsmoothsol} allow us to extend $u$, if necessary,  to the whole interval $[0,T_0(R)]$. This completes the proof of the proposition.
\end{proof}

\subsubsection{\texorpdfstring{$L^2$}{}-Lipschitz bounds and uniqueness}

Let $u,v \in C([0,T'];H^s(\mathbb{T}^2))$ be two solutions of the IVP \eqref{EQBO} defined on $[0,T']$ with initial data $u_0,v_0 \in H^s(\mathbb{T}^2)$ such that $u,v \in F^s(T,T')\cap \mathcal{N}^s(T,T')$, where we denote by $F^s(T,T')$ and $B^s(T,T')$ the spaces defined at time $ T'$ and $0<T\leq T'$. Notice that this implies that $u, v \in  F^s(T,T_0)\cap \mathcal{N}^s(T,T_0)$, whenever $0<T\leq T_0\leq T'$. We collect \eqref{FEEQ6}, \eqref{ShortimEst2} and \eqref{eqenerg13.1} to get
\begin{equation}\label{eqTmain5}
    \left\{\begin{aligned}
    &\|u-v\|_{F^{0}(T,T_0)}\lesssim \|u-v\|_{B^{0}(T,T_0)}+\|\partial_x((u+v)(u-v))\|_{\mathcal{N}^{0}(T,T_0)}, \\
    &\|\partial_x((u+v)(u-v))\|_{\mathcal{N}^{0}(T,T_0)}\lesssim T_0^{1/4}(\|u\|_{F^s(T,T_0)}+\|v\|_{F^{s}(T,T_0)})\|u-v\|_{F^{0}(T,T_0)},  \\
    &\|u-v\|_{B^{0}(T,T_0)}\lesssim \|u_0-v_0\|_{L^2}+ T_0^{\nu}(\|u\|_{F^{s}(T,T_0)}+\|v\|_{F^{s}(T,T_0)})^{1/2}\|u-v\|_{F^{0}(T,T_0)},
    \end{aligned} \right.
\end{equation}
where the implicit constants above are independent of the definition of the spaces. Let $R>0$, satisfying $ \sup_{t\in [0,T']}(\|u(t)\|_{H^s}+\|v(t)\|_{H^s})\leq R$. Following a similar reasoning as in the proof of Proposition \ref{PropaprioE}, there exists a time $T_0=T_0(R)>0$ sufficiently small, for which $T_0^{1/4}R+T_0^{\nu}R^{1/2}\ll 1$ with respect to the constants in \eqref{eqTmain5} and $\|u\|_{F^s(T,T_0)},\|v\|_{F^s(T,T_0)}\lesssim R$. Consequently, \eqref{eqTmain5} and Lemma \ref{LemEmbedding} yield
\begin{equation*}
    \sup_{t\in [0,T]}\|u(t)-v(t)\|_{L^2}\lesssim\|u-v\|_{F^{0}(T,T_0)}\lesssim \|u_0-v_0\|_{L^2},
\end{equation*}
for any $0<T\leq T_0$. Thus, if $u_0=v_0$, the last equation reveals that $u=v$ on $[0,T_0]$. Since $T_0$ depends on $R=R(\sup_{t\in [0,T']}(\|u(t)\|_{H^s}+\|v(t)\|_{H^s}))$, we can employ the same spaces to repeat this procedure a finite number of times obtaining uniqueness in the whole interval $[0,T']$. 

\subsubsection{Existence and continuity of the flow-map}

Let $R>0$ and $3/2<s<3$ fixed. For a given $u_0\in H^{s}(\mathbb{T}^2)$ with $\|u_0\|_{H^s}\leq R$, we consider a sequence $(u_{0,n})\subset H^{\infty}(\mathbb{T}^2)$ converging to $u_0$ in $H^{s}(\mathbb{T}^2)$, such that $\|u_{0,n}\|_{H^s}\leq R$. We denote by $\Phi(u_{0,n})$ the solution of the IVP \eqref{EQBO} with initial data $u_{0,n}$ determined by Theorem \ref{existsmoothsol}. Therefore, according to Proposition \ref{PropaprioE}, there exists $T'=T'(R)>0$, such that $\Phi(u_{0,n})\in C([0,T'];H^{\infty}(\mathbb{T}^2))$ and \eqref{eqTmain0} holds. We shall prove that $(\Phi(u_{0,n}))$ defines a Cauchy sequence in $C([0,T];H^s(\mathbb{T}^2))$ for some $0<T \leq T'$.  To this aim, we will proceed as in \cite{IonescuKeniTata,ZhangKP}. 
\\ \\
For a fixed $M>0$ and $n,l\geq 0$ integers, we have
\begin{equation}\label{eqTmain6}
\begin{aligned}
\sup_{t\in [0,T]}\|\Phi(u_{0,n})(t)-\Phi(u_{0,l})(t)\|_{H^s} \leq  \sup_{t\in[0,T]}&\big(\|\Phi(u_{0,n})(t)-\Phi(P_{\leq M}u_{0,n})(t)\|_{H^s}\\ 
&+\|\Phi(P_{\leq M}u_{0,n})(t)-\Phi(P_{\leq M}u_{0,l})(t)\|_{H^s} \\
&+\|\Phi(u_{0,l})(t)-\Phi(P_{\leq M}u_{0,l})(t)\|_{H^s}\big),
\end{aligned}
\end{equation}
for all $0<T<T'$. Using Sobolev embedding and \eqref{eqTmain0}, we get
\begin{equation}
\begin{aligned}
\|\partial_x\big(\Phi(P_{\leq M}u_{0,n})+\Phi(P_{\leq M}u_{0,l})\big)(t)\|_{L^{\infty}_x} &\lesssim  \|\Phi(P_{\leq M}u_{0,n})(t)\|_{H^3}+\|\Phi(P_{\leq M}u_{0,l})(t)\|_{H^3} \\ &\lesssim  \|P_{\leq M}u_{0,n}\|_{H^3}+\|P_{\leq M}u_{0,l}\|_{H^3}.
\end{aligned}
\end{equation}
Then, the standard energy method and the above inequality show that the second term on the right-hand side of \eqref{eqTmain6} is controlled as follows
\begin{equation}\label{eqTmain6.1}
    \sup_{t\in[0,T]} \|\Phi(P_{\leq M}u_{0,n})(t)-\Phi(P_{\leq M}u_{0,l})(t)\|_{H^s}\leq C(M) \|u_{0,n}-v_{0,l}\|_{H^s},
\end{equation}
for each $0<T<T'$ and some constant $C(M)>0$ depending on $M$. Therefore, it remains to estimate the first and last term in \eqref{eqTmain6}. By symmetry of the argument, we will restrict our considerations to study the former term. To simplify notation, let us denote by $u:=\Phi(u_{0,n})$, $v:=\Phi(P_{\leq M}u_{0,n})$ and $w=u-v$, then taking $T_0\in (0,T']$, we gather \eqref{FEEQ6}, \eqref{ShortimEst1} and \eqref{eqenerg13.2} to find
\begin{equation}\label{eqTmain7}
    \left\{\begin{aligned}
    &\|w\|_{F^{s}(T)}\lesssim \|w\|_{B^{s}(T)}+\|\partial_x((u+v)w)\|_{\mathcal{N}^{s}(T)}, \\
    &\|\partial_x((u+v)w)\|_{\mathcal{N}^{s}(T)}\lesssim T_0^{1/4}(\|u+v\|_{F^{s}(T)}\|w\|_{F^s(T)}),  \\
    &\|w\|_{B^{s}(T)}\lesssim \|u_{0,n}-P_{\leq M}u_{0,n}\|_{H^s}+ T_0^{\nu}(\|v\|_{F^{s}(T)}^{1/2}\|w\|_{F^s(T)}+\|v\|_{F^{s'}(T)}^{1/2}\|w\|_{F^{s}(T)}^{1/2}\|w\|_{F^{0}(T)}^{1/2}),
    \end{aligned} \right.
\end{equation}
for all $0<T\leq T_0$, and where $s+3/2<s'<2s$ is fixed. The above set of inequalities reveal
\begin{equation}\label{eqTmain8}
\begin{aligned}
\|w\|_{F^{s}(T)}\lesssim \|u_{0,n}-P_{\leq M}u_{0,n}\|_{H^s}+(T_0^{1/2}(\|u\|_{F^s(T)}&+\|v\|_{F^s(T)})+T_0^{\nu}\|v\|_{F^s(T)}^{1/2})\|w\|_{F^s(T)} \\
&+T_0^{\nu}\|w\|_{F^s(T)}^{1/2}\|v\|_{F^{s'}(T)}^{1/2} \|w\|_{F^0(T)}^{1/2}.
\end{aligned}
\end{equation}
Repeating the arguments in the proof of Proposition \ref{PropaprioE}, using \eqref{eqTmain1} with $s_1=s'$ and $s_0=s$, we choose $T_0=T_0(R)<T'$ small so that
\begin{equation*}
    \|v\|_{F^{s'}(T)} \lesssim \|P_{\leq M}u_{0,n}\|_{H^{s'}}, \hspace{0.5cm} 0<T\leq T_0,
\end{equation*}
and such that, employing \eqref{eqTmain5} and similar considerations as in the uniqueness proof above, 
\begin{equation*}
    \|w\|_{F^{0}(T)} \lesssim \|u_{0,n}-P_{\leq M}u_{0,n}\|_{L^2}, \hspace{0.5cm} 0<T\leq T_0.
\end{equation*}
Furthermore, we can choose $T_0$ smaller, if necessary, to assure that $T_0^{1/2}R+T_0^{\nu}R^{1/2}\ll 1$ with respect to the implicit constant in \eqref{eqTmain8}, and such that $\|u\|_{F^s(T)},\|v\|_{F^s(T)}\lesssim R$. Then gathering these estimates in \eqref{eqTmain8}, we get
\begin{equation*}
\begin{aligned}
 \|w\|_{F^s(T)} &\lesssim \|u_{0,n}-P_{\leq M}u_{0,n}\|_{H^s} +\|P_{\leq M}u_{0,n}\|_{H^{s'}}^{1/2} \|u_{0,n}-P_{\leq M}u_{0,n}\|_{L^2}^{1/2} \\
&\lesssim \|P_{\geq M}u_{0,n}\|_{H^s}
+M^{(s'-2s)/2}\|P_{\leq M}u_{0,n}\|_{H^{s}}^{1/2} \|P_{>M}u_{0,n}\|_{H^s}^{1/2},
\end{aligned}
\end{equation*}
where, given that $s<s'<2s$, we have used that $\|P_{\leq M}u_{0,n}\|_{H^{s'}}\lesssim M^{s'-s}\|P_{\leq M}u_{0,n}\|_{H^{s}}$. From the inequality above and Lemma \ref{LemEmbedding}, we arrive at
\begin{equation}\label{eqTmain9}
\begin{aligned}
\sup_{t\in [0,T]}\|\Phi(u_{0,n})(t)-\Phi(P_{\leq M}u_{0,n})(t)\|_{H^s} \lesssim (1+\|u_{0,n}\|_{H^s}^{1/2})\|P_{>M}u_{0,n}\|_{H^s}^{1/2},
\end{aligned}
\end{equation}
where $0<T\leq T_0$. Therefore, according to our previous discussion, this completes the estimate for the first and third terms on the r.h.s of \eqref{eqTmain6}. Noticing that for $n$ large, $\|P_{> M}u_{0,n}\|_{H^s},\leq 2\|P_{> M}u_{0}\|_{H^s}$, we can take $M$ large in \eqref{eqTmain9}, and then $n,l$ large in \eqref{eqTmain6.1}, to obtain that $(\Phi(u_{0,n}))$ is a Cauchy sequence in $C([0,T];H^s(\mathbb{T}^2))$ for a fixed time $0<T\leq T_0$.
\\ \\
Since each of the elements in the sequence $(\Phi(u_{0,n}))$ solves the integral equation associated to \eqref{EQBO} in $C([0,T];H^{s-1}(\mathbb{T}^2))$, we find that the limit of this sequence is in fact a solution of the IVP \eqref{EQBO} with initial data $u_0$. This completes the existence part. Finally, it is not difficult to obtain the continuity of the flow-map from the same property for smooth solutions in Theorem \ref{existsmoothsol} and the preceding arguments. We refer to \cite{ZhangKP} for a more detailed discussion.


\section{Well-posedness results in weighted spaces}\label{Secweights}

This section is aimed to establish Theorem \ref{localweigh}. We will start introducing some preliminary results.

Given $n\in \mathbb{Z}^{+}$, we define the truncated weights $w_n : \mathbb{R} \rightarrow \mathbb{R}$ according to
 \begin{equation*}
 w_{n}(x)=\left\{\begin{aligned} 
 &\langle x \rangle, \text{ if } |x|\leq n, \\
 &2n, \text{ if } |x|\geq 3n
 \end{aligned}\right.
 \end{equation*}
in such a way that $w_n(x)$ is smooth and non-decreasing in $|x|$ with $\tilde{w}'_n(x) \leq 1$ for all $x>0$ and there exists a constant $c$ independent of $n$ such that $|\tilde{w}''_n(x)| \leq c\partial_x^2\langle x \rangle$. To explicitly show the dependence on the spatial variables $x,y$, we will denote by $w_{n,x}(x)=w_n(x)$ and $w_{n,y}(y)=w_n(y)$.

Since we are interested in performing energy estimates with the weights $w_{n}$ and then taking the limit $n\to \infty$, we must assure that the computations involving the Hilbert transform and the aforementioned weights are independent of the parameter $n$. In this direction we have:

\begin{prop}\label{propapcond}
For any $\theta \in (-1,1)$ and any $n\in \mathbb{Z}^{+}$, the Hilbert transform is bounded in $L^2(w_n^{\theta} (x)\, dx)$ with a constant depending on $\theta$ but independent of $n$.
\end{prop}

Proposition \ref{propapcond} was stated before in \cite[Proposition 1]{FonPO}. We require the identity
\begin{equation}\label{ident1}
    [H_x,x]f=0 \text{ if and only if } \int_{\mathbb{R}} f(x)\, dx=0. 
\end{equation}

We recall the following characterization of the spaces $L^p_s(\mathbb{R}^d)=J^{-s}L^p(\mathbb{R}^d)$.
\begin{theorem}( \cite{SteinThe})\label{TheoSteDer}
Let $b\in (0,1)$ and $2d/(d+2b)<p<\infty$. Then $f\in L_b^p(\mathbb{R}^d)$ if and only if

\begin{itemize}
\item[(i)]  $f\in L^p(\mathbb{R}^d)$, 
\item[(ii)]$\mathcal{D}^bf(x)=\left(\int_{\mathbb{R}^d}\frac{|f(x)-f(y)|^2}{|x-y|^{d+2b}}\, dy\right)^{1/2}\in L^{p}(\mathbb{R}^d),$
\end{itemize}
with 
\begin{equation*}
\left\|J^b f\right\|_{L^p}=\left\|(1-\Delta)^{b/2} f\right\|_{L^p} \sim \left\|f\right\|_{L^p}+\left\|\mathcal{D}^b f\right\|_{L^p} \sim \left\|f\right\|_{L^p}+\left\|D^b f\right\|_{L^p}.
\end{equation*} 
\end{theorem}
Next, we proceed to show several consequences of Theorem \ref{TheoSteDer}.
When $p=2$ and $b\in (0,1)$ one can deduce 
\begin{equation} \label{prelimneq} 
\left\|\mathcal{D}^b(fg)\right\|_{L^2} \lesssim \left\|f\mathcal{D}^b g\right\|_{L^2}+\left\|g\mathcal{D}^bf \right\|_{L^2},
\end{equation}
and it holds
\begin{equation} \label{prelimneq1}
\left\|\mathcal{D}^{b} h\right\|_{L^{\infty}} \lesssim \big(\left\|h\right\|_{L^{\infty}}+\left\|\nabla h\right\|_{L^{\infty}} \big).
\end{equation} 

\begin{prop}\label{optm1}
Let $p\in (1, \infty)$. If $f\in L^{p}(\mathbb{R})$ such that there exists $x_0\in \mathbb{R}$ for which $f(x_0^{+})$, $f(x_0^{-})$ are defined and $f(x_0^{+})\neq f(x_0^{-})$, then for any $\delta>0 $, $\mathcal{D}^{1/p}f \notin L^p_{loc}(B(x_0,\delta))$ and consequently $f\notin L^p_{1/p}(\mathbb{R})$.
\end{prop}

\begin{prop}\label{propprelimneq2}
Let $b\in (0,1)$. For any $t>0$ 
\begin{equation}\label{prelimneq2}
\mathcal{D}^b (e^{ix|x|t})\lesssim (|t|^{b/2}+|t|^b|x|^b), \hspace{0.5cm} x\in \mathbb{R}
\end{equation}
and 
\begin{equation}\label{prelimneq2.5}
  \mathcal{D}^b (e^{i\sign(x)t\mp i\sign(x)\eta^2 t })\lesssim |x|^{-b}, \hspace{0.5cm} x\in \mathbb{R}\setminus \{0\}, 
\end{equation}
for all $\eta \in \mathbb{R}$.
\end{prop}
\begin{proof}
Estimate \eqref{prelimneq2} follows from the same arguments in \cite{NahPonc}. On the other hand, since $|e^{i\sign(x)t\mp i\sign(x)\eta^2 t }-e^{i\sign(y)t\mp i\sign(y)\eta^2 t }|=0$, whenever $\sign(y)=\sign(x)$, we change variables to find
\begin{equation*}
\begin{aligned}
    \mathcal{D}^{b}(e^{i\sign(x)t\mp i\sign(x)\eta^2 t })&=\Big(\int_{\mathbb{R}} \frac{|e^{i\sign(x)t\mp i\sign(x)\eta^2 t }-e^{i\sign(y)t\mp i\sign(y)\eta^2 t }|^2}{|x-y|^{1+2b}} \, dy \Big)^{1/2} \\
    &\lesssim  \Big(\int_{y\geq |x|} \frac{1}{|y|^{1+2b}} \, dy \Big)^{1/2}\sim |x|^{-b}.
\end{aligned}
\end{equation*}
This completes the deduction of \eqref{prelimneq2.5}.
\end{proof}

The following result will be useful to study the behavior of solutions of \eqref{EQBO} in $L^2(|x|^{2r}\, dx dy)$, whenever $r\in(1/2,1]$. 

\begin{lemma}\label{interpo}
Let $1/2<s\leq 1$ and $f\in H^s(\mathbb{R})$ such that $f(0)=0$. Then, $\|\sign(\xi) f\|_{H^s}\lesssim \|f\|_{H^s}$. 
\end{lemma}
\begin{proof}
Since the case $s=1$ can be easily verified, we will restrict our considerations to the case $1/2<s<1$. We first notice that the same argument in the deduction of \eqref{prelimneq2.5} establishes
\begin{equation*}
\mathcal{D}^{s}(\sign(x))\sim |x|^{-s}.
\end{equation*} 
Thus, an application of \eqref{prelimneq} and the previous result reduces our analysis to prove
\begin{equation}\label{interpoeq1}
\||\cdot|^{-s}f\|_{L^2}\lesssim \|f\|_{H^s}. 
\end{equation}
However, the preceding estimate is a consequence of \cite[Proposition 3.2]{Yafaev} and the assumption $f(0)=0$.
\end{proof}

We shall also employ the following interpolation inequality which is proved in much the same way as in \cite[Lemma 1]{FonPO}:

\begin{lemma}
Let $a,b>0$. Assume that $J^af=(1-\Delta)^{a/2}f \in L^{2}(\mathbb{R}^d)$ and $\langle x \rangle^b f=(1+|x|^2)^{b/2}f\in L^{2}(\mathbb{R}^d)$, $|x|=\sqrt{x_1^2+\dots x_d^2}$. Then for any $\alpha \in (0,1)$, 
\begin{equation}\label{prelimneq3}
\left\|J^{\alpha a}(\langle x \rangle^{(1-\alpha)b}f)\right\|_{L^2}\lesssim \left\|\langle x \rangle^{b} f\right\|_{L^2}^{1-\alpha}\left\|J^a f\right\|_{L^2}^{\alpha}.
\end{equation}
Moreover, the inequality \eqref{prelimneq3} is still valid with $w_n(|x|)$ instead of $\langle x \rangle$ with a constant $c$ independent of $n$.
\end{lemma}

Now we are in the condition to prove Theorem \ref{localweigh}

\subsection{Proof of Theorem \ref{localweigh}}

In view of Theorem \ref{Improwellp}, for a given  $u_0\in Z_{r_1,r_2,s}(\mathbb{R}^2)=H^{s}(\mathbb{R}^2)\cap L^{2}((|x|^{2r_1}+|y|^{2r_2}) \, dxdy)$ there exist $T=T(\left\|u_0\right\|_{H^s})>0$ and $u\in C([0,T];H^s(\mathbb{R}^2))\cap L^1([0,T]; W^{1,\infty}(\mathbb{R}^2))$ solution of the IVP \eqref{EQBO}. Let $0\leq K<\infty$ defined by
\begin{equation}\label{weieq0}
    K=\|u\|_{L^{\infty}_T H^s}+\|u\|_{L^1_T L^{\infty}_{xy}}+\|\nabla u\|_{L^1_T L^{\infty}_{xy}}.
\end{equation}
In what follows, we will assume that $u$ is sufficiently regular to perform all the computations required in this section. Indeed, recalling the comments in the Subsection \ref{subHs}, we consider the sequence of smooth solutions $u_N\in C([0,T];H^{\infty}(\mathbb{R}^{2}))$ with $u_N(0)=P_{\leq N}u_0\in H^{\infty}(\mathbb{R}^2)\cap  L^{2}((|x|^{2r_1}+|y|^{2r_2}) \, dxdy)$, then \eqref{aproxres} holds and  $u_N(0) \to u_0$ in the $Z_{r_1,r_2,s}(\mathbb{R}^2)$ topology. Consequently, applying our arguments to $u_N$ and then taking the limit $N\to \infty $, we can impose the required assumptions on $u$.

\subsubsection{Proof of Theorem \ref{localweigh} (i)}

Let us first prove the persistence property $u \in C([0,T]; L^{2}((|x|^{2r_1}+|y|^{2r_2}) \, dxdy))$. We begin by deriving some estimates in the spaces $L^{2}(|x|^{2r_1}\, dxdy)$ and $L^{2}(|y|^{2r_2}\, dxdy)$.
\\ \\
{\bf Estimate for the $L^{2}(|x|^{2r_1}\, dxdy)$-norm}. Here, $0<r_1 <1/2$ fixed. We apply $\mathcal{H}_x$ to the equation in \eqref{EQBO} to find
\begin{equation}\label{HEQBO}
    \partial_t \mathcal{H}_x u - u + \partial_x^2u\mp \partial_y^2u+\mathcal{H}_{x}(u\partial_x u)=0,
\end{equation}
multiplying then by $\mathcal{H}_xu \, w^{2r_1}_{n,x}$ and integrating in space, we infer
\begin{equation}\label{weieq1}
\begin{aligned}
\frac{1}{2}\frac{d}{dt}\|\mathcal{H}_x u(t) \, w^{r_1}_{n,x}\|_{L^2_{xy}}^2&-\int u \mathcal{H}_xu \, w^{2r_1}_{n,x}\, dx dy+\int \partial_x^2u \mathcal{H}_x u \, w_{n,x}^{2r_1}\, dx dy \\
&\mp \int \partial_y^2 u  \mathcal{H}_x u \, w_{n,x}^{2r_1}\, dxdy +\int\mathcal{H}_x(u\partial_x u) \mathcal{H}_x u \, w_{n,x}^{2r_1} \, dx dy=0.
\end{aligned}
\end{equation}
Multiplying the equation in \eqref{EQBO} by $uw_{N,x}^{2r_1}$ and then integrating in  space, it is seen that
\begin{equation}\label{weieq2}
\begin{aligned}
\frac{1}{2}\frac{d}{dt}\|u(t) w^{r_1}_{n,x}\|_{L^2_{xy}}^2 &+\int \mathcal{H}_xu u w^{2r_1}_{n,x}\, dx dy-\int \mathcal{H}_x\partial_x^2u u w_{n,x}^{2r_1}\, dx dy \\
&\pm \int \mathcal{H}_x\partial_y^2 u u w_{n,x}^{2r_1}\, dxdy +\int u\partial_x u u w_{n,x}^{2r_1} \, dx dy=0.
\end{aligned}
\end{equation}
Adding the differential equations \eqref{weieq1} and \eqref{weieq2}, after integrating by parts in the $y$ variable we deduce
\begin{equation}\label{weieq3}
\begin{aligned}
\frac{1}{2}\frac{d}{dt}\big(\|u(t) w^{r_1}_{n,x}\|_{L^2_{xy}}^2+\|\mathcal{H}_x u(t) \, w^{r_1}_{n,x}\|_{L^2_{xy}}^2 \big)=&\int (\mathcal{H}_x\partial_x^2u u-\partial_x^2u \mathcal{H}_x u) w_{n,x}^{2r_1}\, dx dy \\
&- \int (u\partial_x u u +\mathcal{H}_x(u\partial_x u )\mathcal{H}_x u)w_{n,x}^{2r_1} \, dx dy \\
=&:Q_1+Q_2.
\end{aligned}
\end{equation}
Now, since $0<r_1<1/2$,  $|\partial_xw_{n,x}^{2r_1}| \lesssim w_{n,x}^{r_1}$ with implicit constant independent of $n$, integrating by parts and using the Cauchy-Schwarz inequality we find
\begin{equation*}
    \begin{aligned}
    |Q_1|&=\Big|\int \partial_x\mathcal{H}_x  u u \partial_xw_{n,x}^{2r_1} \, dx dy -\int \partial_x u \mathcal{H}_xu\, \partial_xw_{n,x}^{2r_1}\, dx dy\Big| \\
    & \lesssim \|\partial_x u\|_{L^{\infty}_{T}L^2_{xy}}\| uw_{n,x}^{r_1}\|_{L^2_{xy}}+ \|\partial_x u\|_{L^{\infty}_{T}L^2_{xy}}\| \mathcal{H}_xu \, w_{n,x}^{r_1}\|_{L^2_{xy}}.
    \end{aligned}
\end{equation*}
Notice that the norm $\|\partial_x u\|_{L^{\infty}_{T}L^2_{xy}}$ is controlled by \eqref{weieq0}. Next, since $0<r_1<1/2$, Proposition \ref{propapcond} shows
\begin{equation*}
    \|\mathcal{H}_x(u\partial_x u)w_{n,x}^{r_1}\|_{L^{2}_{xy}} =\|\|\mathcal{H}_x(u\partial_x u)w_{n,x}^{r_1}\|_{L^{2}_x}\|_{L^{2}_{y}}\lesssim \|u\partial_x u\, w_{n,x}^{r_1}\|_{L^{2}_{xy}} \lesssim \|\partial_x u\|_{L^{\infty}_{xy}}\|u \, w_{n,x}^{r_1}\|_{L^2_{xy}}.
\end{equation*}
Hence, we employ H\"older's inequality to get
\begin{equation*}
\begin{aligned}
    |Q_2| &\leq \|\partial_x u\|_{L^{\infty}_{xy}} \|u w_{n,x}^{r_1}\|_{L^{2}_{xy}}^2+ \|\mathcal{H}_x(u\partial_x u)w_{n,x}^{r_1}\|_{L^{2}_{xy}}\|\mathcal{H}_xu \,w_{n,x}^{r_1}\|_{L^{2}_{xy}} \\
    & \lesssim \|\partial_x u\|_{L^{\infty}_{xy}} \|u w_{n,x}^{r_1}\|_{L^{2}_{xy}}^2+ \|\partial_{x}u\|_{L^{\infty}_{xy}}\|u w_{n,x}^{r_1}\|_{L^{2}_{xy}}\|\mathcal{H}_xu \,w_{n,x}^{r_1}\|_{L^{2}_{xy}}.
\end{aligned}
\end{equation*}
Thus, gathering the previous estimates,
\begin{equation}\label{weieq3.0}
\begin{aligned}
    \frac{1}{2}\frac{d}{dt}\big(\|u(t) w^{r_1}_{n,x}\|_{L^2_{xy}}^2&+\|\mathcal{H}_x u(t) \, w^{r_1}_{n,x}\|_{L^2_{xy}}^2 \big) \\
    &\lesssim \big(\|u w_{n,x}^{r_1}\|_{L^2_{xy}}^2+\|\mathcal{H}_xu w_{n,x}^{r_1}\|_{L^2_{xy}}^2\big)^{1/2} +\|\partial_x u\|_{L^{\infty}_{xy}}\big(\|u w_{n,x}^{r_1}\|_{L^2_{xy}}^2+\|\mathcal{H}_xu w_{n,x}^{r_1}\|_{L^2_{xy}}^2\big).
\end{aligned}
\end{equation}
\\ \
{\bf Estimate for the $L^{2}(|y|^{2r_2}\, dx dy)$-norm}. In this case, $r_2>0$ is arbitrary. Multiplying the equation in \eqref{EQBO} by $uw_{n,y}^{r_2}$ and integrating in space yield
\begin{equation}\label{weieq3.1}
\begin{aligned}
 \frac{1}{2}\frac{d}{dt}\|u(t) w_{n,y}^{r_2}\|_{L^2_{xy}}^{2} =&-\int \mathcal{H}_x u w_{n,y}^{r_2} u w_{n,y}^{r_2}\, dx dy+\int \mathcal{H}_x \partial_x^2 u w_{n,y}^{r_2} uw_{n,y}^{r_2} \, dx dy\\
 &\mp \int \mathcal{H}_x \partial_y^2 u w_{n,y}^{r_2} u w_{n,y}^{r_2}\, dx dy-\int u\partial_x u w_{n,y}^{r_2} u w_{n,y}^{r_2}\, dx dy \\
 =:&A_1+A_2+A_3+A_4.
\end{aligned}
\end{equation}
Since the weight function $w_{n,y}^{r_2}=w_{n,y}^{r_2}(y)$ does not depend on $x$, writing $\mathcal{H}_x u w_{n,y}^{r_2}=\mathcal{H}_x( u w_{n,y}^{r_2})$ and using that $\mathcal{H}_x$ determines a skew-symmetric operator, we have that $A_1=0$. Similarly, integrating by parts on the $x$ variable and writing $\mathcal{H}_x \partial_x u w_{n,y}^{r_2}=\mathcal{H}_x( \partial_x u w_{n,y}^{r_2})$, it follows that $A_2=0$. 
\\ \\
Now, integrating by parts and using that $\mathcal{H}_x$ is skew-symmetric, it is not difficult to see
\begin{equation}\label{weieq4}
    \begin{aligned}
    |A_3|=\Big|2 \int \mathcal{H}_x \partial_y u \partial_y w_{n,y}^{r_2}u w_{n,y}^{r_2} \, dx dy  \Big| \lesssim \|\partial_y u \partial_y w_{n,y}^{r_2}\|_{L^2_{xy}}\|u w_{n,y}^{r_2}\|_{L^2_{xy}}.
    \end{aligned}
\end{equation}
From the fact that $|\partial_y^{l} w_{n,y}^{r_2}|\lesssim w_{n,y}^{r_2-l}$, $l=1,2$ with a constant independent of $n$ and \eqref{weieq0}, it follows 
\begin{equation}\label{weieq5}
    \|\partial_y u \partial_y w_{n,y}^{r_2}\|_{L^2_{xy}} \lesssim \|\partial_y u\|_{L^{\infty}_T L^2_{xy}} \lesssim K,
\end{equation}
whenever $0<r_2 \leq 1$. Now, if $r_2>1$, we have
\begin{equation}\label{weieq6}
    \begin{aligned}
    \|\partial_y u \partial_y w_{n,y}^{r_2}\|_{L^2_{x,y}}\lesssim  \|J_y(uw_{n,y}^{r_2-1})\|_{L^2_{xy}}+\|u \partial_y^2w_{n,y}^{r_2}\|_{L^{2}_{xy}} 
    \lesssim  \|J_y(uw_{n,y}^{r_2-1})\|_{L^2_{xy}}+\|u w_{n,y}^{r_2}\|_{L^{2}_{xy}},
    \end{aligned}
\end{equation}
where we have employed the identity $\partial_y u  w_{n,y}^{r_2-1}=\partial_y(u w_{n,y}^{r_2-1})-u \partial_y w_{n,y}^{r_2-1}$. To estimate the last expression in the preceding inequality, we choose $\alpha=r_2^{-1}$, $a=b=r_2$ in \eqref{prelimneq3} and applying Young's inequality it is seen that
\begin{equation}\label{weieq7}
    \begin{aligned}
    \|J_y(uw_{n,y}^{r_2-1})\|_{L^2_{xy}}=\|\|J_y(uw_{n,y}^{r_2-1})\|_{L^2_y}\|_{L^2_{x}} &\lesssim \|\|uw_{n,y}^{r_2}\|_{L^2_y}^{(r_2-1)/r_2}\|J^{r_2}_yu\|_{L^2_y}^{1/r_2}\|_{L^2_{x}} \\
    &\lesssim \|uw_{n,y}^{r_2}\|_{L^2_{xy}}+\|J^{r_2}u\|_{L^2_{xy}}.
    \end{aligned}
\end{equation}
Thus, choosing $s>\max\{3/2,r_2\}$, \eqref{weieq5}-\eqref{weieq7} and \eqref{weieq0} imply
\begin{equation}
    |A_3|\lesssim \|uw_{n,y}^{r_2}\|_{L^2_{xy}}+\|uw_{n,y}^{r_2}\|^2_{L^2_{xy}}.
\end{equation}
Finally, 
\begin{equation}
    |A_4|\lesssim \|\partial_x u\|_{L^{\infty}_{xy}}\|u w_{n,y}^{r_2}\|_{L^2_{xy}}^2.
\end{equation}
Plugging the estimates for $A_j$, $j=1,\dots,4$ in \eqref{weieq3.1} yields
\begin{equation}\label{weieq8}
    \frac{1}{2}\frac{d}{dt}\|u(t)w_{n,y}^{r_2}\|_{L^2_{xy}}^2 \lesssim \|u w_{n,y}^{r_2}\|_{L^2_{xy}}+(1+\|\partial_x u\|_{L^{\infty}_{xy}})\|u w_{n,y}^{r_2}\|_{L^2_{xy}}^2.
\end{equation}
This completes the desired estimate for the $L^{2}(|y|^{2r_2}\, dxdy)$-norm;
\\ \\
Now, we collect the estimates derived for the norms $L^{2}(|x|^{2r_1}\, dxdy)$ and $L^{2}(|y|^{2r_2}\, dxdy)$ to conclude Theorem \ref{localweigh} (i). Letting
\begin{equation*}
    g(t)= \|u(t) w^{r_1}_{n,x}\|_{L^2_{xy}}^2+\|\mathcal{H}_x u(t) \, w^{r_1}_{n,x}\|_{L^2_{xy}}^2+ \|u(t)w_{n,y}^{r_2}\|_{L^2_{xy}}^2,
\end{equation*}
the inequalities \eqref{weieq3.0} and \eqref{weieq8} assure that there exists some constant $c_0$ independent of $n$ such that
\begin{equation}
    \frac{d}{dt}g(t) \leq c_0 g(t)^{1/2}+c_0(1+ \|\partial_x u\|_{L^{\infty}_{xy}})g(t).
\end{equation}
Then, Gronwall's inequality implies

\begin{equation}
\begin{aligned}
   \|u(t) w^{r_1}_{n,x}&\|_{L^2_{xy}}^2+\|\mathcal{H}_x u(t)  w^{r_1}_{n,x}\|_{L^2_{xy}}^2+ \|u(t)w_{n,y}^{r_2}\|_{L^2_{xy}}^2\\
   &\leq \big((\|u_0 \langle x\rangle^{r_1}\|_{L^2_{xy}}^2+ +\|\mathcal{H}_x u_0 \langle x\rangle^{r_1}\|_{L^2_{xy}}^2+ \|u_0\langle y\rangle^{r_2}\|_{L^2_{xy}}^2)^{1/2}+c_0t/2\big)^2e^{c_0t+c_0\int_0^t\|\nabla u(s)\|_{L^{\infty}_{xy}}ds}.
\end{aligned}
\end{equation}
Thus, taking $n\to \infty$ in the previous inequality shows
\begin{equation}\label{weieq9} 
\begin{aligned}
      \|u(t) \langle x\rangle^{r_1}&\|_{L^2_{xy}}^2+\|\mathcal{H}_xu(t) \langle x\rangle^{r_1}\|_{L^2_{xy}}^2+\|u(t) \langle y\rangle^{r_2}\|_{L^2_{xy}}^2 \\
      &\leq \big((\|u_0 \langle x \rangle^{r_1}\|_{L^2_{xy}}^2+\|\mathcal{H}_xu_0 \langle x \rangle^{r_1}\|_{L^2_{xy}}^2+\|u_0 \langle y \rangle^{r_2}\|_{L^2_{xy}}^2)^{1/2}+c_0t/2\big)^2e^{c_0t+c_0\int_0^t\|\nabla u(s)\|_{L^{\infty}_{xy}}ds}.
\end{aligned}
\end{equation}
This shows that $u\in L^{\infty}([0,T];L^{2}(|x|^{2r_1}+|y|^{2r_2} \, dx dy))$. Now, we shall prove  that $u \in C([0,T];L^{2}(|x|^{2r_1}+|y|^{2r_2} \, dx dy))$. Firstly, since $u\in C([0,T];H^s(\mathbb{R}^2))$, it is not difficult to see that $u:[0,T]\mapsto L^2(|x|^{2r_1}+|y|^{2r_2} \, dx dy)$  is weakly continuous. The same is true for the map $\mathcal{H}_x u(t)$ on $L^2(|x|^{2r_1}\, dxdy)$. On the other hand, \eqref{weieq9} implies
\begin{equation}\label{weieq10}
    \begin{aligned}
            \|(&u(t)-u_0) \langle x\rangle^{r_1}\|_{L^2_{xy}}^2+ \|\mathcal{H}_x(u(t)-u_0) \langle x\rangle^{r_1}\|_{L^2_{xy}}^2+\|(u(t)-u_0) \langle y\rangle^{r_2}\|_{L^2_{xy}}^2 \\
            =&\|u(t) \langle x\rangle^{r_1}\|_{L^2_{xy}}^2+\|\mathcal{H}_xu(t) \langle x\rangle^{r_1}\|_{L^2_{xy}}^2+\|u(t) \langle y\rangle^{r_2}\|_{L^2_{xy}}^2+ \|u_0 \langle x\rangle^{r_1}\|_{L^2_{xy}}^2+ \|\mathcal{H}_x u_0 \langle x\rangle^{r_1}\|_{L^2_{xy}}^2 \\
            &+\|u_0 \langle y\rangle^{r_2}\|_{L^2_{xy}}^2-2\int u(t)u_0 \langle x\rangle^{2r_1}\, dx dy-2\int \mathcal{H}_xu(t)\mathcal{H}_xu_0 \langle x\rangle^{2r_1}\, dx dy-2\int u(t)u_0 \langle y\rangle^{2r_2}\, dx dy \\
         \leq & \big((\|u_0 \langle x \rangle^{r_1}\|_{L^2_{xy}}^2+\|\mathcal{H}_xu_0 \langle x \rangle^{r_1}\|_{L^2_{xy}}^2+\|u_0 \langle y \rangle^{r_2}\|_{L^2_{xy}}^2)^{1/2}+c_0t/2\big)^2 e^{c_0t+c_0\int_0^t\|\nabla u(s)\|_{L^{\infty}_{xy}}ds} \\ &+ \|u_0 \langle x\rangle^{r_1}\|_{L^2_{xy}}^2+\|\mathcal{H}_xu_0 \langle x\rangle^{r_1}\|_{L^2_{xy}}^2+\|u_0 \langle y\rangle^{r_2}\|_{L^2_{xy}}^2-2\int u(t)u_0 \langle x\rangle^{2r_1}\, dx dy \\
         &-2\int \mathcal{H}_xu(t)\mathcal{H}_x u_0 \langle x\rangle^{2r_1}\, dx dy-2\int u(t)u_0 \langle y\rangle^{2r_2}\, dx dy.
    \end{aligned}
\end{equation}
Clearly, weak continuity implies that the right-hand side of \eqref{weieq10} goes to zero as $t\to 0^{+}$. This shows right continuity at the origin of the map $u:[0,T]\mapsto L^2(|x|^{2r_1}+|y|^{2r_2} \, dx dy)$. Taking any $\tau \in (0,T)$ and using that the equation in \eqref{EQBO} is invariant under the transformations: $(x,y,t)\mapsto (x,y, t+\tau)$ and  $(x,y,t)\mapsto (-x,-y,\tau-t)$,  right continuity at the origin yields continuity to the whole interval $[0,T]$, in other words, $u \in C([0,T]; L^{2}(|x|^{2r_1}+|y|^{2r_2} \, dxdy))$. 

The continuous dependence on the initial data follows from this property in $H^{s}(\mathbb{R}^2)$ and the same reasoning above applied to the difference of two solutions. This completes the proof of Theorem \ref{localweigh} (i).  

\subsubsection{Proof of Theorem \ref{localweigh} (ii) and (iii)}

Let $u\in C([0,T]; H^s(\mathbb{R}^2))\cap L^1([0,T];W^{1,\infty}(\mathbb{R}^2))$ be the solution of the IVP \eqref{EQBO} with $u_0 \in ZH_{s,1/2,r_2}(\mathbb{R}^2)$ for Theorem \ref{localweigh} (ii), or satisfying $u_0 \in \dot{Z}_{s,r_1,r_2}(\mathbb{R}^2)$, $1/2<r_1<3/2$ for Theorem \ref{localweigh} (iii). Since we have already established that solutions of the IVP \eqref{EQBO} preserve arbitrary polynomial decay in the $y$-variable, we will restrict our considerations to deduce $u, \mathcal{H}_x u\in L^{\infty}([0,T]; L^{2}(|x|^{2r_1}\, dxdy))$, $r_1\geq 1/2$. Once this has been done, following the arguments in \eqref{weieq10}, we will have that $u,\mathcal{H}_x u \in C([0,T];L^{2}(|x|^{2r_1}\, dx dy))$.

Moreover, the continuous dependence on the spaces $ZH_{s,1/2,r_2}(\mathbb{R}^2)$ and $\dot{Z}_{s,r_1,r_2}(\mathbb{R}^2)$, $r_1>1/2$ follows by the same energy estimate leading to $u, \mathcal{H}_x u\in L^{\infty}([0,T]; L^{2}(|x|^{2r_1}\, dxdy))$ applied to the difference of two solutions.  
\\ \\
Now, to assure the persistence property in $\dot{Z}_{s,r_1,r_2}(\mathbb{R}^2)$ for Theorem \ref{localweigh} (iii), we require the following claim:

\begin{claim}\label{claim0} Let $r_1\in(1/2,3/2)$, $s>3/2$ fixed and $$u \in C([0,T],H^s(\mathbb{R}^2))\cap L^1([0,T];W^{1,\infty}_x(\mathbb{R}^2))\cap L^{\infty}([0,T]; L^2(|x|^{2r_1}\, dxdy))$$
be a solution of the IVP \eqref{EQBO}. Assume that $\widehat{u}(0,\eta)=\widehat{u_0}(0,\eta)=0$ for a.e $\eta$. Then, $\widehat{u}(0,\eta,t)=0$ for every $t\in [0,T]$ and almost every $\eta \in \mathbb{R}$.
\end{claim}

\begin{proof}
Since $u$ solves the integral equation associated to \eqref{EQBO}, taking its Fourier transform we find
\begin{equation}\label{FouINte}
    \widehat{u}(\xi,\eta,t)=e^{i\omega(\xi,\eta)t}\widehat{u_0}(\xi,\eta)-\frac{i\xi}{2}\int_0^t e^{i\omega(\xi,\eta)(t-t')} \widehat{u^2}(\xi,\eta,t')\, dt',
\end{equation}
where $\omega(\xi,\eta)$ is defined by \eqref{lieareqsym}. Now, the assumptions imposed on the solution show
$$u^2\in L^1([0,T];L^2(|x|^{2r_1}\, dxdy)).$$
Hence, the above conclusion, Fubini's theorem and Sobolev's embedding on the $\xi$-variable determines $\widehat{u}(\xi,\eta,t)$ and $\int_0^t e^{i\omega(\xi,\eta)(t-t')}\widehat{u^2}(\xi,\eta,t')\, dt'$ are continuous on $\xi$ for every $t\in [0,T]$ and almost every $\eta$. From this, \eqref{FouINte} yields the desired result.
\end{proof}

We begin by considering the case $1/2\leq r_1 \leq 1$. We employ the differential equation \eqref{weieq3} with the present restrictions on $r_1$. Thus, we will derive bounds for $Q_1$ and $Q_2$ defined as in \eqref{weieq3} for this case. Integrating by parts we get
\begin{equation}
    \begin{aligned}
    |Q_1|&=\Big|\int \partial_x\mathcal{H}_x  u u \partial_xw_{n,x}^{2r_1} \, dx dy -\int \partial_x u \mathcal{H}_xu\, \partial_xw_{n,x}^{2r_1}\, dx dy \Big| \\
& \lesssim \|\partial_x u\|_{L^2_{xy}}\|u w_{n,x}^{r_1}\|_{L^2_{xy}}+\|\partial_x u\|_{L^2_{xy}}\|\mathcal{H}_xu w_{n,x}^{r_1}\|_{L^2_{xy}},
    \end{aligned}
\end{equation}
where, given that $1/2\leq r_1 \leq 1$, we have used $|\partial_x w_{n,x}^{2r_1}|\lesssim |w_{n,x}^{r_1}|$. On the other hand,
\begin{equation}
\begin{aligned}
  Q_2&=-\int u^2\partial_x u w_{n,x}^{2r_1}-\frac{1}{2}\int \mathcal{H}_x(\partial_x u^2)\mathcal{H}_x u w_{n,x}^{2r_1} \, dx dy \\
  &=-\int u^2\partial_x u w_{n,x}^{2r_1}-\frac{1}{2}\int [w_{n,x}^{r_1},\mathcal{H}_x]\partial_x u^2 \mathcal{H}_x u w_{n,x}^{r_1} \, dx dy-\frac{1}{2}\int \mathcal{H}_x (\partial_x u^2 w_{n,x}^{r_1})\mathcal{H}_x u w_{n,x}^{r_1} \, dx dy.
\end{aligned}
\end{equation}
Hence, Proposition \ref{CalderonComGU} and H\"older's inequality allow us to deduce
\begin{equation}
\begin{aligned}
|Q_2| \lesssim& \|\partial_x u\|_{L^{\infty}_{xy}}\|u w_{n,x}^{r_1}\|_{L^{2}_{xy}}^2+\|\partial_xw_{n,x}^{r_1}\|_{L^{\infty}_{xy}}\|u\|_{L^{\infty}_{xy}}\|u\|_{L^2_{xy}}\|\mathcal{H}_x u w_{n,x}^{r_1}\|_{L^2_{xy}} \\ &+\|\partial_x u\|_{L^{\infty}_{xy}}\|u w_{n,x}^{r_1}\|_{L^2_{xy}}\|\mathcal{H}_x u w_{n,x}^{r_1}\|_{L^2_{xy}}.
\end{aligned}
\end{equation}
Since, $|\partial_x w_{n,x}^{r_1}|\lesssim 1$ with an implicit constant independent of $n$, we combine the estimates for $Q_1$ and $Q_2$ to obtain the same differential inequality \eqref{weieq3.0} adapted for this case. Consequently,  this estimate, Gronwall's inequality and the assumption $\mathcal{H}u_0\in L^2(|x|\, dxdy)$ imply $u, \mathcal{H}_xu\in L^{\infty}([0,T]; L^{2}(|x|\,dx dy))$. The proof of Theorem \ref{localweigh} (ii) is complete.

On the other hand,  under the hypothesis of Theorem \ref{localweigh} (iii), since $\widehat{u_0}(0,\eta)=0$ a.e $\eta$, Lemma \ref{interpo} and Plancherel's identity assure that $\mathcal{H}_xu_0 \in L^{2}(|x|^{2r_1}\, dx dy)$ for $1/2<r_1\leq 1$. Then Gronwall's inequality and the differential inequality \eqref{weieq3.0} for this case yield $u\in L^{\infty}([0,T]; L^{2}(|x|^{2r_1}\, dx dy))$, whenever $1/2<r_1\leq 1$. This consequence and Claim \ref{claim0} complete the LWP results in $\dot{Z}_{s,r_1,r_2}(\mathbb{R}^2)$, $1/2<r_1\leq 1$.
\\ \\
Now, we assume that $1<r_1<3/2$. We write $r_1=1+\theta$ with $0<\theta<1/2$. We first notice that Claim \ref{claim0}, identity \eqref{ident1} and the preceding well-posedness conclusion yield $u, \mathcal{H}_x u\in C([0,T]; L^2(|x|^{2} \, dx dy))$. Thus, we multiply the equation in \eqref{EQBO} by $ux^2w_{n,x}^{2\theta}$  and \eqref{HEQBO} by $\mathcal{H}_xu x^2w_{n,x}^{2\theta}$, then integrating in space and adding the resulting expressions reveal
\begin{equation}\label{weieq10.1}
\begin{aligned}
\frac{1}{2}\frac{d}{dt}\big(\|u(t) xw^{\theta}_{n,x}\|_{L^2_{xy}}^2+\|\mathcal{H}_x u(t) \, x w^{\theta}_{n,x}\|_{L^2_{xy}}^2 \big)=&\int (\mathcal{H}_x\partial_x^2u u-\partial_x^2u \mathcal{H}_x u) x^2w_{n,x}^{2\theta}\, dx dy \\
&- \int (u\partial_x u u +\mathcal{H}_x(u\partial_x u )\mathcal{H}_x u)x^2w_{n,x}^{2\theta} \, dx dy \\
=&:\widetilde{Q}_1+\widetilde{Q}_2.
\end{aligned}
\end{equation}
Integrating by parts on the $x$-variable,
\begin{equation}
\begin{aligned}
  \widetilde{Q}_1=&-2\big(\int \mathcal{H}_x \partial_x u u\,  x w_{n,x}^{2\theta}\, dx dy-\int  \partial_x u \mathcal{H}_x u  \, x  w_{n,x}^{2\theta} \, dx dy \big) \\
  &-\big( \int \mathcal{H}_x \partial_x u u x^2 \partial_xw_{n,x}^{2\theta}\, dx dy -\int  \partial_x u \mathcal{H}_xu x^2 \partial_xw_{n,x}^{2\theta} \, dx dy\big)=: \widetilde{Q}_{1,1}+\widetilde{Q}_{1,2}.
\end{aligned}
\end{equation}
The Cauchy-Schwarz inequality and Proposition \ref{propapcond} determine
\begin{equation}\label{weieq11}
    \begin{aligned}
    |\widetilde{Q}_{1,1}| &\lesssim \|\mathcal{H}_x\partial_x u w_{n,x}^{\theta}\|_{L^2_{xy}}\| u \,x w_{n,x}^{\theta}\|_{L^2_{xy}}+\|\partial_x u w_{n,x}^{\theta}\|_{L^2_{xy}}\| \mathcal{H}_x u \, x w_{n,x}^{\theta}\|_{L^2_{xy}}\\
    &\lesssim (\|J_x( u w_{n,x}^{\theta})\|_{L^2_{xy}}+\|u\|_{L^2_{xy}})(\| u \,x w_{n,x}^{\theta}\|_{L^2_{xy}}+\| \mathcal{H}_x u \,x w_{n,x}^{\theta}\|_{L^2_{xy}}).
    \end{aligned}
\end{equation}
Where we have used the identity $\partial_x u w_{n,x}^{\theta}=\partial_x (u w_{n,x}^{\theta})-u\partial_{x}w_{n,x}^{\theta}$. By complex interpolation \eqref{prelimneq3} with $\alpha=1/(1+\theta)$ and $a=b=1+\theta$, we argue as in \eqref{weieq7}, using that $|w_{n,x}^{1+\theta}|\lesssim w_{n,x}^{\theta}+|x|w_{n,x}^{\theta}$ to deduce $\|J_x(u w_{n,x}^{\theta})\|_{L^2_{xy}}\lesssim \|u w_{n,x}^{\theta}\|_{L^2_{xy}}+\|u xw_{n,x}^{\theta}\|_{L^2_{xy}}+\|J^{1+\theta}u\|_{L^2_{xy}}$. This previous estimate, the fact that $u\in C([0,T];L^2(|x|^{2r}\, dxdy))$, $0\leq r \leq 1$ and \eqref{weieq11} complete the study of $\widetilde{Q}_{1,1}$. \\ \\
On the other hand, since $|x^2 \partial_x w_{n,x}^{2\theta}|\lesssim w_{n,x}^{1+2\theta}$ with implicit constant independent of $n$, the estimate for $\widetilde{Q}_{1,2}$ follows the same ideas employed to estimate $\widetilde{Q}_{1,1}.$ 

Finally, identity \eqref{ident1} and Proposition \ref{propapcond} show
\begin{equation}
    |\widetilde{Q}_2|\lesssim \|\partial_x u\|_{L^{\infty}_{xy}}\|u\, x w_{n,x}^{\theta}\|^2_{L^{2}_{xy}}+\|\partial_x u\|_{L^{\infty}_{xy}}\|u\, x w_{n,x}^{\theta}\|_{L^{2}_{xy}}\|\mathcal{H}_xu\, x w_{n,x}^{\theta}\|_{L^{2}_{xy}}.
    \end{equation}
Noticing that \eqref{ident1} implies $\mathcal{H}_x u_0 \, x w_{m, x}=\mathcal{H}_x (xu_0) w_{n,x}^{\theta} \in L^{2}(\mathbb{R}^2)$. Thus, we can employ recurrent arguments combining the previous estimates for $\widetilde{Q}_1$, $\widetilde{Q}_2$, \eqref{weieq10.1} and Gronwall's inequality to conclude $u\in L^{\infty}(|x|^{2r_1}\, dxdy)$, whenever $1<r_1<3/2$. The proof of Theorem \ref{localweigh} (iii) is complete.

\subsection{Proof of Theorem \ref{sharpdecay}}

Without loss of generality we shall assume that $t_1=0$, i.e., $u_0\in Z_{s,(1/2)^{+},r_2}(\mathbb{R}^2)$ and $u(t_2)\in Z_{s,1/2,r_2}(\mathbb{R}^2)$. So that $u\in C([0,T];Z_{s,r_1,r_2}(\mathbb{R}^2))\cap L^1([0,T];W_{1,x}^{\infty}(\mathbb{R}^2)),$ where $r_1\in (1/4,1/2)$, $r_2\geq r_1$ and $s\geq \max\{ \frac{ 2r_1}{(4r_1-1)^{-}},r_2\}$.
The solution of the IVP \eqref{EQBO} can be represented by Duhamel's formula
\begin{equation}\label{inteequ}
u(t)=S(t)u_0-\int_0^{t} S(t-t')u\partial_{x} u(t') \, dt'.
\end{equation}
Since our arguments require localizing near the origin, we consider a function $\phi \in C^{\infty}_c(\mathbb{R})$ such that $\phi(\xi)=1$ when $|\xi|\leq 1$. Then taking the Fourier transform to the integral equation \eqref{inteequ}, we have
\begin{equation}\label{weieq11.0}
\widehat{u}(\xi,\eta,t)\phi(\xi)=e^{i\omega(\xi,\eta)t}\widehat{u_0}(\xi,\eta)\phi(\xi)-\int_0^t e^{i\omega(\xi,\eta)(t-t')} \widehat{uu_x}(\xi,\eta,t')\phi(\xi)\, dt',
\end{equation}
where, recalling \eqref{lieareqsym},  $\omega(\xi,\eta)=\sign(\xi)+\sign(\xi)\xi^2\mp \sign(\xi)\eta^2$. 
\begin{claim}\label{claim1} 
Let $0<\epsilon \ll 1$ Then it holds
\begin{equation}\label{eqsharp11.1}
J_{\xi}^{1/2+\epsilon}\big(\int_0^t e^{i\omega(\xi,\eta)(t-t')} \widehat{uu_x}(\xi,\eta,t')\phi(\xi)\, dt' \big)\in L^{\infty}([0,T];L^{2}(\mathbb{R}^2)).
\end{equation} 
\end{claim}
Let us assume for the moment that Claim \ref{claim1} holds, then
\begin{equation}\label{weieq11.1}
    J_{\xi}^{1/2}\big(\widehat{u}(\xi,\eta,t)\phi(\xi) \big)\in L^2(\mathbb{R}^2) \hspace{0.2cm} \text{ if and only if } \hspace{0.2cm} J_{\xi}^{1/2}\big( e^{i\omega(\xi,\eta)t}\widehat{u_0}(\xi,\eta)\phi(\xi)\big) \in L^2(\mathbb{R}^2).
\end{equation}
We first notice that since $u_0\in L^{2}(|x|^{1^{+}}\, dx dy)$, Fubini's theorem and Sobolev embedding on the $\xi$-variable determines that $\widehat{u_0}(\xi,\eta)$ is continuous in $\xi$ for almost every $\eta \in \mathbb{R}$. Therefore, given that \eqref{weieq11.1} holds at $t=t_2$, Fubini's theorem shows that $ J_{\xi}^{1/2}\big( e^{i\omega(\xi,\eta)t_2}\widehat{u_0}(\xi,\eta)\phi(\xi)\big) \in L^2(\mathbb{R})$ for almost every $\eta \in \mathbb{R}$, then an application of Proposition \ref{optm1} imposes that $\widehat{u_0}(0,\eta,t)=0$ for almost every $\eta$. From this fact, the integral equation \eqref{weieq11.0} and Claim \ref{claim1}, we deduce Theorem \ref{sharpdecay}, that is, $\widehat{u}(0,\eta,t)=0$ for all $t\geq 0$ and almost every $\eta$.

\begin{proof}[Proof of Claim \ref{claim1}]  In virtue of Theorem \ref{TheoSteDer},
\begin{equation}\label{weieq12}
    \begin{aligned}
    \|J_{\xi}^{1/2+\epsilon}\big(&\int_0^t e^{i\omega(\xi,\eta)(t-t')} \widehat{uu_x}(\xi,\eta,t')\phi(\xi)\, dt' \big)\|_{L^2_{\xi\eta}} \\
    &\lesssim \int_{0}^T \|\phi\|_{L^{\infty}_{\xi}}\|\widehat{uu_x}(t')\|_{L^2_{\xi\eta}}\, dt'+\int_0^T \|\mathcal{D}_{\xi}^{1/2+\epsilon}\big(e^{i\omega(\xi,\eta)(t-t')} \widehat{uu_x}(t')\phi(\xi)\big)\|_{L^2_{\xi\eta}}\, dt'.
    \end{aligned}
\end{equation}

To estimate the r.h.s of the last inequality, we decompose $\omega(\xi,\eta)=\omega_1(\xi,\eta)+\omega_2(\xi,\eta)$ where $\omega_1(\xi,\eta):=\sign(\xi)\mp \sign(\xi)\eta^2$. Then, writing $\widehat{uu_x}(\xi)=i\xi \widehat{u^2}(\xi)$ and using \eqref{prelimneq} and Proposition \ref{propprelimneq2}, 
\begin{equation}\label{weieq13}
\begin{aligned}
\|&\mathcal{D}_{\xi}^{1/2+\epsilon}\big(e^{i\omega(\xi,\eta)(t-t')} \widehat{uu_x}(\xi,\eta,t)\phi(\xi)\big)\|_{L^2_{\xi\eta}} \\
&\lesssim \|\mathcal{D}^{1/2+\epsilon}_{\xi}(e^{i\omega_1(\xi,\eta)(t-t')})\widehat{uu_x}\phi(\xi)\|_{L^2_{\xi\eta}}+\|\mathcal{D}^{1/2+\epsilon}_{\xi}(e^{i\omega_2(\xi,\eta)(t-t')})\widehat{uu_x}\phi(\xi)\|_{L^2_{\xi\eta}}+\|\mathcal{D}^{1/2+\epsilon}_{\xi}(\widehat{uu_x}\phi(\xi))\|_{L^2_{\xi\eta}} \\
&\lesssim_T \big(\||\xi|^{-1/2-\epsilon}\widehat{uu_x}\|_{L^2_{\xi\eta}}+\|\widehat{uu_x}\|_{L^2_{\xi\eta}}+\||\xi|^{1/2+\epsilon}\widehat{uu_x}\|_{L^2_{\xi\eta}}\big)\|\phi\|_{L^{\infty}_{\xi}}+\|\mathcal{D}^{1/2+\epsilon}_{\xi}(\xi\phi)\widehat{u^2}\|_{L^2_{\xi\eta}} +\|\xi\phi \mathcal{D}_{\xi}^{1/2+\epsilon}(\widehat{u^2})\|_{L^2_{\xi\eta}}\\
&\lesssim_T \|J_x^{1/2-\epsilon}(u^2)\|_{L^2_{xy}}+\|uu_x\|_{L^2_{xy}}+\|J_x^{1/2+\epsilon}(uu_x)\|_{L^2_{xy}}+\|\langle x\rangle^{1/2+\epsilon} u^2\|_{L^2_{xy}} \\
&\lesssim_T  (\|u\|_{L^{\infty}_{xy}}+\|\partial_x u\|_{L^{\infty}_{xy}})\|J^{3/2+\epsilon}_xu\|_{L^2_{xy}}+\|\langle x \rangle^{1/4+\epsilon/2}u \|_{L^4_{xy}}^2,
\end{aligned}
\end{equation}
where the last line is obtained by \eqref{eqfraLR}. We employ Sobolev's embedding and complex interpolation \eqref{prelimneq3} to deduce
\begin{equation}\label{weieq14}
\begin{aligned}
  \|\langle x \rangle^{1/4+\epsilon/2}u\|_{L^4_{xy}} \lesssim  \|\langle |(x,y)| \rangle^{1/4+\epsilon/2}u\|_{L^4_{xy}}&\lesssim \|J^{1/2}\big(\langle |(x,y)| \rangle^{1/4+\epsilon/2}u\big)\|_{L^2_{xy}} \\ &\lesssim \|\langle |(x,y)| \rangle^{r_1}u\|_{L^2_{xy}}^{(1+2\epsilon)/4r_1}  \|J^su\|_{L^2_{xy}}^{(4r_1-1-2\epsilon)/4r_1},
\end{aligned}
\end{equation}
where $s\geq \max\{ \frac{ 2r_1}{(4r_1-1)^{-}},r_2\}$. Hence, \eqref{weieq12}, \eqref{weieq13} and \eqref{weieq14} yield
\begin{equation*}
    \begin{aligned}
    \|J_{\xi}^{1/2+\epsilon}\big(&\int_0^t e^{i\omega(\xi,\eta)(t-t')} \widehat{uu_x}(\xi,\eta,t')\,\phi(\xi) dt' \big)\|_{L^2_{\xi\eta}} \\
    &\lesssim_T (1+\|u\|_{L^1_TL^{\infty}_{xy}}+ \|\partial_xu\|_{L^1_TL^{\infty}_{xy}})(1+\|u\|_{L^{\infty}_T H^s}+\|\langle (x,y)\rangle^{r_1}u\|_{L^{\infty}_{xy}L^2_{xy}})^2.
    \end{aligned}
\end{equation*}
This completes the proof of Claim \ref{claim1}.
\end{proof}


\subsection{Proof of Theorem \ref{sharpdecay1}}

Here we assume that $u\in C([0,T];Z_{s,r_1,r_2}(\mathbb{R}^2))$, $s> \max\{3,r_2\}$, $r_2\geq r_1=3/2-\epsilon$, where $0<\epsilon<3/20$. Without loss of generality, we let $t_1= 0< t_2$, that is, $u_0\in Z_{s,(3/2)^{+},r_2}(\mathbb{R}^2)$ and $u(\cdot,t_2)\in Z_{s,3/2,r_2}(\mathbb{R}^2)$. Taking the Fourier transform in \eqref{inteequ} and differentiating on the $\xi$ variable yield 
\begin{equation}\label{weieq14.1}
\begin{aligned}
  \frac{\partial}{\partial \xi} \widehat{u}(\xi,\eta,t)=&2it|\xi|e^{i\omega(\xi,\eta)t}\widehat{u_0}(\xi,\eta)+e^{i\omega(\xi,\eta)t}\partial_{\xi}\widehat{u_0}(\xi,\eta)-2i\int_0^t e^{i\omega(\xi,\eta)(t-t')}(t-t')|\xi|\widehat{uu_x}(\xi,\eta,t')\, dt' \\
  &-\frac{i}{2}\int_0^t e^{i\omega(\xi,\eta)(t-t')}\widehat{u^2}(\xi,\eta,t')\, dt'-\frac{i}{2}\int_0^t e^{i\omega(\xi,\eta)(t-t')} \xi \, \partial_{\xi}\widehat{u^2}(\xi,\eta,t')\, dt' ,
\end{aligned}
\end{equation}
where $\omega(\xi,\eta)=\sign(\xi)+\sign(\xi)\xi^2\mp \sign(\xi)\eta^2$, we have used that $\widehat{u}_0(0,\eta)=\widehat{uu_x}(0,\eta)=0$ and the identity
\begin{equation*}
    \partial_{\xi}e^{i\omega(\xi,\eta)t} =2i \sin((1\mp \eta^2)t)\delta_{0}^{\xi}+2it|\xi|e^{i\omega(\xi,\eta)t}, 
\end{equation*}
setting $(\delta^{\xi}_0 \phi) (\xi,\eta)=\phi(0,\eta)$. 

\begin{claim}\label{claim2}
It holds that
\begin{equation*}
\begin{aligned}
J^{1/2}_{\xi}\Big(t|\xi|e^{i\omega(\xi,\eta)t}\widehat{u_0}(\xi,\eta)-&\int_0^t e^{i\omega(\xi,\eta)(t-t')}(t-t')|\xi|\widehat{uu_x}(\xi,\eta,t')\, dt' \\
&-\frac{1}{4}\int_0^t e^{i\omega(\xi,\eta)(t-t')} \xi \, \partial_{\xi}\widehat{u^2}(\xi,\eta,t')\, dt'\Big) \in L^{\infty}([0,T];L^2(\mathbb{R}^2)).
\end{aligned}
\end{equation*} 
\end{claim}

\begin{proof}
We first deal with the term determined by the homogeneous part of the integral equation. We use Theorem \ref{TheoSteDer}, \eqref{prelimneq} and Proposition \ref{propprelimneq2} to find
\begin{equation}\label{weieq15}
\begin{aligned}
    \|J_{\xi}^{1/2}(|\xi|e^{i\omega(\xi,\eta)t}\widehat{u_0})\|_{L^{2}_{\xi\eta}} &\lesssim     \||\xi|\widehat{u_0}\|_{L^{2}_{\xi\eta}}+    \|\mathcal{D}_{\xi}^{1/2}(|\xi|e^{i\omega(\xi,\eta)t}\widehat{u_0})\|_{L^{2}_{\xi\eta}} \\
    &\lesssim  \||\xi|\widehat{u_0}\|_{L^{2}_{\xi\eta}}+\||\xi|^{1/2}\widehat{u_0}\|_{L^{2}_{\xi\eta}}+\|\mathcal{D}_{\xi}^{1/2}(|\xi|\widehat{u_0})\|_{L^{2}_{\xi\eta}}.
\end{aligned}
\end{equation}
To estimate the last term on the r.h.s of the above expression, we use \eqref{prelimneq}, \eqref{prelimneq1}, Plancherel's identity and Young's inequality to get
\begin{equation}\label{weieq16}
\begin{aligned}
\|\mathcal{D}_{\xi}^{1/2}(|\xi| \widehat{u_0})\|_{L^{2}_{\xi\eta}}=\|\mathcal{D}_{\xi}^{1/2}(\frac{|\xi|}{\langle \xi \rangle} \langle \xi \rangle \widehat{u_0})\|_{L^{2}_{\xi\eta}} \lesssim \|\|J_{\xi}^{1/2}(\langle\xi \rangle \widehat{u_0})\|_{L^2_{\xi}}\|_{L^2_{\eta}} &\lesssim \|\|\langle \xi \rangle^{3/2} \widehat{u_0}\|_{L^2_{\xi}}^{2/3}\|J^{3/2}_{\xi} \widehat{u_0}\|_{L^2_{\xi}}^{1/3}\|_{L^2_{\eta}} \\
& \lesssim \|J_x^{3/2}u_0\|_{L^2_{xy}}+\|\langle x \rangle^{3/2}u_0\|_{L^2_{xy}},
\end{aligned}
\end{equation}
where we have also used \eqref{prelimneq3} with $\alpha=1/3$ and $a=b=3/2$. Gathering \eqref{weieq15} and \eqref{weieq16}, we complete the analysis of $ \|J_{\xi}^{1/2}(|\xi|e^{i\omega(\xi,\eta)t}\widehat{u_0})\|_{L^{2}_{\xi\eta}}$. Next, we shall prove that
\begin{equation}\label{weieq17}
    uu_x \in L^{\infty}([0,T];H_x^{3/2}(\mathbb{R}^2))\cap L^{\infty}([0,T]; L^2(|x|^3 dxdy)).
\end{equation}
where $H_x^s(\mathbb{R}^2)$ is defined according to the norm $\|f\|_{H_x^s}=\|J_x^{s}f\|_{L^2}$. Once this has been established, following the reasoning in \eqref{weieq15} and \eqref{weieq16}, it will follow
\begin{equation*}
  J^{1/2}_{\xi}\big(\int_0^t e^{i\omega(\xi,\eta)(t-t')}(t-t')|\xi|\widehat{uu_x}(\xi,\eta,t')\, dt'\big) \in L^{\infty}([0,T];L^2(\mathbb{R}^2)).
\end{equation*}
Indeed, \eqref{eqfraLR} and Sobolev's embedding show $\|uu_x\|_{H^{3/2}_x}\lesssim \|u\|_{H^s}^2$, whenever $s\geq 5/2$. Now, complex interpolation \eqref{prelimneq3}, Young's inequality and Sobolev's embedding determine 
\begin{equation}\label{weieq18}
    \begin{aligned}
    \|\langle x \rangle^{3/2} u u_x \|_{L^2_{xy}} &\lesssim \|\langle x \rangle^{1/2} u^2\|_{L^2_{xy}}+\|J_x(\langle x \rangle^{3/2} u^2)\|_{L^2_{xy}} \\
    &\lesssim \|u\|_{L^{\infty}_{xy}}\|\langle x\rangle^{1/2}u\|_{L^2_{xy}}+\|\|\langle x\rangle^{9/4} u^2\|_{L^2_x}^{2/3}\|J^3_x(u^2)\|_{L^2_x}^{1/3}\|_{L^2_y} \\
    &\lesssim \|J^{3}u\|_{L^2}\|\langle x\rangle^{1/2}u\|_{L^2_{xy}}+ \|\langle x\rangle^{9/4} u^2\|_{L^2_{xy}}+\|J^3_x(u^2)\|_{L^2_{xy}}.
    \end{aligned}
\end{equation}
Since $H^3(\mathbb{R}^2)$ is a Banach algebra, $\|J^3_x(u^2)\|_{L^2_{xy}}\lesssim \|J^3(u^2)\|_{L^2_{xy}} \lesssim \|u\|_{H^3}^2$, so it remains to derive a bound for the second term on the right hand side of equation \eqref{weieq18}. Let $0<\epsilon <3/20$, applying Sobolev's embedding and complex interpolation we find
\begin{equation}\label{weieq19}
\begin{aligned}
    \|\langle x\rangle^{9/4} u^2\|_{L^2_{xy}}\lesssim\|\langle |(x,y)|\rangle^{9/8} u\|_{L^4_{xy}}^2 &\lesssim \|J^{1/2}(\langle |(x,y)|\rangle^{9/8} u)\|_{L^2_{xy}}^2 \\
    &\lesssim\|\langle |(x,y)| \rangle^{3/2-\epsilon}u\|_{L^2_{xy}}^{\frac{18}{12-8\epsilon}}\|J^{\frac{6-4\epsilon}{3-8\epsilon}}u\|^{\frac{6-16\epsilon}{12-8\epsilon}}_{L^2_{xy}}.
\end{aligned}
\end{equation}
Notice that since $0<\epsilon <3/20$, $\|J^{\frac{6-4\epsilon}{3-8\epsilon}}u\|_{L^2_{xy}}\leq \|J^3u\|_{L^2_{xy}}$. Plugging \eqref{weieq19} in \eqref{weieq18}, we complete the deduction of \eqref{weieq17}.  
To prove the remaining estimate, i.e., 
\begin{equation}\label{remainEst}
    J_{\xi}^{1/2}\big(\int_0^t e^{i\omega(\xi,\eta)(t-t')} \xi \, \partial_{\xi}\widehat{u^2}(\xi,\eta,t')\, dt'\big) \in L^{\infty}([0,T];L^2(\mathbb{R}^2)),
\end{equation}
we write $\frac{\partial}{\partial \xi}\widehat{u^2}=\widehat{-ixu^2}$, then according to the arguments in \eqref{weieq15} and \eqref{weieq16}, to deduce \eqref{remainEst}, it is enough to show
\begin{equation}\label{weieq20}
    xu^2 \in L^{\infty}([0,T];H_x^{3/2}(\mathbb{R}^2))\cap L^{\infty}([0,T]; L^2(|x|^3 dxdy)).
\end{equation}
To this aim, after some computations applying Theorem \ref{TheoSteDer} and property \eqref{prelimneq}, we employ complex interpolation and Young's inequality to show
\begin{equation*}
    \begin{aligned}
    \|J_x^{1/2}(xu^2)\|_{L^2_{xy}}&\lesssim \|xu^2\|_{L^2_{xy}}+\|J_x^{1/2}(u^2)\|_{L^2_{xy}}+\|J_x^{3/2}(\langle x \rangle u^2)\|_{L^2_{xy}} \\
    &\lesssim \|u\|_{L^{\infty}_{xy}}\|\langle x \rangle u\|_{L^2_{xy}}+\|u\|_{L^{\infty}_{xy}}\|J_x^{1/2}u\|_{L^2_{xy}}+\|\|\langle x\rangle^{9/4}u^2\|_{L^2_x}^{4/9}\|J_x^{27/10}( u^2)\|_{L^2_{x}}^{5/9}\|_{L^2_y} \\
    & \lesssim \|J^3u\|_{L^2_{xy}}\|\langle x \rangle u\|_{L^2_{xy}}+\|J^{3}u\|_{L^2_{xy}}^2+\|\langle x\rangle^{9/4}u^2\|_{L^2_{xy}}.
    \end{aligned}
\end{equation*}
Then, \eqref{weieq19} allows us to conclude that $ xu^2 \in L^{\infty}([0,T];H_x^{3/2}(\mathbb{R}^2))$. Finally, since $u\in C([0,T];H^s(\mathbb{R}^2))$, $s> \max\{3,r_2\} $, there exists some $0<\delta<1$ such that $3+\delta<s$, then we have
\begin{equation}\label{weieq21}
\begin{aligned}
    \|\langle x \rangle^{3/2} x u^2\|_{L^2_{xy}} \lesssim \|\langle x \rangle^{5/4} u\|_{L^4_{xy}}^{2} &\lesssim \|J^{1/2}(\langle (x,y) \rangle^{5/4} u)\|_{L^2_{xy}}^{2} \\
    & \lesssim \|\langle (x,y) \rangle^{3/2-\epsilon} u\|_{L^2_{xy}}^{\frac{10}{6-4\epsilon}}\|J^{\frac{3-2\epsilon}{1-4\epsilon}}u\|_{L^2_{xy}}^{\frac{2-8\epsilon}{6-4\epsilon}}.
\end{aligned}
\end{equation}
Now, taking $0<\epsilon\ll 1$ such that $\frac{3-2\epsilon}{1-4\epsilon}\leq 3+\delta<s$, \eqref{weieq21} shows that $xu^2\in L^{\infty}([0,T]; L^2(|x|^3 dxdy))$. This in turn confirms the validity of \eqref{weieq20}.
\end{proof}
Consequently, from \eqref{weieq14.1} and Claim \ref{claim2}, it follows:
\begin{equation}\label{weieq22}
\begin{aligned}
    J^{1/2}_{\xi}\partial_{\xi}\widehat{u}(\xi,\eta,t) \in L^{2}(\mathbb{R}^2) \, & \text{ if and only if } \, \\
    & J_{\xi}^{1/2}\Big(e^{i\omega(\xi,\eta)t}\partial_{ \xi}\widehat{u_0}(\xi,\eta)-\frac{i}{2}\int_0^t e^{i\omega(\xi,\eta)(t-t')} \widehat{u^2}(\xi,\eta,t')\, dt'\Big) \in L^2(\mathbb{R}^2). 
\end{aligned} 
\end{equation}
Now, since \eqref{weieq21} establishes that $\widehat{u^2} \in H^{1^{+}}(\mathbb{R}^2)$,  Sobolev's embedding determines that $\widehat{u^2}$ can be regarded as a continuous function on the $\xi$ and $\eta$ variables. Additionally, since $\partial_\xi \widehat{u_0}\in H^{(1/2)^{+}}_{\xi}(\mathbb{R}^2)$, Fubinni's theorem and Sobolev's embedding shows that $\partial_{\xi}\widehat{u_0}(\xi,\eta)$ is continuous in $\xi$ for almost every $\eta \in \mathbb{R}$. Given that \eqref{weieq22} holds at $t=t_2$, according to the preceding discussions and Proposition \ref{optm1}, we deduce
\begin{equation*}
    \begin{aligned}
    &e^{i(1\mp\eta^2)t_2}\partial _{\xi}\widehat{u_0}(0,\eta)-\frac{i}{2} \int_{0}^{t_2} e^{i(1\mp \eta^2)(t_2-t')}\widehat{u^2}(0,\eta,t')\, dt' \\
    &= e^{-i(1\mp\eta^2)t_2}\partial_{ \xi}\widehat{u_0}(0,\eta)-\frac{i}{2} \int_0^{t_2} e^{-i(1\mp \eta^2)(t_2-t')}\widehat{u^2}(0,\eta,t')\, dt'
    \end{aligned}
\end{equation*}
so that
\begin{equation}\label{weieq23}
    \begin{aligned}
    &2i \sin((1\mp \eta^2)t_2)\partial _{\xi}\widehat{u_0}(0,\eta)
    = - \int_0^{t_2} \sin((1\mp\eta^2)(t_2-t'))\widehat{u^2}(0,\eta,t')\, dt',
    \end{aligned}
\end{equation}
for almost every $\eta\in \mathbb{R}$. This completes the deduction of identity \eqref{identw1}. Now,  recalling that the quantity  $M(u)=\|u(t)\|_{L^2}$ is invariant for solution of the equation in \eqref{EQBO}, and that $\eta \mapsto \widehat{u^2}(0,\eta,t)$ determines a continuous map,  we let $\eta\to 0$ in \eqref{weieq23} to find
\begin{equation}\label{weieq24}
\begin{aligned}
    J^{1/2}_{\xi}\partial_{\xi}\widehat{u}(\xi,\eta,t_2) \in L^{2}(\mathbb{R}^2) \, & \text{ and }\, \eta \mapsto \partial_{\xi} \widehat{u}_0(0,\eta) \, \text{ continuous at the origin imply } \\
   &2i\sin(t_2)\partial_{\xi}\widehat{u_0}(0,0)=(\cos(t_2)-1)\|u_0\|_{L^2_{xy}}^2.
\end{aligned} 
\end{equation}
Therefore, in the case $u_0 \in Z_{s,2^{+},2^{+}}(\mathbb{R}^2)$, \eqref{weieq24} yields identity  \eqref{identw2}.


\section{Proof of Theorem \ref{SCeqThm}}\label{SecSh}

This section is aimed to briefly indicate the modifications needed to prove Theorem \ref{SCeqThm}. We first recall that the IVP \eqref{EQSH} is LWP in the space $H^s(\mathbb{R}^2)$, $s>3/2$ by the results established in \cite{Paisas1}. To prove well-posedness in the space $\widetilde{X}^s(\mathbb{R}^2)$ determined by the norm
\begin{equation*}
    \|f\|_{\widetilde{X}^s}=\|J^s_xf\|_{L^2_{xy}}+\|D_x^{-1/2}\partial_yf\|_{L^2_{xy}},
\end{equation*}
the key ingredient is the refined Strichartz estimate deduced in \cite{Paisas1}:
\begin{lemma}
The results of Lemma \ref{refinStri} hold for solutions of the IVP \eqref{EQSH}.
\end{lemma}
Once the above lemma has been established, the proof of LWP in $\widetilde{X}^s(\mathbb{R}^2)$ follows the same line of arguments leading to the conclusion of Theorem \ref{Improwellp}. Actually, this case does not require to estimate the norm $\|D_x^{-1/2}u\|_{L^2_{xy}}$, which slightly simplifies our arguments. We emphasize that Lemma \ref{lemaexismooth} assures the existence of solutions of the IVP \eqref{EQSH} in the space $\widetilde{X}^{\infty}(\mathbb{R}^2)=\bigcap_{s\geq 0}\widetilde{X}^{s}(\mathbb{R}^2)$. Consequently, it follows that \eqref{EQSH} is LWP in $\widetilde{X}^s(\mathbb{R}^2)$, $s>3/2$. 

On the other hand, setting
\begin{equation*}
    \widetilde{\omega}(\xi,\eta)=\sign(\xi)\xi^2+\sign(\xi)\eta^2,
\end{equation*}
 the resonant function determined by the equation in \eqref{EQSH} is given by 
\begin{equation*}
    \widetilde{\Omega}(\xi_1,\eta_1,\xi_2,\eta_2)=\widetilde{\omega}(\xi_1+\xi_2,\eta_1+\eta_2)-\widetilde{\omega}(\xi_1,\eta_2)-\widetilde{\omega}(\xi_2,\eta_2). 
\end{equation*}
Then,  it is not difficult to see:
\begin{prop}
The results in Proposition \ref{lembilinEST} are valid replacing the set $D_{N,L}$ by $$
    \widetilde{D}_{N,L}=\left\{(m,n,\tau)\in \mathbb{Z}^2\times \mathbb{R}: |(m,n)|\in I_{N} \text{ and } |\tau-\widetilde{\omega}(m,n)|\leq L \right\},$$ 
    whenever $N,L\in \mathbb{D}$.
\end{prop}
This in turn allows us to follow the same reasoning leading to the deduction of Theorem \ref{LocalwellTorus} to derive that the IVP \eqref{EQSH} is LWP in $H^s(\mathbb{T}^2)$, $s>3/2$.

Concerning well-posedness in weighted spaces, here we replace equation \eqref{HEQBO} by \begin{equation*}
    \partial_t \mathcal{H}_x u+\partial_x^2 u +\partial_y^2u+\mathcal{H}_x(u\partial_x u)=0.
\end{equation*}
Then, employing the above identity, we can adapt the arguments in the proof of Theorem \ref{localweigh} to obtain the same well-posedness conclusion in anisotropic spaces for the equation in \eqref{EQSH}. Besides, the arguments in Proposition \ref{propprelimneq2} show
\begin{equation*}
    \mathcal{D}^b(e^{i\sign(x)\eta^2t}) \lesssim |x|^{-b}, \hspace{0.2cm} x\in \mathbb{R}\setminus\{0\},
\end{equation*}
whenever $b \in (0,1)$ fixed and for all $\eta \in \mathbb{R}$. Thus, the previous estimate allows us to follow the same arguments in the proof of Theorems \ref{sharpdecay} and \ref{sharpdecay1} to obtain the same conclusions for the IVP \eqref{EQSH}. However, instead of \eqref{identw1} we get
\begin{equation*}
    2i\sin(\eta^2(t_2-t_1))\partial_{\xi}\widehat{u}(0,\eta,t_1)=-\int_{t_1}^{t_2}\sin(\eta^2(t_2-t'))\widehat{u^2}(0,\eta,t')\, dt',
\end{equation*}
for almost $\eta \in \mathbb{R}$. This encloses the discussion leading to the deduction of Theorem \ref{SCeqThm}.

\section{Appendix: proof of Proposition \ref{CalderonCom}}

We first require to further decompose the lower frequency operator $N=1$ introduced in \eqref{proje1}. Thus, for all dyadic number $N$, let $\varphi_N(\xi)=\psi_1(\xi/N)-\psi_1(2\xi/N)$, and we denote by $P^x_N$ the associated operator defined as in \eqref{proje1}, i.e., the operator determined by the $L^2$-multiplier by the function $\varphi_N$. 

We shall use the following result.
\begin{lemma}\label{lemmaApped1}
Let $\phi\in C^{\infty}_c(\mathbb{R}^d)$ such that $\supp(\phi)\subset\{|\xi|\leq R\}$ for some $R>0$. Consider the operator $P^{\phi}f$ determined by $\widehat{P^{\phi}f}(\xi)=\phi(\xi)\widehat{f}(\xi)$. Then 
\begin{equation}\label{eqlemm0}
 \sup_{z\in \mathbb{R}^d}\frac{|P^{\phi}f(x-z)|}{(1+R|z|)^d} \lesssim \mathcal{M}(f)(x).
\end{equation}
In the above, $\mathcal{M}(\cdot)$ denotes the usual Hardy-Littlewood maximal function. 
\end{lemma}

Additionally, we will apply the  following particular case of the Fefferman-Stein inequality:
\begin{lemma}(\cite{FefermStein})\label{FSte} Let $f=(f_j)_{j=1}^{\infty}$ be a sequence of locally integrable functions in $\mathbb{R}^d$. Let $1<p<\infty$. Then
\begin{equation}\label{eqfeffSteq1}
\|(\mathcal{M}f_j)_{l_j^2}\|_{L^p} \lesssim \|(f_j)_{l_j^2}\|_{L^p}. 
\end{equation}
\end{lemma}
Now, we are in the condition to deduce Proposition \ref{CalderonCom}.

\begin{proof}[Proof of Proposition \ref{CalderonCom}]

When $\beta=1$ on the l.h.s of \eqref{Comwell}, by writing $D_x=\mathcal{H}_x\partial_x$  and using that $\mathcal{H}_x$ determines a bounded operator in $L^{p}$, we have that \eqref{Comwell} follows from Proposition \ref{CalderonComGU}. 

We will assume that $0<\alpha,\beta<1$ with $\alpha+\beta=1$. We write
\begin{equation}\label{equapeend1}
    D_x^{\alpha}[\mathcal{H}_x,g]D_x^{\beta} f(x)=-i\int |\xi_1+\xi_2|^{\alpha}|\xi_2|^{\beta}\big(\sign(\xi_1+\xi_2)-\sign(\xi_2)\big)\widehat{g}(\xi_1)\widehat{f}(\xi_2)e^{ix\cdot(\xi_1+\xi_2)} \, d\xi_1 d\xi_2,
\end{equation}
then neglecting the null measure sets where $\xi_1+\xi_2=0$ or $\xi_2=0$, we observe that the integral in \eqref{equapeend1} is not null only when $(\xi_1+\xi_2)\xi_2<0$, in order words, when $|\xi_2|<|\xi_1|$. Thus, by Bony's paraproduct decomposition we find
\begin{equation*}
    \begin{aligned}
 D_x^{\alpha}[\mathcal{H}_x,g]D_x^{\beta} f=&\mathcal{H}_x\big(\sum_{N>0} D^{\alpha}(P_N^x g P_{\ll N}^x D_x^{\beta}f ) \big)-\sum_{N>0} D^{\alpha}(P_N^x g P_{\ll N}^x \mathcal{H}_x D_x^{\beta}f)\\
 &+\mathcal{H}_x\big(\sum_{N>0} D^{\alpha}(P_N^x g \widetilde{P}_{N}^x D_x^{\beta}f ) \big) -\sum_{N>0} D^{\alpha}(P_N^x g \widetilde{P}^x_N\mathcal{H}_x D_x^{\beta}f) \\
 =:& \mathcal{A}_1+\mathcal{A}_2+\mathcal{A}_3+\mathcal{A}_4,
\end{aligned}
\end{equation*}
where $P_{\ll N}^xf=\sum_{M\ll N}P^x_{M}f$ and $\widetilde{P}_N^xf=\sum_{M\sim N}P^x_{M}f$. Now, we proceed to estimate each of the factors $\mathcal{A}_j$, $j=1,\dots,4$. Since $\alpha+\beta=1$, $\beta>0$, and the Hilbert transform determines a bounded operator in $L^p$, by the Littlewood-Paley inequality and support considerations we have
\begin{equation}\label{equapeend2}
\begin{aligned}
         \|\mathcal{A}_1\|_{L^p}\lesssim \Big\|\big(P_{M}^x(\sum_{N>0}D^{\alpha}(P_N^x g P_{\ll N}^x D_x^{\beta}f ))\big)_{l^2_{M}}\Big\|_{L^p}&\lesssim \Big\|\big(\sum_{N\sim M}D^{\alpha}P_M^x(P_N^x g P_{\ll N}^x D_x^{\beta}f \big)_{l^2_{M}}\Big\|_{L^p}\\
         &\lesssim \sum_{L\sim 1} \Big\|\big(\overline{P}^x_{LN}(\overline{P}_N^x \partial_x g N^{-\beta} P_{\ll N}^x D_x^{\beta}f \big)_{l^2_{N}}\Big\|_{L^p},
\end{aligned}
\end{equation}
for some adapted projections $\overline{P}_N^x$ supported in frequency on the set $|\xi|\sim N$, and with $L\sim 1$ dyadic. Now, by employing Lemma \ref{lemmaApped1}, we deduce
\begin{equation*}
    \begin{aligned}
      |\overline{P}^x_{LN}(\overline{P}_N^x \partial_x g N^{-\beta} P_{\ll N}^x D_x^{\beta}f)(x)| \lesssim \mathcal{M}(\overline{P}_N^x\partial_x g N^{-\beta} P_{\ll N}^x D_x^{\beta}f)(x).
    \end{aligned}
\end{equation*}
Inserting the above expression on the r.h.s of \eqref{equapeend2}, applying \eqref{eqfeffSteq1} and Lemma \ref{lemmaApped1}, we get
\begin{equation}\label{equapeend3}
    \begin{aligned}
    \|\mathcal{A}_1\|_{L^p}\lesssim \|(\overline{P}_{N}^x \partial_x g N^{-\beta}P_{\ll N}^x D_x^{\beta}f)_{l^2_N}\|_{L^p}&\lesssim \|\mathcal{M}(\partial_x g)(N^{-\beta}P_{\ll N}^x D_x^{\beta}f)_{l^2_N}\|_{L^p}\\
    &\lesssim \|\partial_x g\|_{L^{\infty}}\|(N^{-\beta}P_{\ll N}^x D_x^{\beta}f)_{l^2_N}\|_{L^p}.
    \end{aligned}
\end{equation}
To estimate the preceding inequality, we write $P_N^x=\overline{P}_N^xP_N^x$, then employing Lemma \ref{lemmaApped1}, it follows
\begin{equation*}
    |N^{-\beta}P^{x}_{\ll N}D^{\beta}_xf(x)|\leq N^{-\beta}\sum_{M\ll N}\left|M^{\beta}P_{M}^xf(x)\right|\lesssim \sum_{1\ll L} L^{-\beta} \mathcal{M}(P_{N/L}^x f)(x),
\end{equation*}
so that
\begin{equation}\label{equapeend4}
    (N^{-\beta}P_{\ll N}^x D_x^{\beta}f)_{l^2_N}\lesssim (\mathcal{M}(P_N^x f))_{l^2_N}.
\end{equation}
Hence, plugging \eqref{equapeend4} in \eqref{equapeend3}, by the Fefferman-Stein inequality and the Littlewood-Paley inequality, we conclude
\begin{equation}\label{equapeend5}
\|\mathcal{A}_1\|_{L^p}\lesssim \|\partial_x g\|_{L^{\infty}}\|f\|_{L^p}. 
\end{equation}
Now, replacing $f$ by $\mathcal{H}_x f$ in the arguments above, we derive the same estimate in \eqref{equapeend5} for the term $\mathcal{A}_2$.

A similar reasoning yields the desired estimate for $\mathcal{A}_3$. Indeed, since $\alpha+\beta=1$, $\alpha>0$, by Littlewood-Paley inequality
\begin{equation*}
    \|\mathcal{A}_3\|_{L^p}\lesssim \Big\|\big(\sum_{N\gtrsim M} M^{\alpha}N^{-\alpha}\overline{P}^x_M(\overline{P}_N^x\partial_x g \overline{\widetilde{P}}_N^x \widetilde{P}_N^x f)\big)_{l^2_{M}}\Big\|_{L^p}.
\end{equation*}
Now, by Lemma \ref{lemmaApped1} it follows
\begin{equation*}
\begin{aligned}
\big(\sum_{N\gtrsim M} M^{\alpha}N^{-\alpha}\overline{P}^x_M(\overline{P}_N^x\partial_x g \overline{\widetilde{P}}_N^x \widetilde{P}_N^x f)\big)_{l^2_M} &\lesssim \big(\sum_{L\gtrsim 1} L^{-\alpha}\mathcal{M}(\overline{P}_{LM}^x\partial_x g \overline{\widetilde{P}}_{LM}^x \widetilde{P}_{LM}^x f) \big)_{l^2_M} \\
&\lesssim \big( \mathcal{M}(\overline{P}_{N}^x\partial_x g \overline{\widetilde{P}}_{N}^x \widetilde{P}_{N}^x f) \big)_{l^2_N}.   
\end{aligned}
\end{equation*}
Thus, the preceding estimates and \eqref{eqfeffSteq1} reveal
\begin{equation*}
    \|\mathcal{A}_3\|_{L^p}\lesssim \|\mathcal{M}(\partial_x g)(\mathcal{M}(P_N^x f)_{l^2_N})\|_{L^p}\lesssim \|\partial_x g\|_{L^{\infty}}\|f\|_{L^{p}}.
\end{equation*}
The estimate for $\mathcal{A}_4$ follows from the same arguments employed to analyze $\mathcal{A}_3$. The proof of Proposition \ref{CalderonCom} is complete.

\end{proof}

\subsection*{Acknowledgements}

This work was supported by CNPq Brazil. The author wishes to express his gratitude to Prof. Felipe Linares for bringing this problem to his attention and for the valuable suggestions regarding the manuscript.


\bibliographystyle{abbrv}
{\small  \bibliography{bibli}}

\begin{thebibliography}{10}

\bibitem{AbBOnaFellSaut}
L.~Abdelouhab, J.~Bona, M.~Felland, and J.-C. Saut.
\newblock {Nonlocal models for nonlinear, dispersive waves}.
\newblock {\em Physica D: Nonlinear Phenomena}, 40(3):360--392, 1989.

\bibitem{AkersMile}
B.~Akers and P.~A. Milewski.
\newblock {A Model Equation for Wavepacket Solitary Waves Arising from
  Capillary-Gravity Flows}.
\newblock {\em Studies in Applied Mathematics}, 122(3):249--274, 2009.

\bibitem{Biello}
J.~Biello and J.~K. Hunter.
\newblock {Nonlinear Hamiltonian waves with constant frequency and surface
  waves on vorticity discontinuities}.
\newblock {\em Communications on Pure and Applied Mathematics}, 63(3):303--336,
  2010.

\bibitem{paisas2}
E.~Bustamante, J.~Jim\'enez, and J.~Mej\'ia.
\newblock {Periodic Cauchy problem for one two-dimensional generalization of
  the Benjamin-Ono equation in Sobolev spaces of low regularity}.
\newblock {\em Nonlinear Analysis}, 188:50--69, 2019.

\bibitem{Paisas1}
E.~Bustamante, J.~Jim\'enez, and J.~Mej\'ia.
\newblock {The Cauchy problem for a family of two-dimensional fractional
  Benjamin-Ono equations}.
\newblock {\em Communications on Pure \& Applied Analysis}, 18(3):1177--1203,
  2019.

\bibitem{DawsonMCPON}
L.~Dawson, H.~McGahagan, and G.~Ponce.
\newblock {On the Decay Properties of Solutions to a Class of Schr\"odinger
  Equations}.
\newblock {\em Proceedings of the American Mathematical Society},
  136(6):2081--2090, 2008.

\bibitem{JavierHarmo}
J.~Duoandikoetxea.
\newblock {\em Fourier Analysis}.
\newblock Crm Proceedings \& Lecture Notes. American Mathematical Soc., 2001.

\bibitem{Omarths}
O.~Duque~G\'omez.
\newblock {\em Sobre una versi\'on bidimensional de la ecuaci\' on Benjamin-Ono
  generalizada}.
\newblock PhD thesis, Universidad Nacional de Colombia-Bogot\'a, 2014.

\bibitem{AminPastor}
A.~Esfahani and A.~Pastor.
\newblock {Ill-posedness results for the (generalized)
  Benjamin-Ono-Zakharov-Kuznetsov equation}.
\newblock {\em Proceedings of the American Mathematical Society},
  139(3):943--956, 2011.

\bibitem{Esfahani2018}
A.~Esfahani and A.~Pastor.
\newblock {Two dimensional solitary waves in shear flows}.
\newblock {\em Calculus of Variations and Partial Differential Equations},
  57(4):102, 2018.

\bibitem{FefermStein}
C.~Fefferman and E.~M. Stein.
\newblock {Some Maximal Inequalities}.
\newblock {\em American Journal of Mathematics}, 93(1):107--115, 1971.

\bibitem{FLinaPonceWeBO}
G.~Fonseca, F.~Linares, and G.~Ponce.
\newblock {The IVP for the Benjamin-Ono equation in weighted Sobolev spaces
  II}.
\newblock {\em Journal of Functional Analysis}, 262(5):2031 -- 2049, 2012.

\bibitem{FLinaPioncedGBO}
G.~Fonseca, F.~Linares, and G.~Ponce.
\newblock {The IVP for the dispersion generalized Benjamin-Ono equation in
  weighted Sobolev spaces}.
\newblock {\em Annales de l'Institut Henri Poincare (C) Non Linear Analysis},
  30(5):763 -- 790, 2013.

\bibitem{FonPO}
G.~Fonseca and G.~Ponce.
\newblock {The IVP for the Benjamin-Ono equation in weighted Sobolev spaces}.
\newblock {\em Journal of Functional Analysis}, 260(2):436--459, 2011.

\bibitem{FRAC}
L.~Grafakos and S.~Oh.
\newblock {The Kato-Ponce Inequality}.
\newblock {\em Communications in Partial Differential Equations},
  39(6):1128--1157, 2014.

\bibitem{GuoTadahiro}
Z.~Guo and T.~Oh.
\newblock {Non-Existence of Solutions for the Periodic Cubic NLS below $L^2$}.
\newblock {\em International Mathematics Research Notices}, 2018(6):1656--1729,
  2018.

\bibitem{HADAC2009}
M.~Hadac, S.~Herr, and H.~Koch.
\newblock {Well-posedness and scattering for the KP-II equation in a critical
  space}.
\newblock {\em Annales de l'Institut Henri Poincare (C) Non Linear Analysis},
  26(3):917--941, 2009.

\bibitem{linaO}
J.~Hickman, F.~Linares, O.~Ria\~no, K.~Rogers, and J.~Wright.
\newblock {On a Higher Dimensional Version of the Benjamin--Ono Equation}.
\newblock {\em SIAM Journal on Mathematical Analysis}, 51(6):4544--4569, 2019.

\bibitem{BuHil}
J.~K. Hunter, M.~Ifrim, D.~Tataru, and T.~K. Wong.
\newblock {Long time solutions for a Burgers-Hilbert equation via a modified
  energy method}.
\newblock {\em Proceedings of the American Mathematical Society},
  143(8):3407--3412, 2015.

\bibitem{IonescuKenigPerioKP}
A.~Ionescu and C.~Kenig.
\newblock {Local and global wellposedness of periodic KP-I equations}.
\newblock {\em Annals of Mathematics Studies}, (163):181--211, 2007.

\bibitem{Kenig}
A.~D. Ionescu and C.~E. Kenig.
\newblock {Global Well-Posedness of the Benjamin-Ono Equation in Low-Regularity
  Spaces}.
\newblock {\em Journal of the American Mathematical Society}, 20(3):753--798,
  2007.

\bibitem{IonescuKeniTata}
A.~D. Ionescu, C.~E. Kenig, and D.~Tataru.
\newblock {Global well-posedness of the KP-I initial-value problem in the
  energy space}.
\newblock {\em Inventiones mathematicae}, 173(2):265--304, 2008.

\bibitem{Iorio}
R.~J. I\'orio.
\newblock {On the Cauchy problem for the Benjamin-Ono equation}.
\newblock {\em Communications in Partial Differential Equations},
  11(10):1031--1081, 1986.

\bibitem{IoNu}
R.~J. I\'orio and W.~V.~L. Nunes.
\newblock {On equations of KP-type}.
\newblock {\em Proceedings of the Royal Society of Edinburgh: Section A
  Mathematics}, 128(4):725--743, 1998.

\bibitem{Ioriobook}
R.~J. Iorio, Jr and V.~d. M.~a. Iorio.
\newblock {\em Fourier Analysis and Partial Differential Equations}.
\newblock Cambridge Studies in Advanced Mathematics. Cambridge University
  Press, 2001.

\bibitem{KP}
T.~Kato and G.~Ponce.
\newblock {Commutator estimates and the Euler and Navier-Stokes equations}.
\newblock {\em Communications on Pure and Applied Mathematics}, 41(7):891--907,
  1988.

\bibitem{KenigKP}
C.~E. Kenig.
\newblock {On the local and global well-posedness theory for the KP-I
  equation}.
\newblock {\em Annales de l'Institut Henri Poincare (C) Non Linear Analysis},
  21(6):827 -- 838, 2004.

\bibitem{KenigKo}
C.~E. Kenig and K.~D. Koenig.
\newblock {On the local well-posedness of the Benjamin-Ono and modified
  Benjamin-Ono equations}.
\newblock {\em Mathematical Research Letters}, 10(6):879--895, 2003.

\bibitem{Dli}
D.~Li.
\newblock {On Kato-Ponce and Fractional Leibniz}.
\newblock {\em Revista Matem\'atica Iberoamericana}, 35(1):23--100, 2019.

\bibitem{perioZK}
F.~Linares, M.~Panthee, T.~Robert, and N.~Tzvetkov.
\newblock {On the periodic Zakharov-Kuznetsov equation}.
\newblock {\em Discrete \& Continuous Dynamical Systems-A}, 39(6):3521--3533,
  2019.

\bibitem{LinarFKP}
F.~Linares, D.~Pilod, and J.-C. Saut.
\newblock {The Cauchy Problem for the Fractional Kadomtsev-Petviashvili
  Equations}.
\newblock {\em SIAM J. Math. Analysis}, 50:3172--3209, 2018.

\bibitem{linaresBook}
F.~Linares and G.~Ponce.
\newblock {\em Introduction to Nonlinear Dispersive Equations}.
\newblock Universitext. Springer New York, 2015.

\bibitem{JulioLi}
J.~Lizarazo.
\newblock {\em El problema de Cauchy de la clase de ecuaciones de dispersi\'on
  generalizada de Benjamin-Ono bidimensionales}.
\newblock PhD thesis, Universidad Nacional de Colombia-Bogot\'a, 2018.

\bibitem{MolinetPeriod}
L.~Molinet.
\newblock {Global well-posedness in $L^2$ for the periodic Benjamin-Ono
  equation}.
\newblock {\em American Journal of Mathematics}, 130(3):635--683, 2008.

\bibitem{molinetPilodBO}
L.~Molinet and D.~Pilod.
\newblock {The Cauchy problem for the Benjamin-Ono equation in $L^2$
  revisited}.
\newblock {\em Anal. PDE}, 5(2):365--395, 2012.

\bibitem{muscalu}
C.~Muscalu, J.~Pipher, T.~Tao, and C.~Thiele.
\newblock {Bi-parameter paraproducts}.
\newblock {\em Acta Math.}, 193(2):269--296, 2004.

\bibitem{NahPonc}
J.~Nahas and G.~Ponce.
\newblock {On the Persistent Properties of Solutions to Semi-Linear
  Schr\"odinger Equation}.
\newblock {\em Communications in Partial Differential Equations},
  34(10):1208--1227, 2009.

\bibitem{PShrira}
D.~E. Pelinovsky and V.~I. Shrira.
\newblock {Collapse transformation for self-focusing solitary waves in
  boundary-layer type shear flows}.
\newblock {\em Physics Letters A}, 206(3):195 -- 202, 1995.

\bibitem{Ponce1991}
G.~Ponce.
\newblock {On the global well-posedness of the Benjamin-Ono equation}.
\newblock {\em {Differential Integral Equations}}, 4(3):527--542, 1991.

\bibitem{RibaVento}
F.~Ribaud and S.~Vento.
\newblock {Local and global well-posedness results for the
  Benjamin-Ono-Zakharov-Kuznetsov equation}.
\newblock {\em Discrete \& Continuous Dynamical Systems-A}, 37(1):449--483,
  2017.

\bibitem{RobertCKP}
T.~Robert.
\newblock {Global well-posedness of partially periodic KP-I equation in the
  energy space and application}.
\newblock {\em Annales de l'Institut Henri Poincare (C) Non Linear Analysis},
  35(7):1773 -- 1826, 2018.

\bibitem{RoberTtFKP}
T.~Robert.
\newblock {On the Cauchy problem for the periodic fifth-order KP-I equation}.
\newblock {\em Differential Integral Equations}, 32(11/12):679--704, 2019.

\bibitem{RobertS}
R.~{Schippa}.
\newblock {On the Cauchy problem for higher dimensional Benjamin-Ono and
  Zakharov-Kuznetsov equations}.
\newblock {\em arXiv e-prints}, page arXiv:1903.02027, 2019.

\bibitem{SCHIPPAShreq}
R.~Schippa.
\newblock {On short-time bilinear Strichartz estimates and applications to the
  Shrira equation}.
\newblock {\em Nonlinear Analysis}, 198:111910, 2020.

\bibitem{SteinThe}
E.~M. Stein.
\newblock {The characterization of functions arising as potentials}.
\newblock {\em Bull. Amer. Math. Soc.}, 67(1):102--104, 1961.

\bibitem{TaoBO}
T.~Tao.
\newblock {Global well-posedness of the Benjamin-Ono equation in
  $H^1(\mathbb{R})$}.
\newblock {\em Journal of Hyperbolic Differential Equations}, 01(01):27--49,
  2004.

\bibitem{Yafaev}
D.~Yafaev.
\newblock {Sharp Constants in the Hardy-Rellich Inequalities}.
\newblock {\em Journal of Functional Analysis}, 168(1):121--144, 1999.

\bibitem{ZhangKP}
Y.~Zhang.
\newblock {Local well-posedness of KP-I initial value problem on torus in the
  Besov space}.
\newblock {\em Communications in Partial Differential Equations},
  41(2):256--281, 2016.

\end{thebibliography}

\end{document}